\documentclass[11pt]{amsart}

\usepackage{etoolbox}

\makeatletter
\let\old@tocline\@tocline
\let\section@tocline\@tocline
\newcommand{\subsection@dotsep}{4.5}
\newcommand{\subsubsection@dotsep}{4.5}
\patchcmd{\@tocline}
{\hfil}
{\nobreak
	\leaders\hbox{$\m@th
		\mkern \subsection@dotsep mu\hbox{.}\mkern \subsection@dotsep mu$}\hfill
	\nobreak}{}{}
\let\subsection@tocline\@tocline
\let\@tocline\old@tocline

\patchcmd{\@tocline}
{\hfil}
{\nobreak
	\leaders\hbox{$\m@th
		\mkern \subsubsection@dotsep mu\hbox{.}\mkern \subsubsection@dotsep mu$}\hfill
	\nobreak}{}{}
\let\subsubsection@tocline\@tocline
\let\@tocline\old@tocline

\let\old@l@subsection\l@subsection
\let\old@l@subsubsection\l@subsubsection

\def\@tocwriteb#1#2#3{%
	\begingroup
	\@xp\def\csname #2@tocline\endcsname##1##2##3##4##5##6{%
		\ifnum##1>\c@tocdepth
		\else \sbox\z@{##5\let\indentlabel\@tochangmeasure##6}\fi}%
	\csname l@#2\endcsname{#1{\csname#2name\endcsname}{\@secnumber}{}}%
	\endgroup
	\addcontentsline{toc}{#2}%
	{\protect#1{\csname#2name\endcsname}{\@secnumber}{#3}}}%

\newlength{\@tocsectionindent}
\newlength{\@tocsubsectionindent}
\newlength{\@tocsubsubsectionindent}
\newlength{\@tocsectionnumwidth}
\newlength{\@tocsubsectionnumwidth}
\newlength{\@tocsubsubsectionnumwidth}
\newcommand{\settocsectionnumwidth}[1]{\setlength{\@tocsectionnumwidth}{#1}}
\newcommand{\settocsubsectionnumwidth}[1]{\setlength{\@tocsubsectionnumwidth}{#1}}
\newcommand{\settocsubsubsectionnumwidth}[1]{\setlength{\@tocsubsubsectionnumwidth}{#1}}
\newcommand{\settocsectionindent}[1]{\setlength{\@tocsectionindent}{#1}}
\newcommand{\settocsubsectionindent}[1]{\setlength{\@tocsubsectionindent}{#1}}
\newcommand{\settocsubsubsectionindent}[1]{\setlength{\@tocsubsubsectionindent}{#1}}

\renewcommand{\l@section}{\section@tocline{1}{\@tocsectionvskip}{\@tocsectionindent}{}{\@tocsectionformat}}%
\renewcommand{\l@subsection}{\subsection@tocline{1}{\@tocsubsectionvskip}{\@tocsubsectionindent}{}{\@tocsubsectionformat}}%
\renewcommand{\l@subsubsection}{\subsubsection@tocline{1}{\@tocsubsubsectionvskip}{\@tocsubsubsectionindent}{}{\@tocsubsubsectionformat}}%
\newcommand{\@tocsectionformat}{}
\newcommand{\@tocsubsectionformat}{}
\newcommand{\@tocsubsubsectionformat}{}
\expandafter\def\csname toc@1format\endcsname{\@tocsectionformat}
\expandafter\def\csname toc@2format\endcsname{\@tocsubsectionformat}
\expandafter\def\csname toc@3format\endcsname{\@tocsubsubsectionformat}
\newcommand{\settocsectionformat}[1]{\renewcommand{\@tocsectionformat}{#1}}
\newcommand{\settocsubsectionformat}[1]{\renewcommand{\@tocsubsectionformat}{#1}}
\newcommand{\settocsubsubsectionformat}[1]{\renewcommand{\@tocsubsubsectionformat}{#1}}
\newlength{\@tocsectionvskip}
\newcommand{\settocsectionvskip}[1]{\setlength{\@tocsectionvskip}{#1}}
\newlength{\@tocsubsectionvskip}
\newcommand{\settocsubsectionvskip}[1]{\setlength{\@tocsubsectionvskip}{#1}}
\newlength{\@tocsubsubsectionvskip}
\newcommand{\settocsubsubsectionvskip}[1]{\setlength{\@tocsubsubsectionvskip}{#1}}

\patchcmd{\tocsection}{\indentlabel}{\makebox[\@tocsectionnumwidth][l]}{}{}
\patchcmd{\tocsubsection}{\indentlabel}{\makebox[\@tocsubsectionnumwidth][l]}{}{}
\patchcmd{\tocsubsubsection}{\indentlabel}{\makebox[\@tocsubsubsectionnumwidth][l]}{}{}

\newcommand{\@sectypepnumformat}{}
\renewcommand{\contentsline}[1]{%
	\expandafter\let\expandafter\@sectypepnumformat\csname @toc#1pnumformat\endcsname%
	\csname l@#1\endcsname}
\newcommand{\@tocsectionpnumformat}{}
\newcommand{\@tocsubsectionpnumformat}{}
\newcommand{\@tocsubsubsectionpnumformat}{}
\newcommand{\setsectionpnumformat}[1]{\renewcommand{\@tocsectionpnumformat}{#1}}
\newcommand{\setsubsectionpnumformat}[1]{\renewcommand{\@tocsubsectionpnumformat}{#1}}
\newcommand{\setsubsubsectionpnumformat}[1]{\renewcommand{\@tocsubsubsectionpnumformat}{#1}}
\renewcommand{\@tocpagenum}[1]{%
	\hfill {\mdseries\@sectypepnumformat #1}}

\let\oldappendix\appendix
\renewcommand{\appendix}{%
	\leavevmode\oldappendix%
	\addtocontents{toc}{%
		\protect\settowidth{\protect\@tocsectionnumwidth}{\protect\@tocsectionformat\sectionname\space}%
		\protect\addtolength{\protect\@tocsectionnumwidth}{2em}}%
}
\makeatother

\makeatletter
\settocsectionnumwidth{2em}
\settocsubsectionnumwidth{2.5em}
\settocsubsubsectionnumwidth{3em}
\settocsectionindent{1pc}%
\settocsubsectionindent{\dimexpr\@tocsectionindent+\@tocsectionnumwidth}%
\settocsubsubsectionindent{\dimexpr\@tocsubsectionindent+\@tocsubsectionnumwidth}%
\makeatother

\settocsectionvskip{10pt}
\settocsubsectionvskip{0pt}
\settocsubsubsectionvskip{0pt}

\settocsectionformat{\bfseries}
\settocsubsectionformat{\mdseries}
\settocsubsubsectionformat{\mdseries}
\setsectionpnumformat{\bfseries}
\setsubsectionpnumformat{\mdseries}
\setsubsubsectionpnumformat{\mdseries}

\let\oldtableofcontents\tableofcontents
\renewcommand{\tableofcontents}{%
	\vspace*{-\linespacing}
	\oldtableofcontents}

\usepackage[dvipsnames]{xcolor}
\usepackage[hyperfootnotes=false]{hyperref}
\usepackage{enumerate}
\usepackage{tikz}
\usepackage{tikz-cd}
\usepackage{hyperref}
\usepackage{amsthm,amsfonts,amsmath,amscd,amssymb}
\usepackage{xcolor,float}
\usepackage[all]{xy}
\usepackage{mathrsfs}
\usepackage{graphics,graphicx}
\usepackage{caption}

\let\oldmarginpar\marginpar
\renewcommand\marginpar[1]{\-\oldmarginpar[\raggedleft\footnotesize #1]%
	{\raggedright\footnotesize #1}}
\theoremstyle{plain}
\newtheorem{thm}{Theorem}[section]
\newtheorem{lemma}[thm]{Lemma}

\newtheorem{prop}[thm]{Proposition}

\newtheorem{cor}[thm]{Corollary}

\theoremstyle{definition}

\newtheorem{definition}[thm]{Definition}

\newtheorem{ex}[thm]{Example}
\newtheorem{remark}[thm]{Remark}

\numberwithin{equation}{section}

\renewcommand{\S}{\mathbb{S}}

\newcommand{\D}{\mathbb{D}}
\newcommand{\N}{\mathbb{N}}
\newcommand{\Z}{\mathbb{Z}}

\newcommand{\R}{\mathbb{R}}

\newcommand{\SA}{\mathcal{A}}

\newcommand{\SH}{\mathcal{H}}

\newcommand{\SL}{\mathcal{L}}

\newcommand{\La}{\Lambda}
\newcommand{\la}{\lambda}

\renewcommand{\a}{\alpha}

\renewcommand{\d}{\delta}
\newcommand{\e}{\varepsilon}
\newcommand{\dd}{\partial}

\newcommand{\sse}{\subset}
\newcommand{\lr}{\rightarrow}

\newcommand{\Br}{\operatorname{Br}}

\newcommand{\sgn}{\operatorname{sgn}}
\newcommand{\Aut}{\operatorname{Aut}}

\newcommand{\GL}{\operatorname{GL}}

\newcommand{\wt}{\widetilde}

\newcommand{\st}{\text{st}}
\newcommand{\std}{\text{st}}

\newcommand{\comb}{{\operatorname{comb}}}
\newcommand{\Lie}{{\operatorname{Lie}}}
\newcommand{\NC}{{\operatorname{NC}}}

\def\Op{{\mathcal O}{\it p}}
\newcounter{daggerfootnote}

\newcommand{\bD}{\mathbb{D}}

\newcommand{\cL}{\mathcal{L}}
\newcommand{\cM}{\mathcal{M}}

\newcommand{\cS}{\mathcal{S}}

\begin{document}
	
	\title{Braid loops with infinite monodromy on the Legendrian contact DGA}
	
	\author{Roger Casals}
	\address{University of California Davis, Dept. of Mathematics, Shields Avenue, Davis, CA 95616, USA}
	\email{casals@math.ucdavis.edu}
	
	\author{Lenhard Ng}
	\address{Duke University, Department of Mathematics, Durham, NC 27708, USA}
	\email{ng@math.duke.edu}
	
\maketitle

\begin{abstract} We present the first examples of elements in the fundamental group of the space of Legendrian links in $(\S^3,\xi_\st)$ whose action on the Legendrian contact DGA is of infinite order. 
This allows us to construct the first families of Legendrian links that can be shown to admit infinitely many Lagrangian fillings by Floer-theoretic techniques. These new families include the first known Legendrian links with infinitely many fillings that are not rainbow closures of positive braids, and the smallest Legendrian link with infinitely many fillings known to date. We discuss how to use our examples to construct other links with infinitely many fillings, and in particular give the first Floer-theoretic proof that Legendrian $(n,m)$ torus links have infinitely many Lagrangian fillings if $n\geq3,m\geq6$ or $(n,m)=(4,4),(4,5)$. In addition, for any given higher genus, we construct a Weinstein 4-manifold homotopic to the 2-sphere whose wrapped Fukaya category can distinguish infinitely many exact closed Lagrangian surfaces of that genus in the same smooth isotopy class, but distinct Hamiltonian isotopy classes. A key technical ingredient behind our results is a new combinatorial formula for decomposable cobordism maps between Legendrian contact DGAs with integer (group ring) coefficients.
	
\end{abstract}
\setcounter{tocdepth}{1}
\tableofcontents
\section{Introduction}\label{sec:intro}
In this article, we construct Legendrian loops for several families of Legendrian links in the standard contact $3$-sphere $(\S^3,\xi_\st)$ and show that their monodromy action on their Legendrian contact DGA is of infinite order. These are the first examples of such a Floer-theoretic infinite order, in sharp contrast with the known finite order DGA action of all previously studied loops. We provide several new consequences of these results, including the first known examples of Legendrian links with infinitely many Lagrangian fillings which are {\it not} the rainbow closure of a positive braid\footnote{Previously known methods to build infinite Lagrangian fillings, including the techniques from microlocal sheaf theory, do not apply in this general setting.}, and can be distinguished via Floer theory.\footnote{Note that, before this manuscript, none of the infinite Lagrangian fillings in \cite{CasalsHonghao} or \cite{CasalsZaslow} was known to be distinguished via Floer theory.} These Lagrangian fillings are all smoothly isotopic, but their Hamiltonian isotopy classes are all distinct. One of these new Legendrian links has $2$ components and, with its Lagrangian fillings being of genus $1$, is arguably the smallest known Legendrian link to date, in terms of genus and components, with infinitely many Lagrangian fillings. In addition, for any given genus $g\geq2$, we construct Weinstein 4-manifolds homotopic to the 2-sphere whose wrapped Fukaya categories can distinguish infinitely many (Hamiltonian isotopy classes of) exact closed Lagrangian surfaces of that genus, all in the same smooth type. Finally, we show how to Floer-theoretically detect the existence of infinitely many Lagrangian fillings for the Legendrian $(n,m)$ torus links of maximal Thurston--Bennequin number (``max-tb''), with $n\geq3,m\geq6$ and $(n,m)=(4,4),(4,5)$, and many other Legendrian links, by using the Legendrian DGA.\footnote{In all cases being considered, the max-tb condition is a necessary condition on the Legendrian links in order to admit an embedded exact Lagrangian filling, see e.g. \cite{Chantraine10}.}

The manuscript also develops technical results on the Legendrian contact DGA, of independent interest, needed for our argument. In particular, we present a combinatorial model for computing DGA morphisms associated to decomposable Lagrangian cobordisms $L$, where the morphisms are enhanced over {\it integer} group ring coefficients. We show that this is isomorphic to the abstract enhancement previously developed by Karlsson, thus proving invariance and allowing us to perform explicit computations over $\Z[H_1(L)]$. This integrally enhanced package is then used to prove the above Floer-theoretical results concerning infinitely many Lagrangian fillings.


\subsection{Context} Legendrian links in contact 3-manifolds \cite{Bennequin83,ArnoldSing} are instrumental in the study of 3-dimensional contact geometry \cite{OzbagciStipsicz04,Geiges08}. The study of their Lagrangian fillings yields non-trivial DGA representations of the Legendrian contact DGA associated to any Legendrian link, which themselves are effective invariants for distinguishing Legendrian representatives in the same smooth type \cite{Chekanov,Ng03,Sivek_Bordered}. In particular, Floer theory has provided far-reaching methods to address questions on Legendrian links; for instance, along the lines of this paper, see \cite{EliashbergPolterovich96,Etnyre03,Kalman,Chantraine10}.

Recently, the first examples of Legendrian links in $(\S^3,\xi_\st)$ which admit {\it infinitely} many Lagrangian fillings in $(\D^4,\la_\st)$ were discovered \cite{CasalsHonghao}. Indeed, \cite[Corollary 1.5]{CasalsHonghao} shows that the max-tb Legendrian $(n,m)$-torus link $\La(n,m)$ admits infinitely many Lagrangian fillings if $n\geq3,m\geq6$ or $(n,m)=(4,4),(4,5)$. The method of proof itself relies on the theory of microlocal sheaves, and it remained unclear whether the existence of infinitely many Lagrangian fillings, even for one Legendrian link, could also be proven via Floer-theoretic methods. It also remained unknown whether (typically smaller) links which were not rainbow closures of positive braids -- from which the current sheaf methods do not apply -- could actually admit infinitely many Lagrangian fillings.

(i) First, we show that the Legendrian DGA detects infinitely many fillings and it does so for new Legendrian links (including links that are not the rainbow closure of a positive braid). In fact, we significantly improve on \cite[Corollary 1.5]{CasalsHonghao} by showing that simpler classes of Legendrian braids already admit infinitely many exact Lagrangian fillings, and doing so Floer-theoretically. For instance, the family of Legendrian braids of affine $D_n$-type depicted in Figure \ref{fig:BraidsIntro} (right) is one such class. This also gives an alternative Floer-theoretical proof that the torus links in \cite[Corollary 1.5]{CasalsHonghao} admit infinitely many Lagrangian fillings.

\begin{center}
	\begin{figure}[h!]
		\centering
		\includegraphics[scale=0.8]{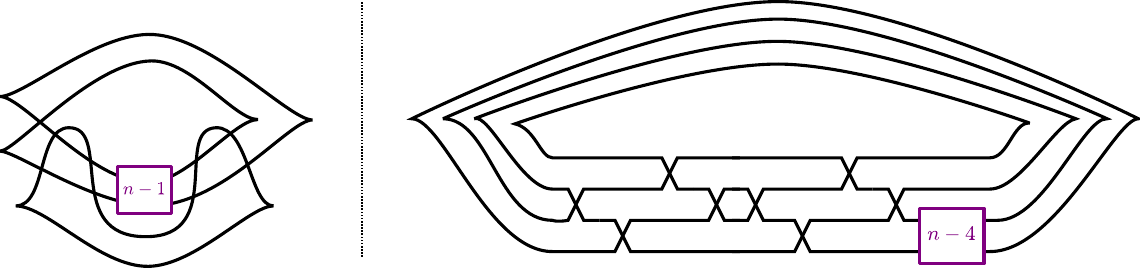}
		\caption{The family of Legendrian links $\La_n\sse(\S^3,\xi_\st)$, $n\geq1$, on the left. The Legendrian links of affine $D_n$-type are depicted on the right, $n\geq4$. All of these have infinitely many fillings. The boxes indicate a series of positive crossings.
		}
		\label{fig:BraidsIntro}
	\end{figure}
\end{center}

Note that, since the appearance of \cite{CasalsHonghao}, the articles \cite{CasalsZaslow,GSW,GSW2} have also continued to develop various cluster and sheaf-theoretic methods that detect infinitely many Lagrangian fillings for a Legendrian link $\La\sse(\S^3,\xi_\st)$. Nevertheless, all these techniques are currently only effective at studying Legendrian links which are positive braids, i.e. when $\La\sse(\S^3,\xi_\st)$ admits a Legendrian front given by the rainbow closure of a positive braid, and do {\it not} apply to several of our smallest links. In contrast, the Floer-theoretic argument we develop {\it also} applies to certain Legendrian links $\La\sse(\S^3,\xi_\st)$ which are {\it not} the rainbow closure of positive braids. For instance, we show that each of the Legendrian links $\La_n$, $n\in\N$, depicted in Figure \ref{fig:BraidsIntro} (left) admits infinitely many Lagrangian fillings. For $n=1$, this yields a Legendrian link $\La_1$ which is not the rainbow closure of a positive braid because it contains a $tb=-3$ stabilized unknot component.

(ii) Second, the existence of infinitely many Lagrangian fillings for our families of Legendrian links $\La\sse(\S^3,\xi_\st)$ is deduced from a stronger result on Legendrian loops, Theorem \ref{thm:main}, as we  explain shortly. In particular, we provide the first examples of Legendrian loops whose induced monodromy action on the Legendrian contact DGA has infinite order. In addition, we present the first example of a Weinstein 4-manifold homotopic to the 2-sphere with infinitely many Hamiltonian isotopy classes of exact Lagrangian surfaces of genus 2 (and no Lagrangian 2-spheres nor exact Lagrangian tori). This is part of the family of Weinstein 4-manifolds in Corollaries \ref{cor:Stein1} and \ref{cor:Stein2}, which construct such Weinstein 4-manifolds for all genera $g\geq2$. Note that, at the level of smooth topology, the concatenation of these Legendrian loops with any decomposable Lagrangian filling does {\it not} change the smooth type of the Lagrangian filling. Thus, we can use these Legendrian loops to produce infinitely many Lagrangian fillings (and surfaces in Weinstein 4-manifolds) which are distinct up to Hamiltonian isotopy, but these surfaces are all {\it smoothly} isotopic.

(iii) Third, at a technical level, we study the lifts of the DGA maps induced by exact Lagrangian cobordisms to $\Z$-coefficients, which is required to argue the infinite order in our argument. This is interesting on its own, as it provides correct signs for Floer theoretical invariants, such as augmentations, and it is a necessary ingredient for the study of cluster structures\footnote{In characteristic different from $2$. In particular, the correct signs are needed for arguing in characteristic $0$, the most studied case in cluster theory.} on augmentation varieties and their holomorphic symplectic structures, as this requires Floer theory in characteristic 0. In particular,  these results from this manuscript are used in the recent article \cite{CGGS} to construct a holomorphic symplectic structure on the augmentation varieties associated to Legendrian positive braids.

\begin{center}
	\begin{figure}[h!]
		\centering
		\includegraphics[scale=0.75]{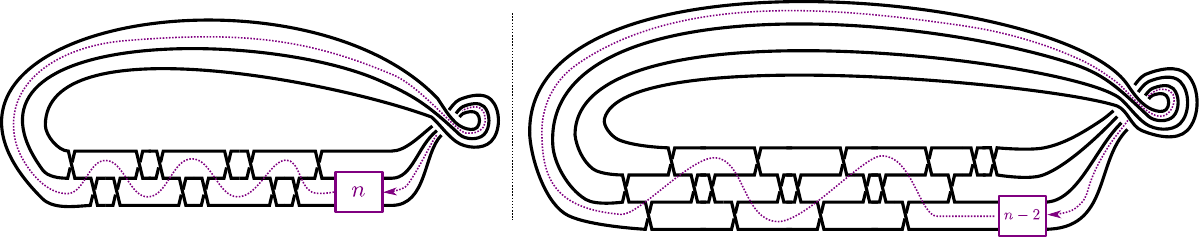}
		\caption{On the left, Lagrangian projection of the Legendrian links $\La_n$, $n\geq1$, and the purple box, which contains $n$ positive crossings. The purple-box Legendrian loop $\vartheta:\S^1\lr\cL(\La_n)$ is depicted by the dashed purple trajectory. On the right, Lagrangian projection of the Legendrian link $\La(\wt D_n)$, $n\geq4$, and the purple box, which contains $(n-2)$ positive crossings. The purple-box Legendrian loop $\vartheta:\S^1\lr\cL(\La(\wt D_n))$ is also illustrated by the dashed purple trajectory.}
		\label{fig:BraidsIntroLoops}
	\end{figure}
\end{center}

\subsection{Main Results} Let $\beta$ be a positive braid, representing an element in the $N$-stranded positive braid monoid $\mbox{Br}^+_N$, $N\in\N$. We can associate a Legendrian link $\La(\beta)\sse(\S^3,\xi_\st)$ to $\beta$ such that $\La(\beta)$ is topologically the $(-1)$-framed closure of $\beta$: this is achieved by placing $\beta$ in a standard contact neighborhood of the standard Legendrian unknot in $\S^3$ of Thurston--Bennequin number $tb=-1$. See Figure \ref{fig:Front_Intro} for a depiction of $\La(\beta)$, where throughout this paper we will describe Legendrian links through their front and/or Lagrangian projections (see Section~\ref{sec:prelim} for a review). We now define the Legendrian links that we will study in this paper.

\begin{center}
	\begin{figure}[h!]
		\centering
		\includegraphics[scale=0.75]{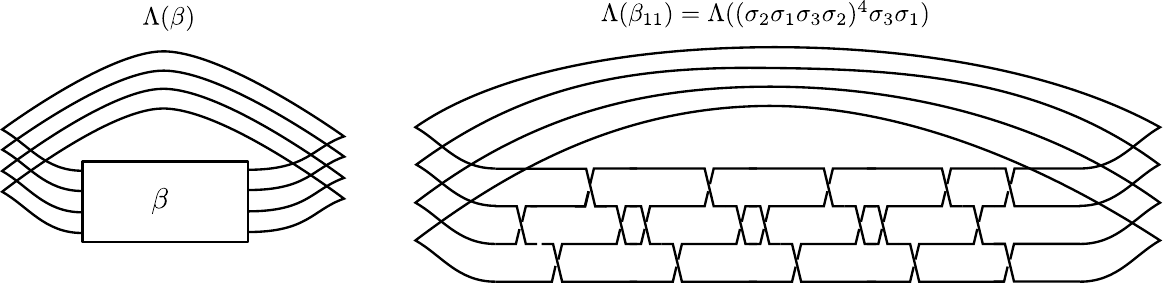}
		\caption{The general front for the Legendrian links $\La(\beta)$, on the left, and the specific example of the Legendrian link $\La(\beta_{11})$, on the right.}
		\label{fig:Front_Intro}
	\end{figure}
\end{center}

By definition, the $\wt D_n$--Legendrian link is the Legendrian $\La(\beta(\wt D_n))\sse(\S^3,\xi_\st)$ associated to
$$\beta(\wt D_n)=(\sigma_2\sigma_1\sigma_3\sigma_2\sigma_2\sigma_3\sigma_1\sigma_2)\sigma_1^{n-4}\Delta^2,\quad n\geq4,$$

where $\Delta=\sigma_1(\sigma_2\sigma_1)(\sigma_3\sigma_2\sigma_1)$ is the $4$-stranded half-twist. Figure \ref{fig:BraidsIntro} (right) shows a Legendrian front projection for $\La(\wt D_n)$, and the terminology will be explained in Section \ref{sec:prelim}.\footnote{The $\wt D_n$-Legendrian should be read as the {\it affine} $D_n$-Legendrian.} 

Similarly, the Legendrian link $\La_n\sse(\S^3,\xi_\st)$ is the Legendrian link $\La_n=\Lambda(\beta_n)$ associated to the braid word

$$\beta_n=(\sigma_2\sigma_1\sigma_1\sigma_2)^3\sigma_1^n,\quad n\geq1.$$

Figure \ref{fig:BraidsIntro} (left) shows a Legendrian front projection for $\La_n$. These are two distinct families of Legendrian links, with the exception of the accidental Legendrian isotopy $\La_2\cong\La(\wt D_5)$. Finally, we will also consider the Legendrian links associated to the following braids:
$$\beta_{ab}=(\sigma_2\sigma_1\sigma_3\sigma_2)^4 \sigma_3^a\sigma_1^b,\quad a,b\in\{1,2\}.$$
Note that $\La(\beta_{22})$ is Legendrian isotopic to $\La(\wt D_4)$. See Figure \ref{fig:Front_Intro} (right) for a drawing of $\La(\beta_{11})$.
Following the above Dynkin-diagram notation, $\La(\beta_{11})$ can also be referred to as the $\La(\wt A_{2,1})$-Legendrian link. From now onwards, we denote by $$\SH=\{\La_n\}_{n\geq1}\cup\{\La(\wt D_m)\}_{m\geq4}\cup \{\La(\beta_{11}),\La(\beta_{12}),\La(\beta_{21})\}$$ the set-theoretic union of the Legendrian links in the $\La_n$ and $\La(\wt D_n)$ families described above and the three Legendrians links $\La(\beta_{11}),\La(\beta_{12}),\La(\beta_{21})$. The Legendrian links in $\SH$ allow us to tackle a wide range of additional Legendrian links, thanks to Corollary \ref{cor:infinitelymanyfillings} below. This includes torus links, as in Corollary \ref{cor:toruslinks}, and the knots discussed in Section \ref{sec:CorRmks}, see Remark \ref{rmk:introlinks} below.

Let $\cL(\La)$ be the space of Legendrian links isotopic to the Legendrian link $\La\sse(\S^3,\xi_\st)$, with base point an arbitrary but fixed Legendrian representative. In Section \ref{sec:prelim}, for each of the links $\La\in\SH$, we will define a certain loop $\vartheta$ of Legendrians based at $\La$: that is, a continuous map $\vartheta:(\S^1,\text{pt})\lr(\cL(\La),\La)$. For instance, for the Legendrians in Figure~\ref{fig:BraidsIntroLoops}, the loop arises from moving the purple box around the link in the manner depicted. 
We will refer to this Legendrian loop $\vartheta$ as the purple-box Legendrian loop.

The graph of the Legendrian loop $\vartheta$ produces an exact Lagrangian concordance $L_\vartheta$ in the symplectization of $(\S^3,\xi_\st)$ from $\La$ to itself. Given any filling $L\sse(\D^4,\la_\st)$ of $\Lambda$, which we can view as an exact Lagrangian cobordism from the empty link to $\Lambda$, we can concatenate $L$ with any number of copies of $L_\vartheta$ to produce an infinite family of fillings
\[
L \# L_\vartheta^n, ~ n\in\N,
\]
of $\La$. What we will show is that for $\La \in \SH$, we can choose a filling $L$ of $\La$ such that all of these fillings $L \# L_\vartheta^n$ are distinct.

As discussed earlier, our method of proof involves the Legendrian contact DGA $\SA_\La$ of $\La$,
which is an invariant of the Legendrian isotopy class of $\La\sse(\S^3,\xi_\st)$, up to stable tame DGA isomorphism. The concordance $L_\vartheta$ induces a DGA isomorphism
$$\SA(L_\vartheta):\SA_\La\lr\SA_\La$$
while the filling $L$ induces a DGA morphism (``augmentation'')
$$\varepsilon_L:\SA_\La\lr(\Z[H_1(L)],0),$$
where $(\Z[H_1(L)],0)$ is the DGA with trivial differential, concentrated in degree $0$. Functoriality then implies that the filling $L \# L_\vartheta^n$ induces the augmentation $\varepsilon_L \circ \SA(L_\vartheta)^n$. To distinguish the fillings $L \# L_\vartheta^n$ from each other, we will distinguish the augmentations $\varepsilon_L \circ \SA(L_\vartheta)^n$, even allowing for different choices of local systems on the fillings. 

To be precise, we say that the $\vartheta$-orbit of the augmentation $\varepsilon_L$ is \textit{entire} if for any $k,l\in\N$ distinct, there is no automorphism $\varphi\in\Aut(\Z[H_1(L)])$ such that
$$\varphi(\varepsilon_L\circ\SA(L_\vartheta)^k)=\varepsilon_L\circ\SA(L_\vartheta)^l :\thinspace \SA_\La \to \Z[H_1(L)].$$
The first result in our article is the following:

\begin{thm}\label{thm:main} Let $\La\in\SH$ be a Legendrian link. The purple-box Legendrian loop $\vartheta:\S^1\lr\cL(\La)$ induces a DGA map $\SA(L_\vartheta) :\thinspace \SA(\Lambda) \to \SA(\Lambda)$ of infinite order. In fact, there exists an exact Lagrangian filling $L\sse(\D^4,\la_\st)$ such that the $\vartheta$-orbit of the corresponding augmentation $\varepsilon_L:\SA_\La\lr\Z[H_1(L)]$ is entire.
\end{thm}

To our knowledge, Theorem \ref{thm:main} presents the first Legendrian loops which induce an {\it infinite} order action on the augmentations of a Legendrian contact DGA $\SA_\La$. Our work is a spiritual successor to the work of T. K\'alm\'an \cite{Kalman}, who studied Legendrian loops for positive torus links $\La(n,m)$ whose induced action on $\SA(\La(n,m))$ has finite order $(n+m)$.

Theorem \ref{thm:main} implies the following:

\begin{cor}\label{cor:cyclicsubgroup}
	Let $\La\in\SH$. Then the purple-box Legendrian loop $\vartheta$ generates an infinite subgroup $\Z\langle\vartheta\rangle\sse\pi_1(\cL(\La))$. In addition, the graph of the Legendrian loop $\vartheta$ produces a Lagrangian self-concordance of $\La$ which has infinite order as an element of the Lagrangian concordance monoid based at $\La$.
\end{cor}

Let us now focus on Lagrangian fillings. Theorem \ref{thm:main} implies that each of the Legendrian links $\La\in\SH$ admits infinitely many Lagrangian fillings, up to Hamiltonian isotopy. 
More precisely, there exists a countably infinite collection $\{L_i\}_{i\in\N}$ of oriented embedded exact Lagrangian fillings $L_i\subset(\D^4,\la_\st)$ of the Legendrian link $\La$ in the boundary $\S^3 = \dd\D^4$ such that all $L_i$ are smoothly isotopic for $i\in\N$, relative to a neighborhood of the boundary $\La$, but none of the $L_i$ are Hamiltonian isotopic to each other; that is, if $i\neq j$, there exists no compactly supported Hamiltonian isotopy $\{\varphi_t\}\in \mbox{Ham}^c(\D^4,\la_\st)$, $\varphi_0=\mbox{Id}$, such that $\varphi_1(L_i)=L_j$.

We note that among the Legendrian links in $\SH$, four links---$\Lambda_1$, $\Lambda(\beta_{11})$, $\La(\beta_{12})$, and $\La(\beta_{21})$---have a component which is a stabilized unknot with Thurston--Bennequin number $-3$. (In fact $\La(\beta_{11})$ has two such components.) It follows that none of these four links is the rainbow closure of a positive braid. We emphasize that the methods developed in \cite{CasalsHonghao,CasalsZaslow,GSW,GSW2} for the detecting of infinitely many Lagrangian fillings only apply to rainbow closures of positive braids, and thus our Floer-theoretic techniques provide new results that we currently do not know how to address through cluster algebras \cite{GSW,GSW2} or the study of microlocal sheaves \cite{CasalsHonghao,CasalsZaslow}.\footnote{By Corollary~\ref{cor:infinitelymanyfillings}, we can in fact construct an infinite family of links with infinitely many fillings that are not the rainbow closure of a positive braid: $\La((\sigma_2\sigma_1\sigma_3\sigma_2)^4 \sigma_3^a\sigma_1)$ for $n\geq\N$.}

We can use Legendrian links with infinitely many fillings to produce other Legendrian links with infinitely many fillings. Roughly speaking, if there is an exact Lagrangian cobordism from $\Lambda_-$ to $\Lambda_+$ and $\Lambda_-$ has infinitely many fillings, then $\Lambda_+$ does as well. (We only prove this statement subject to some important hypotheses; see Proposition~\ref{prop:aug-infinite-condition} for the precise result.) In particular, we have the following consequence of Theorem \ref{thm:main}.

\begin{cor}[see Proposition~\ref{prop:aug-infinite-condition}]
\label{cor:infinitelymanyfillings}
	Let $\La_0,\La\sse(\S^3,\xi_\st)$ be Legendrian links with $\La_0$ in the list $\SH$, and suppose that there is a Lagrangian cobordism from $\La_0$ to $\La$ consisting of a sequence of saddle moves at contractible Reeb chords of degree $0$. Then the Legendrian link $\La$ admits infinitely many exact Lagrangian fillings, distinct up to Hamiltonian isotopy.
	\end{cor}

As a special case, since there are such cobordisms to the max-tb Legendrian $(n,m)$ torus links $\La(n,m)$ from $\La_1$ for $n=3,m\geq 6$, and from $\La(\wt D_4)$ for $n,m\geq 4$, we recover the following result of \cite{CasalsHonghao}.

\begin{cor}[\cite{CasalsHonghao}]\label{cor:toruslinks} The Legendrian torus links $\La(n,m)$ each admit infinitely many exact Lagrangian fillings if $n\geq3,m\geq6$ or $(n,m)=(4,4),(4,5)$.
\end{cor}

\begin{remark}\label{rmk:introlinks}
As we will discuss in Section~\ref{ssec:aug-infinite}, among the universe of Legendrian links with infinitely many fillings, a sensible notion of ``simplicity'' is given by the Thurston--Bennequin number, or equivalently the sum $2g+m$, where $g$ is the genus of an exact Lagrangian filling and $m$ is the number of connected components of the link: the smaller $2g+m$ is, the simpler the link is. Among the Legendrian links that we can prove have infinitely many fillings, the simplest by this measure is $\La(\beta_{11})$, which has $(m,g) = (2,1)$ and thus $2g+m=4$.

If we focus on Legendrian {\it knots}, rather than Legendrian {\it links}, Corollary \ref{cor:infinitelymanyfillings} implies that, for instance, the knot types $10_{139}$, $m(10_{145})$, $m(10_{152})$, $10_{154}$, and $m(10_{161})$ all have Legendrian representatives with infinitely many fillings; see Proposition~\ref{prop:knots}. Among these, the simplest is $m(10_{145})$, with $g=2$ and $2g+m=5$. Two of these knots, $10_{139}$ and $m(10_{152})$, are positive braid closures and indeed their Legendrian representatives are rainbow closures of positive braids. We remark that the only other knots with crossing number $\leq 10$ that are positive braid closures are the torus knots $T(2,3)$, $T(2,5)$, $T(2,7)$, $T(3,4)$, $T(2,9)$, and $T(3,5)$; it is conjectured that the (max-tb) Legendrian representatives of each of these knots has finitely many fillings \cite[Conjecture 5.1]{CasalsLagSkel}.\hfill$\Box$
\end{remark}

The above results on Lagrangian fillings also have consequences in the study of Stein surfaces. For each $g\in\N$ and $g\geq6$, the article \cite{CasalsHonghao} gave the first examples of Stein surfaces homotopic to the $2$-sphere $\S^2$ with infinitely many Hamiltonian isotopy classes of embedded exact Lagrangian surfaces of genus $g$ (and none of genus less than $g$). The lower bound was recently improved to $g\geq4$ in \cite{GSW2}. In the present work, we can further improve this bound:

\begin{cor}\label{cor:Stein1}
	Let $g\in\N$ and $g\geq2$. Then, there exists a Stein surface $W$ homotopic to the $2$-sphere $\S^2$ which admits infinitely many Hamiltonian isotopy classes of embedded exact Lagrangian surfaces of genus $g$. In addition, $W$ contains no embedded exact Lagrangian surfaces of genus $h$, $h\leq g-1$.
\end{cor}

In Corollary \ref{cor:Stein1}, the Stein surface $W$ for $g=2$ can be constructed by attaching a Weinstein 2-handle to the standard symplectic 4-ball $(\D^4,\la_\st)$ along a max-tb Legendrian representative of the smooth knot $m(10_{145})$. The results we prove also allow us achieve $g=1$ if we allow ourselves a bouquet of just two $2$-spheres as the given homotopy type, instead of the $2$-sphere $\S^2$:

\begin{cor}\label{cor:Stein2}
	The Stein surface $W$ obtained by attaching two Weinstein 2-handles along $\La(\beta_{11})\sse(\dd\D^4,\xi_\st)$, one per connected component, contains infinitely many Hamiltonian isotopy classes of embedded exact Lagrangian tori.
\end{cor}

Corollaries \ref{cor:Stein1} and \ref{cor:Stein2} are proven in Section \ref{sec:CorRmks}. It remains an outstanding problem to construct a Legendrian knot with infinitely many distinct embedded Lagrangian 2-disk fillings (pairwise smoothly isotopic), or show no such knot exists.\footnote{The case of a link with infinitely many planar Lagrangian fillings (pairwise smoothly isotopic) might already be an interesting start. In terms of Stein surfaces, the analogue of the knot case would be to construct a Stein surface homotopic to $\S^2$ with infinitely many Hamiltonian isotopy classes of pairwise smoothly isotopic Lagrangian 2-spheres.}\\

{\bf Organization.} Here is an outline of the rest of the paper. In Section~\ref{sec:prelim}, we review some necessary background and formally describe the Legendrian links discussed in this introduction. 
The Floer-theoretical core of the article is developed in Sections \ref{sec:cobordism}, \ref{sec:elementary}, and \ref{sec:DGA}. In particular, Sections \ref{sec:cobordism} and \ref{sec:elementary}, jointly with Appendix \ref{sec:saddle-map}, develop a new combinatorial model for the maps between Legendrian contact DGAs with {\it integral coefficients} associated to a decomposable exact Lagrangian cobordism. We believe these results are of independent interest for 3-dimensional contact topology and Floer theory. We then apply these maps in Sections \ref{sec:Monodromy} to prove Theorem~\ref{thm:main}, and prove a number of corollaries and other ancillary results in Section~\ref{sec:CorRmks}.\\


{\bf Acknowledgements.} 
We thank Tobias Ekholm, Honghao Gao, Eugene Gorsky, Linhui Shen, and Daping Weng for illuminating conversations.
R.~Casals is supported by the NSF grant DMS-1841913, the NSF CAREER grant DMS-1942363 and the Alfred P. Sloan Foundation. L.~Ng is partially supported by the NSF grants DMS-1707652 and DMS-2003404.\hfill$\Box$\\


\section{Legendrian Links and $\vartheta$-loops}\label{sec:prelim}

In this section we describe the classes of Legendrian links $\La\sse(\R^3,\xi_\st)$ and Legendrian loops that we study in this article. We begin in Section~\ref{ssec:background} with a review of Legendrian links and exact Lagrangian cobordisms, and then proceed in Sections~\ref{ssec:LegendrianLinks} and~\ref{ssec:satellite} to describe the particular links of interest to us, which include the links in $\SH$ presented in the introduction. We conclude in Section~\ref{ssec:PurpleBox} by describing the purple-box Legendrian loops that are a key ingredient in our constructions.

\subsection{Legendrian links, exact Lagrangian cobordisms, and fillings}
\label{ssec:background}

Here we briefly review the basic geometric terminology that we will need for this paper. There is now an extensive literature on exact Lagrangian cobordisms, including the papers cited in the introduction, to which we refer the reader for further details; specifically, the paper \cite{EHK} has a full exposition of the setting we will use here.

Rather than work with the contact manifold $(\S^3,\xi_\st)$ directly, it is convenient to remove a point and work in the contact manifold $(\R^3,\xi_\st)$, where $\xi_\st$ is the contact structure given by the kernel of the standard contact $1$-form $\alpha_\st:= dz-y\,dx$ on $\R^3$, endowed with Cartesian coordinates $(x,y,z)\in\R^3$. By definition, a link $\Lambda \subset (\R^3,\xi_\st)$ is \textit{Legendrian} if it is everywhere tangent to $\xi_\st$, or equivalently if $\alpha_\st|_\Lambda = 0$; all Legendrian links in this paper are oriented. 

As is customary, we will describe Legendrian links in $(\R^3,\xi_\st)$ by their front and Lagrangian projections. These are the images of the link under the projections $\Pi_{xz},\Pi_{xy} :\thinspace \R^3\to\R^2$ to the $xz$ and $xy$ planes, respectively. Given a Legendrian link $\Lambda$, the Reeb chords of $\Lambda$ are integral curves of the Reeb vector field $\dd_z$ with endpoints on $\Lambda$; these correspond to the crossings of the Lagrangian projection $\Pi_{xy}(\Lambda)$. One numerical invariant associated to a Legendrian link $\Lambda$ is the \textit{Thurston--Bennequin number} $tb(\Lambda)$, which is the number of crossings of $\Pi_{xy}(\Lambda)$ counted with sign. 

\begin{ex}
\label{ex:unknot}
The simplest Legendrian knot in $\R^3$ is the standard Legendrian unknot with $tb=-1$, which we will denote by $U$. The front projection $\Pi_{xz}(U)$ is a ``flying saucer'' with two cusps, while the Lagrangian projection $\Pi_{xy}(U)$ is a ``figure eight'' diagram with a single crossing; see the top left of Figure~\ref{fig:elementary}.\hfill$\Box$
\end{ex}

The \textit{symplectization} of $\R^3$ is the $4$-manifold $\R^4 = \R_t \times \R^3$ equipped with the exact symplectic form $\omega_\st = d\la_\st$ with $\la_\st = e^t\alpha_\st$. Note that this symplectic manifold is symplectomorphic to $(\R^4,d\la_0)$, where $\la_0:=\frac{1}{2}(x_1dy_1-y_1dx_1+x_2dy_2-y_2dx_2)$, $(x_1,y_1,x_2,y_2)\in\R^4$, is the radial Liouville form in $\R^4$. Given that $(\R^4,d\la_\st)$ is symplectomorphic to the Liouville completion of the standard symplectic Darboux ball $(\D^4,d\la_0)$, we will also write $\la_\st$ for $\la_0$ and denote by $(\D^4,\la_\st)$ the unique exact symplectic filling of $(\S^3,\xi_\st)$, with a radial primitive Liouville form.

We will be interested in \textit{Lagrangian submanifolds} of $(\R^4,d\la_\st)$, which are surfaces $L \subset \R^4$ such that $\omega_\st|_L = 0$. One class of Lagrangian submanifolds is given by cylinders over Legendrians: if $\Lambda \subset \R^3$ is Legendrian, then $\R\times\Lambda \subset \R^4$ is Lagrangian. 

More generally, suppose that $\Lambda_+,\Lambda_-$ are Legendrian links in $\R^3$. A \textit{Lagrangian cobordism from $\Lambda_-$ to $\Lambda_+$} is a Lagrangian $L \subset \R^4$ such that, for some $T>0$,
$$L \cap ((-\infty,-T)\times\R^3) = (-\infty,-T)\times\Lambda_-\quad\mbox{ and }\quad L\cap ((T,\infty)\times\R^3) = (T,\infty)\times\Lambda_+.$$ The Lagrangian cobordism $L$ is \textit{exact} if there is a function $f :\thinspace L \to \R$ such that $\la_\st|_L = df$ and $f$ is constant on each of the ends $(-\infty,-T)\times\Lambda_-$ and $(T,\infty)\times\Lambda_+$ separately. All Lagrangian cobordisms considered in this paper will be oriented, embedded, and exact. In the special case where the negative end is empty, an exact Lagrangian cobordism from $\emptyset$ to $\Lambda$ is called a \textit{filling} of $\Lambda$. See Figure~\ref{fig:cobordism}.

\begin{center}
	\begin{figure}[h!]
		\centering
				\includegraphics[scale=0.7]{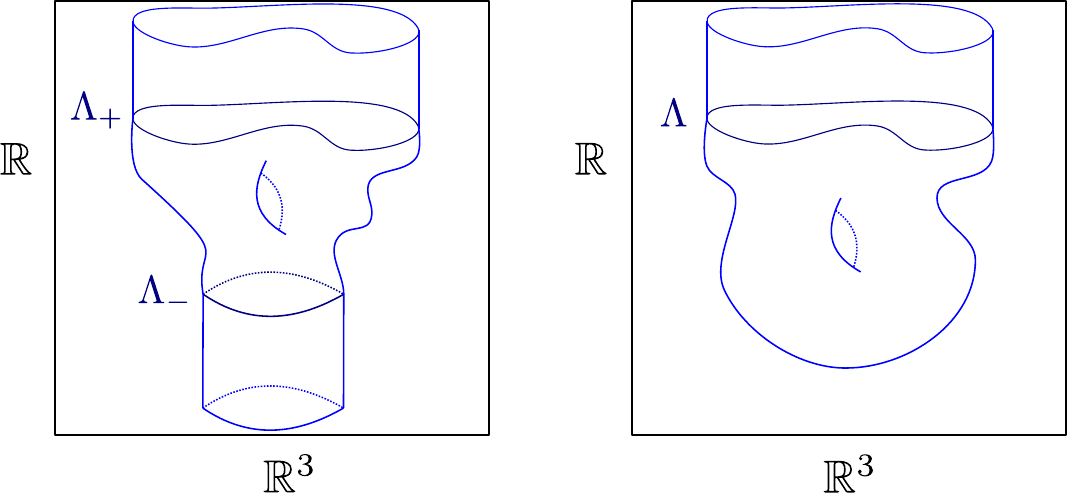}
		\caption{
A Lagrangian cobordism from $\Lambda_-$ to $\Lambda_+$ (left) and a filling of $\Lambda$ (right).
}
		\label{fig:cobordism}
	\end{figure}
\end{center}

We will be interested in Lagrangian cobordisms and fillings up to \textit{exact Lagrangian isotopy}, which is an isotopy through exact Lagrangian cobordisms that fixes the two cylindrical ends (or the positive cylindrical end, in the case of fillings). In the setting of $\R^4$, this is the same as a Hamiltonian isotopy, see e.g.~\cite[Section 3.6]{Oh_BookVol1}, which is an isotopy through Hamiltonian diffeomorphisms fixing the two ends $(-\infty,-T)\times\R^3$ and $(T,\infty) \times \R^3$.

\begin{remark}
Associated to a Legendrian link in $\R^3$ or a Lagrangian surface in $\R^4$ is its \textit{Maslov number}, which takes values in $\Z$. For a Lagrangian surface $L$, this is the greatest common divisor of the Maslov numbers of all closed loops in $L$, where the Maslov number of a loop in $L$ is understood to be the Maslov number of the corresponding loop in the Lagrangian Grassmannian of $\R^4$. For a Legendrian link $\Lambda$, the Maslov number is the Maslov number of the surface $\R\times\Lambda$. All Legendrians and Lagrangians that we consider in this paper will have Maslov number $0$.\hfill$\Box$
\end{remark}

We will construct exact Lagrangian cobordisms out of key building blocks called \textit{elementary cobordisms}, due to \cite{EHK}. There are three types of elementary cobordisms between Legendrian links, which we describe in turn. 

\textbf{(i) Isotopy cobordisms.}
If $\Lambda_-$ and $\Lambda_+$ are Legendrian links that are related by a Legendrian isotopy $\Lambda_t$, then the trace of this isotopy (the union of $\{t\}\times\Lambda_t$ over all $t$) can be perturbed to an exact Lagrangian cobordism from $\Lambda_-$ to $\Lambda_+$, which we will call the \textit{isotopy cobordism} associated to this isotopy. The isotopy represents a path in the space of Legendrian links from $\Lambda_-$ to $\Lambda_+$, and homotopic paths lead to isotopy cobordisms that are exact Lagrangian isotopic.

\textbf{(ii) Minimum cobordisms.}
Let $U$ denote a standard Legendrian unknot as in Example~\ref{ex:unknot}. By \cite{EliashbergPolterovich96}, $U$ has a filling by a Lagrangian $2$-disk, which is necessarily exact, and this filling is unique up to exact Lagrangian isotopy. Thus if $\Lambda_-$ is any Legendrian link and $\Lambda_+$ is the split union of $\Lambda_-$ and a standard unknot $U$, then there is an exact Lagrangian cobordism from $\Lambda_-$ to $\Lambda_+$ given by the union of the filling of $U$ and the cylinder $\R\times\Lambda_-$. This cobordism is called a \textit{minimum cobordism} and corresponds topologically to the addition of a $0$-handle.

\textbf{(iii) Saddle cobordisms.}
Let $\Lambda_+$ be a Legendrian link. Reeb chords of $\Lambda_+$ correspond to crossings in the Lagrangian projection $\Pi_{xy}(\Lambda_+)$. A Reeb chord is called \textit{contractible} if there is a Legendrian isotopy of $\Lambda_+$ inducing a planar isotopy of $\Pi_{xy}(\Lambda_+)$ and ending in a Legendrian where the height of the Reeb chord is arbitrarily small. Suppose that we have a contractible Reeb chord $a$ of $\Lambda_+$ that corresponds to a positive crossing of $\Pi_{xy}(\Lambda_+)$ (in symplectic terms, the Conley--Zehnder index of $a$ is even). One can modify the diagram $\Pi_{xy}(\Lambda_+)$ by replacing the corresponding crossing by its oriented resolution to produce the Lagrangian projection of another Legendrian link $\Lambda_-$; see Figure~\ref{fig:elementary}. There is then an exact Lagrangian cobordism from $\Lambda_-$ to $\Lambda_+$ called a \textit{saddle cobordism}. This is sometimes called a \textit{pinch move} because of what it looks like in the front projection, and we will also sometimes refer to this as ``resolving'' the Reeb chord; it corresponds topologically to the addition of a $1$-handle.

\begin{center}
	\begin{figure}[h!]
		\centering
		\includegraphics[scale=1.5]{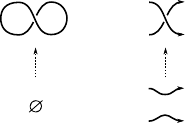}
		\caption{
Two elementary Lagrangian cobordisms, depicted in terms of their $xy$ projections: a minimum cobordism (left) and a saddle cobordism (right). The dotted arrows go from the bottom to the top of the cobordisms. The diagram on the top left is the standard Legendrian unknot $U$.
}
		\label{fig:elementary}
	\end{figure}
\end{center}

We can build more cobordisms out of elementary pieces through the operation of \textit{concatenation}. Suppose that $L_1$ and $L_2$ are exact Lagrangian cobordisms that go from $\Lambda_0$ to $\Lambda_1$ and from $\Lambda_1$ to $\Lambda_2$, respectively. We can remove the top cylinder of $L_1$ and the bottom cylinder of $L_2$ and glue the resulting Lagrangians along their common boundary $\Lambda_1$ to produce a new exact Lagrangian cobordism $L_1 \# L_2$ from $\Lambda_0$ to $\Lambda_2$, the concatenation of $L_1$ and $L_2$. An exact Lagrangian cobordism is \textit{decomposable} if it is the concatenation of some number of elementary cobordisms. 
All of the cobordisms and fillings that we consider in this paper will be decomposable. Now that we have reviewed the basic geometric concepts and terminology, let us delve into the specific objects of interest with a view towards the new contributions of this manuscript.

\subsection{Legendrian links associated to positive braids}\label{ssec:LegendrianLinks} 
We now describe the specific Legendrian links in $(\R^3,\xi_\st)$ that we will consider in this paper. These are a natural family of Legendrian links associated to positive braids, topologically given by the closures of these braids with one full negative twist.

Let $\Br_N$ denote the $N$-strand braid group, $N\in\N$. The standard presentation of $\Br_N$ is given by Artin generators $\sigma_1,\ldots,\sigma_{N-1}$, where $\sigma_i$ corresponds to a single positive crossing between strands $i$ and $i+1$ of the braid, with relations $\sigma_i\sigma_{i+1}\sigma_i = \sigma_{i+1}\sigma_i\sigma_{i+1}$ for all $i$ and $\sigma_i\sigma_j = \sigma_j\sigma_i$ for $|i-j| >1 $. Within $\Br_N$, let $\Br_N^+$ denote the monoid of positive braids; any element $\beta$ of $\Br_N^+$ can be written as a braid word
$$\beta:=\prod_{j=0}^{l(\beta)}\sigma_{i_j},\quad i_j\in[1,N-1],$$
where $l(\beta)$ is the length of $\beta\in\Br_N^+$, equivalently its number of crossings. 

Given a positive braid $\beta$, the \textit{rainbow closure} of $\beta$ is the Legendrian link whose front projection is given by drawing $\beta$ horizontally and joining the left and right ends of $\beta$ by a nested set of non-intersecting arcs with a single left and right cusp; see the top diagrams in Figure~\ref{fig:LagrangianProjectionSatelliteClosure}. Topologically this link is the $0$-framed closure of $\beta$. As mentioned in the introduction, rainbow closures are the subject of several other papers on fillings of Legendrian links, including \cite{CasalsHonghao,CasalsZaslow,GSW,GSW2}. 

We note that not all Legendrian links are (isotopic to) rainbow closures of positive braids. In particular, if $\Lambda$ is the rainbow closure of a positive braid $\beta\in\Br_N^+$, then $\Lambda$ has Thurston--Bennequin number $tb(\Lambda) = l(\beta)-N$. If we write $g(\Lambda)$ for the Seifert genus of the topological link type of $\Lambda$, then there is an obvious Seifert surface for $\Lambda$ whose Euler characteristic is $N-l(\beta)$. It follows that the Bennequin inequality $tb(\Lambda) \leq 2g(\Lambda)-1$ must be sharp in this case, and in particular that $\Lambda$ must maximize Thurston--Bennequin number within its topological type. Thus even if $\Lambda$ represents a topological link that is a positive braid closure, it can only be a rainbow closure if it maximizes $tb$.

We will focus on another Legendrian link associated to a positive braid $\beta$, which we describe next and call the $(-1)$-closure of $\beta$. This is topologically the closure of $\beta$ with a full negative twist, and is arguably more naturally associated to $\beta$ than the rainbow closure, due to its connection to Legendrian satellites as described below. We remark that any rainbow closure of a positive braid is also (Legendrian isotopic to) the $(-1)$-closure of another braid, namely the concatenation of the original braid with a full positive twist.

There is a well-defined (up to isotopy) Legendrian link $\widetilde{\La}(\beta) \sse (J^1\S^1,\xi_\st)$ associated to $\beta$ (cf.\ \cite{Satellite1}).
By definition, the Legendrian $\widetilde{\La}(\beta)\sse(J^1\S^1,\xi_\st)$ is the Legendrian link whose front in $\S^1\times\R$ (image of the projection map $J^1\S^1 = T^*\S^1\times\R\to \S^1\times\R$) consists of the $N$ horizontal strands $\S^1\times\{j\}$, $j=1,\ldots,N$, where (positive) crossings are added left to right according to the braid word $\beta$. Figure \ref{fig:FrontProjection} depicts $\widetilde{\La}(\beta)$ with an explicit example: the $\S^1$-coordinate is horizontal, and the two vertical yellow walls are identified with each other.

\begin{center}
	\begin{figure}[h!]
		\centering
		\includegraphics[scale=0.8]{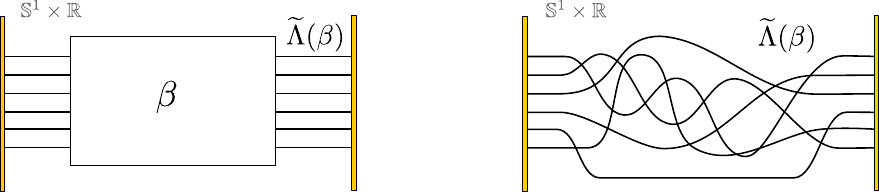}
		\caption{The front for the Legendrian link $\widetilde{\La}(\beta)\sse(J^1\S^1,\xi_\st)$ associated to the braid word $\beta$ (left). Explicit example of $\widetilde{\La}(\beta)$ associated to the braid word $\beta=\sigma_1\sigma_5\sigma_4\sigma_2\sigma_3\sigma_5\sigma_4\sigma_3\sigma_4\sigma_3\sigma_2\sigma_4\sigma_3\sigma^2_2\sigma_4\sigma_5\sigma_3\sigma_4\sigma_2\sigma_5\sigma_1\sigma_2\in\Br_6^+$ (right).}
		\label{fig:FrontProjection}
	\end{figure}
\end{center}

Given a Legendrian link $\widetilde{\La}(\beta)\sse(J^1\S^1,\xi_\st)$, we denote by $\La(\beta)\sse(\R^3,\xi_\st)$ the Legendrian link obtained by satelliting $\widetilde{\La}(\beta)$ along the standard Legendrian unknot $U \subset \R^3$.
To be precise, $(J^1\S^1,\xi_\st)$ is contactomorphic to a standard contact neighborhood $\Op(\La_0)\sse(\R^3,\xi_\st)$ of $\La_0\sse(\R^3,\xi_\st)$, and $\La(\beta)$ is the image of $\widetilde{\La}(\beta)$ under the resulting inclusion $J^1\S^1 \hookrightarrow \R^3$; this is a special case of the Legendrian satellite construction \cite{NgTraynor04}.

Figure \ref{fig:LagrangianProjectionSatelliteClosure} (bottom left) shows the front projection for the Legendrian link $\La(\beta)$. 
The transition from $\Pi_{xz}(\La)$ to $\Pi_{xy}(\La)$ (``resolution'') can be performed as in \cite[Proposition 2.2]{Ng03} and it is also depicted in Figure \ref{fig:LagrangianProjectionSatelliteClosure}, both for rainbow and ($-1$)-closures. In this Lagrangian projection, a combinatorial advantage is that Reeb chords for $\La(\beta)\sse(\R^3,\xi_\st)$ are in bijection with the (positive) crossings of $\Pi_{xy}(\La(\beta))$.

\begin{center}
	\begin{figure}[h]
		\centering
		\includegraphics[scale=0.8]{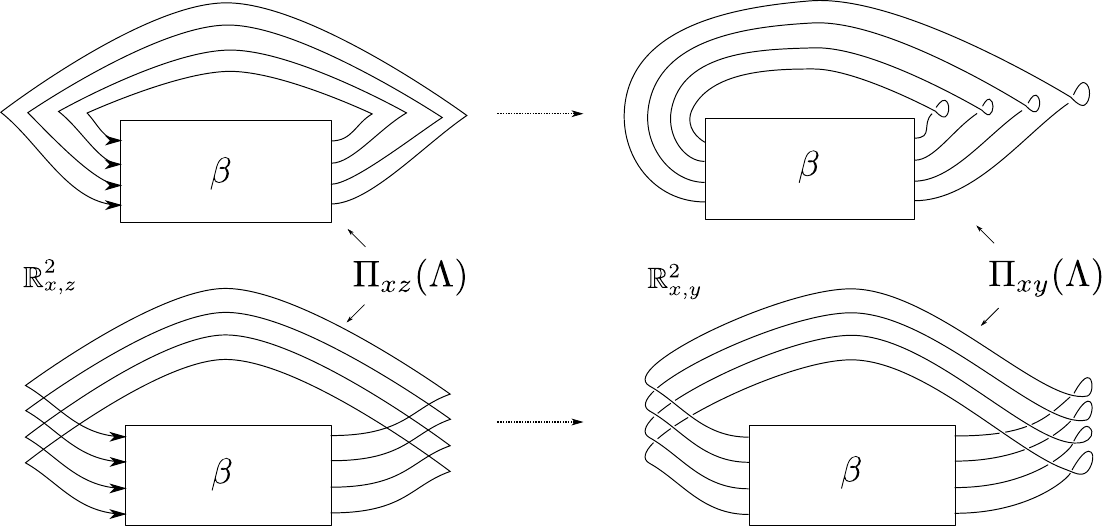}
		\caption{Two different Legendrian links $\La$ associated to a positive braid $\beta$, depicted in their front projection (left) and Lagrangian projection (right). The top row depicts the rainbow closure of $\beta$; the bottom row is $\La(\beta)$. The Lagrangian projections are obtained by resolving the corresponding front projections. The arrows indicate the orientation of the link.}
		\label{fig:LagrangianProjectionSatelliteClosure}
	\end{figure}
\end{center}

We will be interested in the Legendrian contact DGA and associated monodromy of $\Lambda(\beta)$, both of which can be combinatorially described via the $xy$ projection of $\Lambda(\beta)$. Rather than use the $xy$ projection shown in the bottom right of Figure~\ref{fig:LagrangianProjectionSatelliteClosure}, it will significantly simplify our computations to change $\Lambda(\beta)$ by a Legendrian isotopy to have a slightly different $xy$ projection, as we describe next.

\begin{definition}
Let $\beta$ be a positive braid. Consider the link diagram in $\R^2$ given by the blackboard-framed satellite closure of $\beta$ around the figure-eight unknot diagram $\Pi_{xy}(U)$, as depicted in the rightmost diagram of Figure~\ref{fig:LagrangianProjectionSatelliteClosure2}. If this diagram is the Lagrangian projection of a Legendrian link, then we call this Legendrian link the \textit{$(-1)$-closure} of $\beta$.
\label{def:closure}
\hfill$\Box$
\end{definition}

It is apparent that $\Lambda(\beta)$ and the $(-1)$-closure of $\beta$ represent smoothly isotopic links, as they are both the $(-1)$-framed closures of the braid $\beta$. Furthermore, their Lagrangian projections are regularly homotopic: an isotopy between them is indicated in Figure~\ref{fig:LagrangianProjectionSatelliteClosure2}. The first step in this isotopy is just a planar isotopy moving the ${N\choose 2}$ negative crossings to the left of $\beta$ to the top of the diagram. We then use a sequence of Reidemeister II and III moves to obtain the square-grid configuration of crossings shown in the blue box in Figure \ref{fig:LagrangianProjectionSatelliteClosure2} (right). However, the smooth isotopy from the center diagram to the right diagram does not always represent a Legendrian isotopy---and in particular the right diagram does not even necessarily represent a Legendrian link---as we illustrate by an example.

\begin{center}
	\begin{figure}[h!]
		\centering
		\includegraphics[scale=0.75]{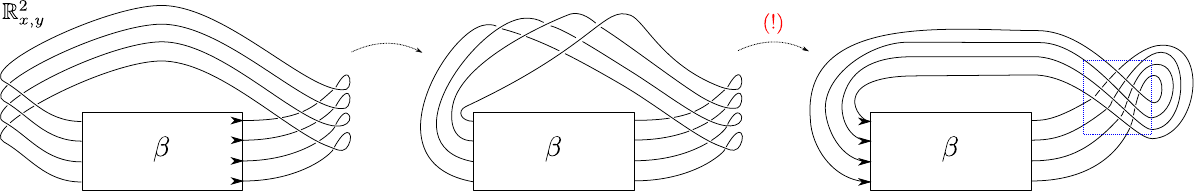}
		\caption{An isotopy from the Lagrangian projection obtained by the resolution of a front projection (left) to the $(-1)$-closure (right).
The red exclamation mark indicates that the smooth isotopy between the center and right diagrams does not necessarily represent a Legendrian isotopy, depending on the choice of $\beta$.}
		\label{fig:LagrangianProjectionSatelliteClosure2}
	\end{figure}
\end{center}

\begin{ex}\label{ex:NestedPigTail}
	Consider the Legendrian link $\Lambda(e)$ where $e$ is the trivial $2$-stranded braid $[\varnothing]\in\Br_2^+$. Following the resolution procedure as in Figure~\ref{fig:LagrangianProjectionSatelliteClosure}, we find that the Lagrangian projection of $\Lambda(e)$ is the exact Lagrangian $L\sse\R^2$ depicted in Figure~\ref{fig:Admissible1}.(i).
	In what follows, we use D. Sauvaget's calculus \cite[Section II.2]{Sauvaget04} for exact Lagrangian projections -- see also \cite[Section 2]{Lin16} for an introduction. Let $A,B,C,P_1,P_2\in\R^+$ be the areas of the bounded regions $\R^2\setminus L$, as shown in Figure \ref{fig:Admissible1}.(i); we may and do assume that we have $B<C$. The two area constraints for this projection read
	$$P_1=A+B,\quad P_2=A+C.$$
	In order to perform a Reidemeister III in the region with area $B$, we first empty the area in that region, leading to Figure \ref{fig:Admissible1}.(ii), and the corresponding exactness constraints are satisfied:
	$$P_1=0+(A+B),\quad P_2=(A+B)-(A+C).$$
	The Reidemeister III move leads to Figure \ref{fig:Admissible1}.(iii) and an additional Reidemeister II move, creating a canceling pair of crossings, to Figure \ref{fig:Admissible1}.(iv). A second Reidemeister III move, which is admissible due to the zero area in its triangular region, yields Figure \ref{fig:Admissible1}.(v). The area constraints are still satisfied, as they coincide with those in Figure \ref{fig:Admissible1}.(ii). These moves concatenate to a Hamiltonian isotopy from Figure \ref{fig:Admissible1}.(i) to Figure \ref{fig:Admissible1}.(v), through exact Lagrangians. Now, we claim that the transition from Figure \ref{fig:LagrangianProjectionSatelliteClosure2} (Center) to Figure \ref{fig:LagrangianProjectionSatelliteClosure2} (right) cannot exist through exact Lagrangians: the resulting Lagrangian -- shown in Figure \ref{fig:Admissible1} -- is not an exact Lagrangian. This can be directly seen by the area constraints:
	$$\alpha=\gamma+\delta+\varepsilon,\quad \alpha+\beta+\varepsilon=\gamma,\quad \alpha,\beta,\gamma,\delta,\varepsilon\in\R^+,$$
	which imply $\delta+\beta+2\varepsilon=0$, contradicting positivity of the areas $\delta,\beta,\varepsilon\in\R^+$. Alternatively, it is rather immediate that the two curves in Figure \ref{fig:Admissible1}.(vi) bound an immersed annulus, with positive area. Hence, the conclusion is that a constraint on $\beta$ needs to be imposed, should we want to work with a Legendrian link through a Lagrangian projection of the form shown in Figure \ref{fig:LagrangianProjectionSatelliteClosure2} (right).\hfill$\Box$
	\begin{center}
		\begin{figure}[h!]
			\centering
			\includegraphics[scale=0.75]{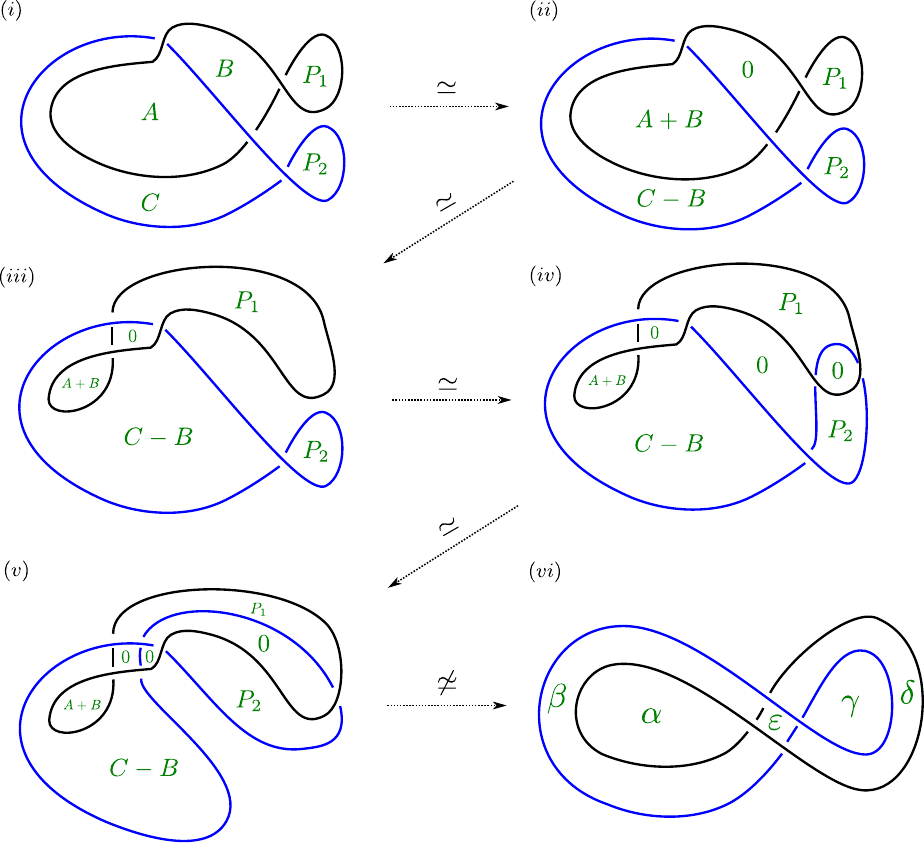}
			\caption{The sequence of Lagrangian projections discussed in Example \ref{ex:NestedPigTail}. The transitions from (i) to (v) are all realized by Hamiltonian isotopies, preserving the exactness of the immersed Lagrangian. Item (vi) displays an example of a Lagrangian which is {\it not} exact. The consequence of (v) not being necessarily Legendrian isotopic to (vi) leads to Definition \ref{def:NestedPigTail}.}
			\label{fig:Admissible1}
		\end{figure}
	\end{center}
\end{ex}

We will want to consider braids where the isotopy in Figure~\ref{fig:LagrangianProjectionSatelliteClosure2} is legal. In the following definition, let $\pi :\thinspace J^1\S^1 \to T^*\S^1$ denote the projection to the first factor, where $J^1\S^1 = T^*\S^1 \times\R_z$ with the standard contact form $dz-\la_\st$.

\begin{definition}
Let $\beta\in\Br_N^+$, $N\in\N$, and consider its smooth braid closure $c(\beta)$ in $\S^1\times\R\cong T^*\S^1$, depicted as a (horizontal) link diagram. Then $\beta\in\Br_N^+$ is said to be {\it admissible} if $c(\beta)$ is the Lagrangian projection of a Legendrian link $\La\sse(J^1\S^1,\xi_\st)$: that is, if there exists a Legendrian link $\La\sse(J^1\S^1,\xi_\st)$ such that $c(\beta)=\pi(\La)$ as link diagrams, where crossings are taken into account.\hfill$\Box$\label{def:NestedPigTail}
\end{definition}

As we now explain, if $\beta$ is admissible, then the isotopy in Figure~\ref{fig:LagrangianProjectionSatelliteClosure2} is legal and in particular it makes sense to refer to the $(-1)$-closure of $\beta$.

\begin{prop}
Suppose that $\beta$ is admissible. Then the diagram on the right of Figure~\ref{fig:LagrangianProjectionSatelliteClosure2} is the Lagrangian projection of a Legendrian link in $\R^3$, and the sequence of moves in Figure~\ref{fig:LagrangianProjectionSatelliteClosure2} represents a Legendrian isotopy.
\label{prop:admissible}
\end{prop}

\begin{proof}
Suppose $\beta$ is admissible, and let $\La$ be the Legendrian link in $J^1\S^1$ whose Lagrangian projection is $c(\beta)$. If we cut $\La$ at a point in $\S^1$ then we obtain a Legendrian braid in $J^1( [0,1])$ in the terminology of \cite{Satellite1}. By the classification of positive Legendrian braids \cite[Theorem 3.4]{Satellite1}, this braid is Legendrian isotopic to the Legendrian braid whose \textit{front} projection is $\beta$. It follows that this remains true when we satellite these braids around the standard Legendrian unknot $U \subset \R^3$. The satellite of the latter braid is $\Lambda(\beta)$ as defined in Section~\ref{ssec:LegendrianLinks}, which in the Lagrangian projection is the leftmost diagram in Figure~\ref{fig:LagrangianProjectionSatelliteClosure2}. On the other hand, one directly sees (without passing to the front projection) that the Lagrangian projection of the satellite of $\La$ is the rightmost diagram in Figure~\ref{fig:LagrangianProjectionSatelliteClosure2}. The result follows.
\end{proof}

Example \ref{ex:NestedPigTail} shows that not every braid $\beta\in\Br_N^+$ is admissible. Let us introduce a sufficiency criterion for a braid $\beta\in\Br_N^+$ to be admissible. For that, let $$\Delta_N=\prod_{i=1}^{N-1}\prod_{j=1}^{N-i}\sigma_j\in\Br_N^+$$ denote the half-twist on $N$ strands, i.e., 
the Garside element of the $N$-stranded braid group $\Br_N$.

\begin{prop}\label{prop:halftwist}
Any positive braid containing a half-twist is admissible, i.e. if $\beta_1,\beta_2$ are braids in $\Br_N^+$, then $\beta_1\Delta_N\beta_2$ is admissible.
\end{prop}

\begin{proof} 
Since admissibility depends only on the closure of the braid in the solid torus, we may move $\beta_2$ to the beginning of the braid; it thus suffices to show that if $\beta\in\Br_N^+$ then $\beta\Delta_N$ is admissible.
	For this, consider the standard front for $\wt\La(\beta\Delta_N)$ in $\S^1\times\R$
	and deform it, scanning left-to-right, using the resolution procedure in \cite[Section 2.1]{Ng03}: see \cite[Figure 3]{Ng03} and Figure \ref{fig:Admissible2} (left). The procedure described in \cite{Ng03} uses a front projection in $\R_q\times\R_z$, instead of $\S_q^1\times\R_z$, but can still be used with this latter base $\S^1\times\R$ by using the half-twist $\Delta_N\in\Br_N^+$, which is part of the braid $\beta\Delta_N$ by hypothesis. Indeed, this is depicted in Figure \ref{fig:Admissible2} (left), where the half-twist is shown in the yellow box. The Lagrangian projection associated to this deformed front is depicted in Figure \ref{fig:Admissible2} (right), where the half-twist $\Delta_N$ now appears thanks to the crossings associated to the (green) Reeb chords that appear at the right-most part of the front in $\S_q^1\times\R_z$. This concludes the statement.
\end{proof}

\begin{center}
	\begin{figure}[h!]
		\centering
		\includegraphics[scale=0.75]{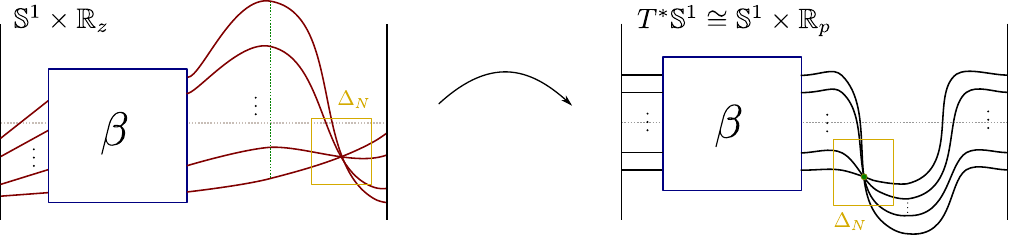}
		\caption{Deforming the front projection for $\wt\La(\beta\Delta_N)$ in $\S^1\times\R$ (left) so that the corresponding Lagrangian projection in $T^*\S^1$ is as shown on the right.}
		\label{fig:Admissible2}
	\end{figure}
\end{center}

For future reference we note the following variant on Proposition~\ref{prop:halftwist}. Each crossing in the Lagrangian projection in $T^*\S^1$ of a Legendrian link in $J^1\S^1$ corresponds to a Reeb chord of the link. A Reeb chord is called \textit{contractible} if its height can be made arbitrarily small without changing the Lagrangian projection of the link (up to planar isotopy).

\begin{prop}
If $\beta_1,\beta_2$ are braids in $\Br_N^+$, then any crossing coming from $\beta_1$ or $\beta_2$ in the admissible braid $\beta_1\Delta_N\beta_2$ is contractible.
\label{prop:halftwist2}
\end{prop}

\begin{proof}
First consider the special case where $\beta_2$ consists of a single crossing. We claim that this crossing is contractible. Indeed, a slight variant on the construction from Figure~\ref{fig:Admissible2} involving swapping two of the strands in the yellow box in the front projection gives the desired contractible crossing: see Figure~\ref{fig:Admissible3}.

\begin{center}
	\begin{figure}[h!]
		\centering
		\includegraphics[scale=0.75]{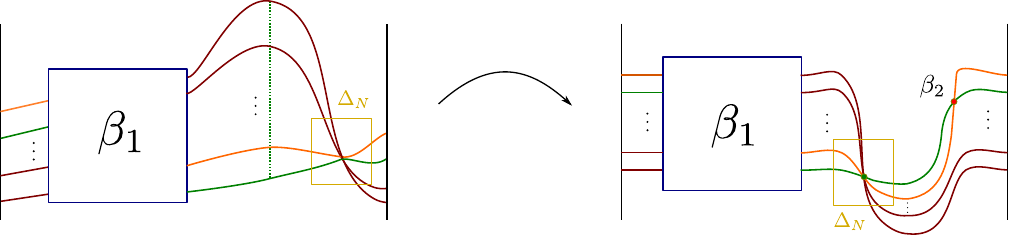}
		\caption{A variant on the argument from Proposition~\ref{prop:halftwist}. Here two of the strands passing through the yellow $\Delta_N$ box in the front projection swap places, producing a contractible crossing in the Lagrangian projection: the red dot on the right, corresponding to the single crossing of $\beta_2$.}
		\label{fig:Admissible3}
	\end{figure}
\end{center}

In the general case, cut the closure of the braid $\beta_1\Delta_N\beta_2$ at the specified crossing. Push $\Delta_N$ to the end of the resulting braid by a sequence of Reidemeister III moves. From the above special case, we can realize the resulting braid as the Lagrangian projection of a Legendrian link in such a way that the distinguished crossing is contractible. Then push $\Delta_N$ back into its original position without disturbing a neighborhood of the contractible crossing; this is a braid isotopy and thus corresponds to a Legendrian isotopy by the classification of positive Legendrian braids \cite{Satellite1}.
\end{proof}

\subsection{A class of Legendrian $(-1)$-closures} \label{ssec:satellite}
The Legendrian links that we use in this manuscript are particular examples of $(-1)$-closures,
obtained by the following procedure. Let $w(\beta)\in S_N$ be the permutation given by the Coxeter projection $w:\Br_N\lr S_N$ of $\beta$ onto the symmetric group, where the relations $\sigma_i^2=1$ are imposed for the Artin generators $i=1,\ldots,N-1$. Suppose that the bijection $w(\beta):[1,N]\lr[1,N]$ has a fixed point $i$, for some $i\in[1,N]$. Then the Legendrian $\La(\beta)$ contains a connected component $\La(\beta)_i$ which is a standard Legendrian unknot. Since $\La(\beta)_i\sse(\R^3,\xi_\st)$ is a Legendrian link, there exists a neighborhood $\Op(\La(\beta)_i)$, disjoint from $\La(\beta)\setminus \La(\beta)_i$, which is contactomorphic to $\Op(\La(\beta)_i)\cong (J^1\La(\beta)_i,\xi_\st)$, where the contactomorphism sends $\La(\beta)_i\sse\Op(\La(\beta)_i)$ to the zero section $\La(\beta)_i\sse(J^1\La(\beta)_i,\xi_\st)$.

Now, let $\gamma\in\Br_M^+$ be a positive $M$-stranded braid, $M\in\N$. Let us denote by $\La(\gamma)_i\sse\Op(\La(\beta)_i)$ the Legendrian link obtained by satelliting $\wt\La(\gamma)\sse(J^1\S^1,\xi_\st)$ along the standard Legendrian unknot $\La(\beta)_i\sse \Op(\La(\beta)_i)$.

\begin{definition}\label{def:LegSatComponent}
	Let $\beta\in\Br_N^+$ be such that $i\in[1,N]$ is a fixed point of $w(\beta)$, and $\gamma\in\Br_M^+$ be an $M$-stranded braid, $M\in\N$. The Legendrian link $\La(\beta,i;\gamma)\sse(\R^3,\xi_\st)$ is the Legendrian link $(\La(\beta)\setminus\La(\beta)_i)\cup\La(\gamma)_i\sse(\R^3,\xi_\st)$, where $\La(\gamma)_i\sse\Op(\La(\beta)_i)$ is embedded in an arbitrarily but fixed neighborhood of the component $\La(\beta)_i$. Colloquially, $\La(\beta,i;\gamma)$ is the result of satelliting the braid $\gamma$ around the component of the Legendrian link $\Lambda(\beta)$ labeled by $i$.
		\hfill$\Box$
\end{definition}

The Legendrian links in Theorem \ref{thm:main} are of the form $\La(\beta,i;\gamma)$ for $\gamma\in\Br_2^+$ and $\beta\in\Br^+_N$, where $N=2,3$. For instance, the Legendrian links $\La_n$ come from setting $\beta=\sigma_1^6$ and $\gamma=\sigma_1^n$ with $N=M=2$:
$$\La_n\cong\La(\sigma_1^6,1;\sigma_1^{n}).$$
Similarly, the Legendrian links $\La(\wt D_n)$, $n\geq4$, come from setting $\beta=(\sigma_1\sigma_2\sigma_2\sigma_1)^2\sigma_2^2$ and $\gamma=\sigma_1^{n-2}$ with $N=3,M=2$:
$$\La(\wt D_n)\cong\La((\sigma_1\sigma_2\sigma_2\sigma_1)^2\sigma_1^2,1;\sigma_1^{n-2}).$$

\begin{remark}\label{rmk:affineD4Notation} As noted in the introduction, the Legendrian links $\La(\wt D_n)$, $n\geq4$, are also the rainbow closures of the positive braids $$\eta_n=(\sigma_2\sigma_1\sigma_3\sigma_2\sigma_2\sigma_3\sigma_1\sigma_2)\sigma_1^{n-4},\quad n\geq4,\quad\eta_n\in\Br_4^+.$$
	The brick diagram \cite{Rudolph_QuasipositiveAnnuli,BLL18} associated to this positive braid word $\eta_n$ coincides with the Coxeter–Dynkin diagrams $\wt {D}_n$ associated to the affine Coxeter group of $D$-type. This affine Coxeter diagram also arises from two natural constructions starting with $\eta_n$. First, the quiver associated to the positive braid $\eta_n$, according to the algorithm in \cite{BFZ05}, and second, as the diagram for the intersection form associated to a set of (distinguished) generators in the first homology group of a minimal-genus Seifert surface associated to the link given by $\eta_n$ \cite{Misev17,BLL18}. In addition, the augmentation variety associated to $\La(\wt D_n)$ admits a cluster structure of $\wt D_n$-type. These reasons lead us to the notation $\La(\wt D_n)$ and referring to these braids as the (maximal-tb) affine $D_n$-Legendrian links.\hfill$\Box$
\end{remark}

Definition \ref{def:LegSatComponent} is rather direct diagrammatically. Indeed, given the front diagram for $\La(\beta)\sse(\R^3,\xi_\st)$ shown in Figure \ref{fig:LagrangianProjectionSatelliteClosure} (left), a front diagram for $\La(\beta,i;\gamma)\sse(\R^3,\xi_\st)$ is obtained by taking the $M$-copy Reeb push-off of the $i$-th component of $\La(\beta)$, corresponding to the $i$-th strand in $\Br_N^+$, and inserting the front diagram for $\wt\La(\gamma)$. This is shown in Figure \ref{fig:FrontProjection2}.

\begin{center}
	\begin{figure}[h!]
		\centering
		\includegraphics[scale=0.8]{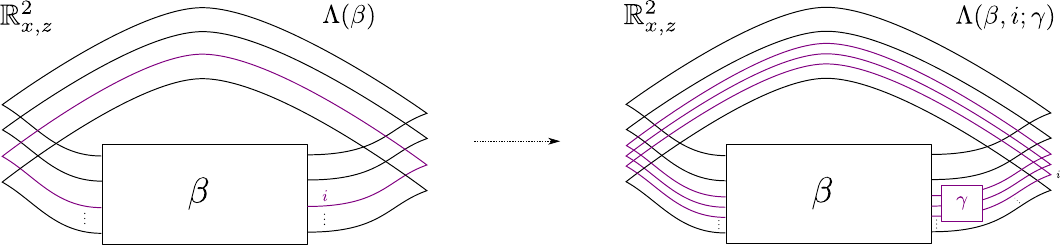}
		\caption{Front projections for the Legendrian links $\La(\beta)\sse(\R^3,\xi_\st)$ (left) and $\La(\beta,i;\gamma)\sse(\R^3,\xi_\st)$ (right).}
		\label{fig:FrontProjection2}
	\end{figure}
\end{center}

Similarly, this construction is depicted in the Lagrangian projection in Figure \ref{fig:SatelliteUnknotComponent}.

\begin{center}
	\begin{figure}[h!]
		\centering
		\includegraphics[scale=1]{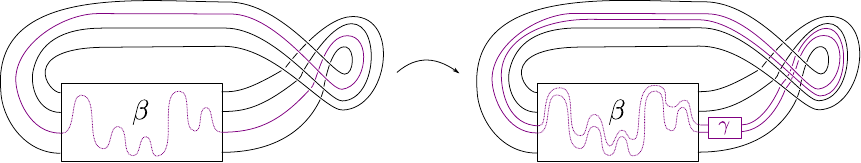}
		\caption{Lagrangian projections for the Legendrian links $\La(\beta)\sse(\R^3,\xi_\st)$ (left) and $\La(\beta,i;\gamma)\sse(\R^3,\xi_\st)$ (right). These are the Lagrangian projections that we use in order to compute the Legendrian contact DGA.}
		\label{fig:SatelliteUnknotComponent}
	\end{figure}
\end{center}

The crucial property of the Legendrian links $\La(\beta,i;\gamma)\sse(\R^3,\xi_\st)$ is the existence of a specific contact isotopy $\varphi_t:(\R^3,\xi_\st)\lr(\R^3,\xi_\st)$, $t\in[0,1]$, such that $\varphi_1(\La(\beta,i;\gamma))=\La(\beta,i;\gamma)$ and $\varphi_t|_{\R^3\setminus\Op(\Lambda(\beta)_i)}=\mbox{Id}$ for all $t\in[0,1]$, as we now explain.

\subsection{The purple-box Legendrian loop}\label{ssec:PurpleBox} Let $\beta\in\Br_N^+,\gamma\in\Br_M^+$ and consider the Legendrian link $\La(\beta,i;\gamma)\sse(\R^3,\xi_\st)$. We construct a Legendrian loop $\vartheta:\thinspace\S^1\to\SL(\La(\beta,i;\gamma))$ based at $\La(\beta,i;\gamma)$, whose action on the Legendrian contact DGA of $\La(\beta,i;\gamma)$ will be studied in Section \ref{sec:DGA}, and subsequently lead to Theorem \ref{thm:main}. Intuitively, the Legendrian loop $\vartheta$ will fix the components of the Legendrian link $\La(\beta)$ which do not belong to the satellite $\La(\gamma)\sse\La(\beta)$, and induce a rotation of $\La(\gamma)$ corresponding to one full revolution of the $\S^1$ direction in $J^1\S^1$. Let us provide the details for its rigorous description.

Consider the component $\La(\beta)_i\sse\La(\beta)$ with a standard neighborhood $\Op(\La(\beta)_i)$ and the Legendrian link $\wt\La(\gamma)\sse\Op(\La(\beta)_i)$. Fix a contactomorphism
$$\Op(\La(\beta)_i)\cong (J^1\S^1_\theta,\ker(dz-p_\theta d\theta)),$$
where $(J^1\S^1_\theta,\ker(dz-p_\theta d\theta))$ is the 1-jet space with coordinates $(\theta,p_\theta)\in T^*\S^1$, $z\in\R$. Fix the standard round metric in $\S^1$, and choose $R\in\R^+$ such that $\wt\La(\gamma)\sse B_R$, where $B_R=\D_R(T^*\S^1)\times[-R,R]  \subset T^*\S^1\times\R$, with $\D_R(T^*\S^1)$ being the radius $R$ (open) disk bundle.

Now, consider the Hamiltonian $p_\theta:J^1\S^1_\theta\lr\R$ and its associated contact vector field $X_{p_\theta}=-\dd_\theta$. Let $\varepsilon\in\R^+$, and choose a smooth cut-off function $\chi:J^1\S^1\lr\R$ such that
$$\chi|_{B_{R+\varepsilon}}\equiv1,\qquad \chi|_{B_{R+2\varepsilon}\setminus B_{R+\varepsilon}}\equiv0.$$
The contact vector field $X_\vartheta$ associated to the Hamiltonian $\chi\cdot p_\theta:J^1\S^1\lr\R$ restricts to $-\dd_\theta$ in the tube $B_R$ containing $\wt\La(\gamma)$, and it vanishes away from $B_R$. The contact flow of $X_\vartheta$ yields a compactly supported contact isotopy $\wt\Theta_t:(J^1\S^1,\xi_\st)\lr (J^1\S^1,\xi_\st)$, which we parametrize such that $t=1$ is the smallest $t\in\R^+$ with $\wt\Theta_t(\wt\La(\gamma))=\wt\La(\gamma)$ pointwise.

\begin{definition}\label{def:LegLoop}
	Let $\beta\in\Br_N^+$ be such that $i\in[1,N]$ is a fixed point of $w(\beta)$, and let $\gamma\in\Br_M^+$ be an $M$-stranded braid, $M\in\N$. The $\Theta_t$-contact isotopy associated to $\La(\beta,i;\gamma)\sse(\R^3,\xi_\st)$, $t\in[0,1]$, is the compactly supported isotopy obtained by extending the compactly supported contact isotopy $\wt\Theta_t:\Op(\La(\beta)_i)\lr\Op(\La(\beta)_i)$, $t\in[0,1]$, by the identity map on the complement of $\Op(\La(\beta)_i)$.
A Legendrian loop $\vartheta:\S^1\lr\SL(\La(\beta,i;\gamma))$ is said to be a \textit{$\vartheta$-loop} if it is obtained as $\Theta_t(\La(\beta,i;\gamma))$, $t\in[0,1]$, for a $\Theta_t$-contact isotopy associated to $\La(\beta,i;\gamma)\sse(\R^3,\xi_\st)$.\hfill$\Box$
\end{definition}

We will also call the Legendrian $\vartheta$-loops in Definition \ref{def:LegLoop} \textit{purple-box Legendrian loops}, as they are obtained by moving the purple box which contains the braid $\gamma$ clockwise around until it comes back to itself. Figure \ref{fig:SatelliteUnknotComponent2} (left) provides a schematic picture of such a $\vartheta$-loop.

\begin{center}
	\begin{figure}[h!]
		\centering
		\includegraphics[scale=1]{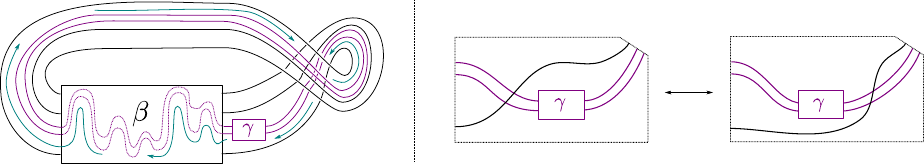}
		\caption{Left: a Legendrian $\vartheta$-loop for the Legendrian link $\La(\beta,i;\gamma)\sse(\R^3,\xi_\st)$, where the purple $\gamma$-box moves around clockwise around the $\beta$-box and comes back to itself using the upper strands). Right: the local move, consisting of a sequence of $l(\gamma)$ Reidemeister III moves, which we use in order to push the purple $\gamma$-box, right to left, through the $\beta$-box.}
		\label{fig:SatelliteUnknotComponent2}
	\end{figure}
\end{center}

From a computational viewpoint, it is important to stress that a Legendrian $\vartheta$-loop can be described in the Lagrangian projection strictly in terms of Reidemeister III moves and planar isotopies.\footnote{One could instead use the resolution of the front projection as in Figure~\ref{fig:LagrangianProjectionSatelliteClosure}, and similarly push the purple $\gamma$-box around the front projection; this is e.g. what K\'alm\'an does in \cite{Kalman}. However, this version of the isotopy requires the use of both Reidemeister III and II moves. Our setup does not require Reidemeister II moves and this consequently simplifies our computations with the Legendrian contact DGA.}
In precise terms, a Legendrian $\vartheta$-loop consists of two pieces:
\begin{itemize}
	\item[(i)] Transferring the purple $\gamma$-box through the $\beta$-box, through a sequence of Reidemeister III moves. Indeed, it suffices to notice that moving the purple $\gamma$-box through one strand is achieved by $l(\gamma)$ consecutive Reidemeister III moves, one per each crossing of $\gamma$. This local move, past one strand, is shown in Figure \ref{fig:SatelliteUnknotComponent2} (right). Thus, the purple $\gamma$-box can be pushed through the $\beta$-box, right to left, by performing $l(\beta)\cdot l(\gamma)$ Reidemeister III moves.
	\item[(ii)] Moving the purple $\gamma$-box from the left of the $\beta$-box to its right {\it using the upper strands}. This is achieved by a planar isotopy, which moves the purple $\gamma$-box up and to the right (leaving the $\beta$-box beneath and passing above it), and then applying $N^2\cdot l(\gamma)$ Reidemeister III moves to make the purple $\gamma$-box go around the pig-tailed loop until it returns to its initial position.
\end{itemize}
Hence, using a total of $l(\gamma)\cdot(N^2+l(\beta))$ Reidemeister III moves in the Lagrangian projection, we can realize the Legendrian $\vartheta$-loops in Definition \ref{def:LegLoop}.

\begin{remark} Legendrian $\vartheta$-loops can be considered as elements in $\pi_1(\SL(\La(\beta,i;\gamma)))$, or we can graph them in the symplectization as Lagrangian self-concordances $L_\vartheta\sse(\R\times\R^3,\la_\st)$ from the Legendrian link $\La(\beta,i;\gamma)\sse(\R^3,\xi_\st)$ to itself. Most interestingly, given an exact Lagrangian filling $L\sse(\D^4,\la_\st)$ of $\La(\beta,i;\gamma)\sse(\S^3,\xi_\st)$, we can concatenate $L$ with $L_\vartheta$, at the convex end of $L$ and the concave end of $L_\vartheta$. One may ask whether concatenating Lagrangian fillings with $L_\vartheta$ yields new Lagrangian fillings not Hamiltonian isotopic to $L$. Theorem \ref{thm:main} shows that there are Legendrian links where concatenating certain Lagrangian fillings with $k$ consecutive copies of $L_\vartheta$ yields (infinitely many) pairwise distinct Lagrangian fillings, for different values of $k\in\N$. \hfill$\Box$
\end{remark}

\begin{ex} Legendrian $\vartheta$-loops behave differently depending on the choice of braids $\beta\in\Br_N^+$ and $\gamma\in\Br_M^+$. For example, if $\gamma\in\Br_1^+$ is the trivial 1-stranded braid, then the $\vartheta$-loop is constant on the entire link $\La(\beta,i;\gamma)\sse(\R^3,\xi_\st)$, regardless of the choice of $\beta\in\Br_N^+$. On the other hand, if we choose the braid $\beta$ to be 1-stranded and the purple box $\gamma=\sigma_1^{n+2}\in\Br_2^+$ to be 2-stranded, then we recover K\'alm\'an's Legendrian loop of $(2,n)$-torus links \cite{Kalman}. In this case, \cite[Theorem 1.3]{Kalman} shows that the action of the $\vartheta$-loop on the degree-0 Legendrian contact homology of $\La(\beta,i;\gamma)$ is nontrivial but of finite order. See Section~\ref{ssec:Kalman} for further discussion of the K\'alm\'an loop.\hfill$\Box$
\end{ex}



\section{Legendrian Contact DGAs and Cobordism Maps}
\label{sec:cobordism}

In this section, we review the definition of the Legendrian contact DGA, with particular attention paid to integer and group-ring coefficients and the role of spin structures. We then proceed to discuss maps between DGAs induced by exact Lagrangian cobordisms, including exact Lagrangian fillings. There is now a reasonably large literature about these cobordism maps, beginning with work of Ekholm, Honda, and K\'alm\'an \cite{EHK} defining the maps over $\Z_2$; we will need to compute a lift of these maps to $\Z$, which abstractly exists by work of Karlsson \cite{Karlsson-compute,Karlsson-cob}. In this section we will present a framework that will allow us to perform explicit combinatorial computations of the cobordism maps over $\Z$, building them out of maps corresponding to particular elementary cobordisms. The maps for these elementary cobordisms are then presented in the following section, Section~\ref{sec:elementary}.

\subsection{
The Legendrian contact DGA}
\label{ssec:dga-def}

The Legendrian contact DGA, also known as the Chekanov--Eliashberg DGA, has been well-studied in the literature, especially in the setting of $(\R^3,\xi_\st)$. For the definition of the DGA in this setting, we refer the reader e.g.\ to \cite{Chekanov} for the original definition over $\Z_2$, \cite{ENS} for the definition over $\Z[t^{\pm 1}]$ (see also the survey \cite{ENsurvey}), and \cite{Satellite2,NRSSZ} for an upgraded definition with multiple base points. Here we will briefly review the definition that we will use, with $\Z$ coefficients and multiple base points.

Let $\Lambda$ be an oriented Legendrian link in $(\R^3,\xi_\st)$ equipped with a number of base points, such that there is at least one base point on each component. We will assume that $\Lambda$ is sufficiently generic that the $xy$ projection $\Pi_{xy}(\Lambda)$ in $\R^2$ is immersed with only transverse double point singularities, and no base point lies at one of these double points. We label the crossings of $\Pi_{xy}(\Lambda)$, which correspond to Reeb chords of $\Lambda$, as $a_1,\ldots,a_r$, and decorate each base point with a monomial of the form $\pm s_i^{\pm 1}$. Let $\{s_1,\ldots,s_q\}$ be the collection of indeterminates that appear in the labeling of the base points. To this decorated oriented Legendrian link $\Lambda$, we can associate the Legendrian contact DGA $(\SA_\Lambda,\dd)$, as follows.

{\bf Generators.} The algebra $\SA_\Lambda$ is the unital tensor algebra over the coefficient ring $\Z[s_1^{\pm 1},\ldots,s_q^{\pm 1}]$ generated by $a_1,\ldots,a_r$. (One can lift this to the ``fully noncommutative'' algebra where the coefficients $s_i^{\pm 1}$ do not commute with Reeb chords $a_i$, and in our computations we will sometimes order our monomials accordingly. However, for the purposes of this paper, we will always assume that coefficients and Reeb chords commute.) 

{\bf Grading.} We assume for simplicity that each component of $\Lambda$ has rotation number $0$, which will be the case for the Legendrian links we study. The algebra $\SA_\Lambda$ is then graded over the integers $\Z$; if $\Lambda$ has a single component, then this grading is well-defined, while if $\Lambda$ has multiple components, the grading depends on some additional choices. We will fix the grading by choosing a collection of distinguished base points, one on each component, such that the oriented tangent vectors to $\Pi_{xy}(\Lambda)$ at these points are all parallel in $\R^2$. Label these base points by $t_1,\ldots,t_m$, where $m$ is the number of components of $\Lambda$ and the base point $t_j$ is on the $j$-th component. Consider a Reeb chord $a\in\SA_\Lambda$ that ends on component $r(a)$ and begins on component $c(a)$; we define a capping path $\gamma_a$ along $\Lambda$ to be the concatenation of a path from the beginning point (undercrossing) of $a$ to $t_{c(a)}$, and a path from $t_{r(a)}$ to the ending point of $a$, following the orientation of $\Lambda$ for both paths. As we traverse $\gamma_a$, the unit tangent vector to $\Pi_{xy}(\gamma_a)$ changes continuously from the tangent vector to the undercrossing at $a$ to the tangent vector to the overcrossing; let $r(\gamma_a) \in \R$ denote the number of counter-clockwise revolutions around $\S^1$ that the tangent vector makes during this process, and note that $r(\gamma_a) \not\in \frac{1}{2}\Z$ because of transversality. Then the grading of $a$ is defined to be $-\lceil 2r(\gamma_a) \rceil\in\Z$. We also place all the marked point monomials $s_i$ in grading $0$, which completes the grading of $\SA_\Lambda$.

{\bf Differential.} In order to set up the differential $\dd$ on $\SA_\Lambda$, we first decorate the four quadrants at each crossing of $\Pi_{xy}(\Lambda)$ by two signs, a Reeb sign and an orientation sign. At each crossing, two opposite quadrants have Reeb sign $+$ and the others have Reeb sign $-$, while the orientation signs depend on whether the crossing is positive (even degree) or negative (odd degree): for positive crossings, two quadrants have orientation sign $+$ and two have $-$, while for negative crossings, all four quadrants have orientation sign $+$. See Figure~\ref{fig:Reeb-ori-signs}.

\begin{center}
	\begin{figure}[h!]
		\centering
				\includegraphics[scale=1.5]{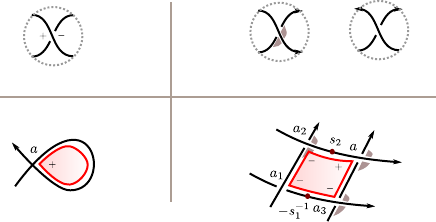}
		\caption{
		In the top row, the Reeb signs (left diagram) and orientation signs (two right diagrams) at a crossing. Quadrants that have $-$ orientation sign are shaded, while all other quadrants have $+$ orientation sign. In the bottom row, two examples of disks in $\Delta(a)$. Both disks have $\sgn = +1$ (on the right, the corner with negative orientation sign cancels the $-$ in $-s_1^{-1}$) and they contribute $+1$ and $+s_2^{-1}a_2a_1s_1^{-1}a_3$, respectively, to $\dd(a)$.}
		\label{fig:Reeb-ori-signs}
	\end{figure}
\end{center}

The differential now counts immersions of a disk $D^2$ with boundary punctures to $\R^2$, mapping the boundary of $D^2$ to $\Pi_{xy}(\Lambda)$, such that a neighborhood of each boundary puncture is mapped to one of the four quadrants at a crossing of $\Pi_{xy}(\Lambda)$. We call such a disk an \textit{immersed disk} for short; each corner of an immersed disk is a \textit{positive $(+)$ corner} or a \textit{negative $(-)$ corner} depending on the Reeb sign of the quadrant. For a Reeb chord $a$, define $\Delta(a)$ to be the set of immersed disks (up to reparametrization) with a single $+$ corner at $a$ and no other $+$ corners. To any such disk $\Delta\in\Delta(a)$, we can define two quantities. One is the sign $\sgn(\Delta) \in \{\pm 1\}$, given by the product of the orientation signs over all corners of $\Delta$, multiplied by the signs of any base points traversed by the boundary of the disk ($+1$ for any base point labeled by $s_i^{\pm 1}$ and $-1$ for any base point labeled by $-s_i^{\pm 1}$). The other is the word $w(\Delta) \in \SA_\Lambda$, which is the product, in order, of the Reeb chords at the $-$ corners and the base points that are encountered as we traverse the boundary of the disk counterclockwise, beginning and ending at the corner at $a$. A base point labeled by $\pm s_i^{\pm 1}$ contributes $s_i^{\pm 1}$ if it is traversed along the orientation of $\Lambda$ and $s_i^{\mp 1}$ if it is traversed oppositely.
The differential $\dd(a)$ is now defined to be:
\[
\dd(a): = \sum_{\Delta\in\Delta(a)} \sgn(\Delta) w(\Delta).
\]
See Figure~\ref{fig:Reeb-ori-signs} for an example.

\begin{remark}[multiple base points]
In order to count augmentations over $\Z$, it is important that each component of $\Lambda$ have at least one base point. Adding extra base points beyond one per component changes the DGA in a simple way. First note that moving a base point labeled $\pm s_i^{\pm 1}$ along $\Lambda$ and through a crossing $a$ has the effect of replacing $a$ by $(\pm s_i^{\pm 1})^{\pm 1} a$: that is, the algebra is the same before and after the move, and the differential changes by conjugation by the automorphism that sends $a$ to $(\pm s_i^{\pm 1})^{\pm 1} a$ and fixes all other generators. Thus if we have multiple base points on a single component, then up to a $\Z[s_1^{\pm 1},\ldots,s_q^{\pm 1}]$-algebra isomorphism of the DGA, we can assume that all of the base points lie on the same segment of $\Pi_{xy}(\Lambda)$. In this case we can replace the multiple base points by a single base point labeled by the product of their labels, and the differential is unchanged.\hfill$\Box$
\label{rmk:multiple}
\end{remark}

\begin{remark}[dependence on spin structure]
In the differential over $\Z$ of the Legendrian contact DGA $(\SA_\Lambda,\dd)$ of a link $\Lambda$, the signs depend on a choice of spin structure on $\Lambda$, as laid out by the construction of Ekholm, Etnyre, and Sullivan \cite{EkholmEtnyreSullivan05c}. 
\label{rmk:spin}
For each $\S^1$ connected component of the Legendrian link $\Lambda$, there are two spin structures: the \textit{Lie group} spin structure, induced by the fact that the 1-sphere $\S^1$ is a Lie group, and the \textit{null-cobordant} spin structure, induced by the fact that $\S^1$ bounds a 2-disk $D^2$ and we can restrict the unique spin structure on $D^2$ to the boundary $\S^1$. Here we review the discussion in \cite{EkholmEtnyreSullivan05c} about how the choice of spin structure affects the differential $\dd$ in $(\SA_\Lambda,\dd)$.

Choose one base point on each of the $m$ components of $\Lambda$, so that $\SA_\Lambda$ is an algebra over $R=\Z[t_1^{\pm 1},\ldots,t_m^{\pm 1}]$, and 
write $\dd^\comb$ for the combinatorial differential on $\SA_\Lambda$ as defined above. The set of spin structures on $\Lambda$ is an affine space based on $H_1(\Lambda,\Z_2) \cong \Z_2^m$; of interest to us will be two spin structures differing by $(1,\ldots,1)$, given by choosing the Lie group spin structure or the null-cobordant spin structure on {\it all} components of $\Lambda$. We will write $\dd^\Lie$ and $\dd^\NC$ for the geometric differentials on $\SA_\Lambda$ corresponding to these two spin structures.\footnote{The superscript $\dd^\NC$ stands for Null-Cobordant.} The two differentials $\dd^\Lie$ and $\dd^\NC$ depend on a number of auxiliary choices, including capping operators for Reeb chords---see Section~\ref{ssec:natural} below for further discussion---but up to $R$-algebra isomorphism, $(\SA_\Lambda,\dd^\Lie)$ and $(\SA_\Lambda,\dd^\NC)$ are well-defined.

The combinatorial differential $\dd^\comb$ comes from the Lie group spin structure on $\Lambda$. To be precise, in \cite[Theorem~4.32]{EkholmEtnyreSullivan05c} it is shown that one can make choices so that $\dd^\Lie$ agrees with our definition of $\dd^\comb$ with signs as in Figure~\ref{fig:Reeb-ori-signs}, except that for positive crossings (the left diagram on the top right of Figure~\ref{fig:Reeb-ori-signs}), the opposite two quadrants are shaded.\footnote{In fact \cite[Theorem~4.32]{EkholmEtnyreSullivan05c} presents two choices of signs for $\dd^\Lie$, of which we are describing one; however, it was subsequently proven in \cite{Ng-SFT} that the two choices lead to isomorphic DGAs.} This change of shading corresponds to the $R$-algebra isomorphism of $\SA_\Lambda$ sending each Reeb chord $a$ to $-a$ for even-graded Reeb chords and $+a$ for odd-graded Reeb chords, and so this isomorphism sends $(\SA_\Lambda,\dd^\comb) \stackrel{\cong}{\longrightarrow} (\SA_\Lambda,\dd^\Lie)$.

For cobordisms, the null-cobordant spin structure is more natural than the Lie group spin structure. To compute $\dd^\NC$, we can appeal to \cite[Theorem~4.29]{EkholmEtnyreSullivan05c} (see also Remark 4.35 from the same paper), which implies that changing the spin structure by $(c_1,\ldots,c_m) \in \Z_2^m$ has the effect of replacing $t_i$ by $(-1)^{c_i} t_i$ for $i=1,\ldots,m$. In particular, define the $\Z$-algebra isomorphism $\phi :\thinspace \SA_\Lambda \to \SA_\Lambda$ by $\phi(a)=a$ for all Reeb chords $a$ and $\phi(t_i)=-t_i$ for all $i$; then
\[
\phi :\thinspace (\SA_\Lambda,\dd^\Lie) \stackrel{\cong}{\longrightarrow} (\SA_\Lambda,\dd^\NC).
\]

More generally, suppose that we have multiple base points on each component of $\Lambda$ as in Remark~\ref{rmk:multiple}, each decorated by a monomial of the form $s_i^{\pm 1}$. Then, since no base point introduces a sign, the resulting combinatorial DGA $(\SA_\Lambda,\dd^\comb)$ over $\Z[s_1^{\pm 1},\ldots,s_q^{\pm 1}]$ has signs corresponding to the Lie group spin structure. Now suppose that $\cS$ is any subset of these base points. If we replace the decoration $s_i^{\pm 1}$ of each base point in $\cS$ by $-s_i^{\pm 1}$, we obtain a new differential $\dd^\cS$ on $\SA_\Lambda$. Then $\dd^\cS$ gives the differential corresponding to the spin structure that differs from the Lie group spin structure by $(c_1,\ldots,c_m) \in \Z_2^m$, where $c_i$ is the number of base points in $\cS$ that lie on component $i$. In particular, if $\cS$ has an odd number of points on each component\footnote{From a geometric viewpoint, $\cS$ indicates points where we add a $\pi$-rotation to the Lie group trivialization of the stabilized tangent bundle to $\S^1$. Doing this an odd number of times on each component yields the null-cobordant trivialization. See \cite[Remark 4.35]{EkholmEtnyreSullivan05c}.}, then we have an isomorphism of DGAs over $\Z[s_1^{\pm 1},\ldots,s_q^{\pm 1}]$:
\[
(\SA_\Lambda,\dd^\cS) \cong (\SA_\Lambda,\dd^\NC).
\]
\hfill$\Box$
\end{remark}

\subsection{Link automorphisms}
\label{ssec:linkaut}

In the case where $\Lambda$ is a multi-component Legendrian link, rather than a knot, there is a structure on the Legendrian contact DGA of $\Lambda$ that is hidden in the knot case. This is the ``link grading'' first introduced by K.\ Mishachev \cite{Mishachev}, which essentially gives the DGA the structure of a path algebra (the ``composable algebra'') on a graph whose vertices are components of $\Lambda$ and whose edges are Reeb chords of $\Lambda$. This structure leads to a family of automorphisms of the DGA of the Legendrian link $\Lambda$, which we call {\it link automorphisms}. These will feature in our discussion at various points, and we discuss them now in detail.\\

Let $\Lambda = \Lambda_1 \cup \cdots \cup \Lambda_m$ be an $m$-component Legendrian link. For any Reeb chord $a$ of $\Lambda$, define $r(a),c(a) \in \{1,\ldots,m\}$ to be the number of the component containing the endpoint (for $r(a)$) or beginning point (for $c(a)$) of $a$. The key observation of Mishachev is the following: in the DGA for $\Lambda$, any term in the differential $\dd a$ of a Reeb chord $a$ must be of the form $a_{i_1}\cdots a_{i_k}$, where $r(a) = r(a_{i_1}), c(a_{i_1}) = r(a_{i_2}), \ldots, c(a_{i_{k-1}}) = r(a_{i_k}), c(a_{i_k}) = c(a)$. This motivates the following definition.

\begin{definition}
Let $\Lambda$ be an $m$-component Legendrian link and $(\SA_\Lambda,\dd)$ its DGA. A \textit{link automorphism} of $\Lambda$ is an algebra automorphism $\Omega:\SA_\Lambda\lr\SA_\Lambda$ of the following form: there exist units $u_1,\ldots,u_m$ in the coefficient ring of $\SA_\Lambda$ such that for all Reeb chords $a$,
\label{def:linkaut}
\[
\Omega(a) = u_{r(a)} u_{c(a)}^{-1} a.
\]\hfill$\Box$
\end{definition}

The following is an immediate consequence of Mishachev's observation.

\begin{prop}
Let $\La$ be a Legendrian link. Any link automorphism $\Omega:\SA_\Lambda\lr\SA_\Lambda$ is a chain map of the Legendrian DGA $(\SA_\Lambda,\dd)$.\hfill$\Box$
\end{prop}

In addition, Mishachev's link grading structure is preserved by Legendrian isotopy, as can be checked by keeping track of components in the DGA chain maps induced by Legendrian isotopy. See \cite{Mishachev}, and see Section~\ref{ssec:isotopy} below for explicit formulas for these chain maps. As a consequence, link automorphisms persist under Legendrian isotopy:

\begin{prop}
Suppose $\Lambda$ and $\Lambda'$ are Legendrian isotopic links with respective DGAs $(\SA_\Lambda,\dd)$ and $(\SA_{\Lambda'},\dd)$. Suppose that $\Psi :\thinspace (\SA_\Lambda,\dd) \to (\SA_{\Lambda'},\dd)$ is the DGA map induced by a Legendrian isotopy. If $\Omega :\thinspace \SA_\Lambda \to \SA_\Lambda$ is a link automorphism of $\Lambda$, then there is a corresponding link automorphism $\Omega' :\thinspace \SA_\Lambda' \to \SA_\Lambda'$ of $\Lambda'$ such that $\Omega' \circ \Psi = \Psi \circ \Omega$.
\label{prop:linkaut}
\end{prop}

\begin{proof}
The numbering of the components of the Legendrian link $\Lambda$ induces a corresponding numbering of the components of the link $\Lambda'$. If $\Omega$ is defined by $\Omega(a) = u_{r(a)} u_{c(a)}^{-1} a$ for some $(u_1,\ldots,u_m)$, then we define $\Omega'$ in the same way: $\Omega'(a):= u_{r(a)} u_{c(a)}^{-1} a$. Since $\Psi$ preserves the link grading, it follows that it intertwines $\Omega$ and $\Omega'$, as desired.
\end{proof}

Link automorphisms will appear in our discussion in two related ways. First, they naturally arise when considering the family of augmentations induced by an exact Lagrangian filling, as we will next describe in Section~\ref{ssec:geom-cob}. Second, in Section~\ref{ssec:EHK-map} below,  we describe a formula for the cobordism map over $\Z$ associated to a saddle cobordism; our proof that the formula is correct is indirect and essentially reduces to arguing that there is only one possible candidate for the cobordism map that is actually a chain map over $\Z$. However, the existence of link automorphisms forces us to qualify this statement, since composing a chain map with a link automorphism produces another chain map. See Proposition~\ref{prop:EHK} and Appendix \ref{sec:saddle-map}.

\subsection{The geometric map induced by an exact Lagrangian cobordism}
\label{ssec:geom-cob}

Suppose that $\Lambda$ is a Legendrian link in $(\R^3,\xi_\std)$ and that $L\sse(\R^4,\la_\st)$ is a Lagrangian filling of $\Lambda$. Then $L$ induces an augmentation of the Legendrian contact DGA $(\SA_\Lambda,\dd)$. More precisely, the filling $L$ equipped with a rank $1$ local system induces an augmentation; put another way, the filling gives a family of augmentations and the additional choice of a local system picks out one of these. In the setting of $(\R^3,\xi_\std)$, the study of augmentations coming from fillings was initiated by Ekholm, Honda, and K\'alm\'an \cite{EHK}, who proved that an exact filling induces an augmentation over the group ring $\Z_2[H_1(L)]$ through a count of rigid holomorphic disks in the symplectization of $\R^3$ with boundary on $L$. 
Karlsson \cite{Karlsson-cob} subsequently lifted $\Z_2$ to $\Z$ by showing that the relevant moduli spaces of holomorphic disks can be coherently oriented. We summarize all of this work as follows.

\begin{thm}[\cite{EHK,Karlsson-cob}]
Suppose that $L$ is an (oriented, embedded, exact) Lagrangian filling of the Legendrian link $\Lambda \subset (\R^3,\xi_\std)$ with Maslov number $0$. 
\label{thm:EHK}
Then $L$ induces a DGA map
\[
\varepsilon_L :\thinspace (\SA_\Lambda,\dd) \to (\Z[H_1(L)],0)
\]
where $\Z[H_1(L)]$ lies entirely in grading $0$. $($The map $\varepsilon_{L}$ is referred to as an augmentation.$)$ Furthermore, if $L$ and $L'$ are Lagrangian fillings of $\Lambda$ which are isotopic through exact Lagrangian fillings of $\Lambda$, then the corresponding augmentations $\varepsilon_L$ and $\varepsilon_{L'}$ are DGA homotopic maps.\hfill$\Box$
\end{thm}

Note that an exact Lagrangian isotopy extends to an ambient Hamiltonian isotopy, e.g.~ by \cite[Section 3.6]{Oh_BookVol1}, especially \cite[Theorem 3.6.7]{Oh_BookVol1}, and see also \cite[Exercise 6.1.A]{Polterovich_Book01}. Conversely, the image of an exact Lagrangian submanifold under a Hamiltonian diffeomorphism remains exact, and thus exact Lagrangian isotopies are equivalent to Hamiltonian isotopies. In fact, this also holds with compact support: \cite[Theorem 3.6.7]{Oh_BookVol1} implies that a compactly supported exact Lagrangian isotopy extends to a compactly supported Hamiltonian isotopy.

\begin{remark}
\label{rmk:DGAhomotopy}
For the definition of DGA homotopic maps, see e.g. \cite{Kalman,EHK,NRSSZ}; we omit the definition here because for the Legendrian links that we consider in this paper, we can replace ``DGA homotopic'' by ``the same''.
All of our links $\Lambda$ have rotation number $0$ on each component, and all of the fillings that we construct are composed of minimum cobordisms and saddle cobordisms at Reeb chords with degree $0$. It follows that each of these fillings has Maslov number $0$. In addition, for all choices of $\Lambda$ in this paper, all Reeb chords lie in nonnegative degree (in fact, in degree $0$ or $1$), and so $\SA_\Lambda$ is supported entirely in nonnegative degree. In this setting, two DGA maps $(\SA_\Lambda,\dd) \to (\Z[H_1(L)],0)$ are DGA homotopic if and only if they are equal. Thus if two fillings $L,L'$ produce augmentations to $\Z[H_1(L)]$ that are distinct (under an isomorphism identifying $H_1(L)$ and $H_1(L')$), then we can use Theorem~\ref{thm:EHK} to conclude that $L,L'$ are not exact Lagrangian isotopic (or, equivalently in this setting, not Hamiltonian isotopic).\hfill$\Box$
\end{remark}

\begin{remark}
The augmentation $\varepsilon_L$ depends on a choice of spin structure on the filling $L$, as explained in \cite{Karlsson-cob}. If we change the spin structure by an element $\vartheta\in H^1(L;\Z_2)$, then we can define an isomorphism $\Z[H_1(L)] \to \Z[H_1(L)]$ by $x \mapsto (-1)^{\vartheta(x)} x$, and the augmentation changes by composition with this isomorphism. This does not change the augmentation up to equivalence, in the sense of Definition~\ref{def:system} below.\hfill$\Box$
\end{remark}

It will be convenient for us to enlarge the coefficient ring $\Z[H_1(L)]$ to incorporate link automorphisms, as introduced in Section \ref{ssec:linkaut} above. Suppose that $\La$ is a Legendrian link with $m$ components. Recall that given units $u_1,\ldots,u_m$, we can define a link automorphism $\Omega :\thinspace (\SA_\Lambda,\dd) \to (\SA_\Lambda,\dd)$. Any augmentation of $(\SA_\Lambda,\dd)$ can be composed with this link automorphism to produce another augmentation, and so a single augmentation produces an $(m-1)$-parameter family of augmentations. This family is parametrized by $s_1,\ldots,s_{m-1}$, where we define $s_i = u_i/u_m$ for $i\leq m-1$. We restate this observation as follows.

Consider the ring $\Z[H_1(L)][s_1^{\pm 1},\ldots,s_{m-1}^{\pm 1}] \cong \Z[H_1(L) \oplus \Z^{m-1}]$. Then the augmentation $\varepsilon_L :\thinspace (\SA_{\Lambda},\dd) \to (\Z[H_1(L)],0)$ lifts to an augmentation
\[
\tilde{\varepsilon}_L :\thinspace (\SA_{\Lambda},\dd) \to  (\Z[H_1(L) \oplus \Z^{m-1}],0)
\]
defined as follows: for any Reeb chord $a$ of $\Lambda$ ending on component $r(a)$ and beginning on component $c(a)$, we define $\tilde{\varepsilon}_L(a):= u_{r(a)} u_{c(a)}^{-1} \varepsilon_L(a)$, where $u_i = s_i$ for $i\leq m-1$ and $u_m = 1$. 

The augmentation $\tilde{\varepsilon}_L$ to $\Z[H_1(L) \oplus \Z^{m-1}]$ incorporates both the geometry of the filling $L$ and link automorphisms; henceforth we will view it as ``the'' augmentation coming from the filling $L$ and will drop the tilde. We will also not need the distinction between generators of $H_1(L)$ and generators of $\Z^{m-1}$. It is then convenient to recast the augmentation $\varepsilon_L$ in the following definition.

\begin{definition}
A \textit{$k$-system of augmentations} of $\Lambda$ is an algebra map
\[
\varepsilon :\thinspace \SA_\Lambda \to \Z[s_1^{\pm 1},\ldots,s_k^{\pm 1}]
\]
such that $\varepsilon \circ \dd = 0$. By definition, two $k$-systems of augmentations
$$\varepsilon :\thinspace \SA_\Lambda \to \Z[s_1^{\pm 1},\ldots,s_k^{\pm 1}],\quad \varepsilon' :\thinspace \SA_\Lambda \to \Z[(s_1')^{\pm 1},\ldots,(s_k')^{\pm 1}]$$ are considered to be {\it equivalent} if there exists a $\Z$-algebra isomorphism
$$\psi :\thinspace \Z[s_1^{\pm 1},\ldots,s_k^{\pm 1}]\to \Z[(s_1')^{\pm 1},\ldots,(s_k')^{\pm 1}]$$ such that $\varepsilon' = \psi \circ \varepsilon$. Note that the space of such isomorphisms is parametrized by $\Z_2^k \times \GL_k(\Z)$.\hfill$\Box$
\label{def:system}
\end{definition}

Finally, we now recast Theorem~\ref{thm:EHK} for our purposes in the following proposition; note that if $L$ has genus $g$ then $H_1(L) \oplus \Z^{m-1}$ has rank $2g+2m-2$.

\begin{prop}
Let $\Lambda$ be an $m$-component Legendrian link. 
\label{prop:aug-system}
Let $L$ be a connected, orientable exact Lagrangian filling of $\Lambda$ of genus $g$ and Maslov number $0$. Then $L$ gives rise to a $(2g+2m-2)$-system of augmentations of $\Lambda$, and this system is well-defined, independent of choices, up to equivalence. Furthermore, if all Reeb chords of $\Lambda$ have nonnegative degree, then isotopic fillings of $\Lambda$ give rise to equivalent systems of augmentations.\hfill$\Box$
\end{prop}

\subsection{Signs and functoriality of the cobordism map}
\label{ssec:natural}

In order to establish our main results, such as Theorem~\ref{thm:main}, we will apply Proposition~\ref{prop:aug-system} to systems of augmentations that we will explicitly compute for particular fillings. For that, we will divide our fillings into elementary cobordism pieces, calculate the cobordism map for each elementary piece, and compose the resulting cobordism maps, using the fact that the cobordism map is functorial. This functoriality over $\Z$ is established in the work of Karlsson \cite{Karlsson-cob}, and we summarize in this subsection the results from \cite{Karlsson-cob} that we need.

Given an orientable exact Lagrangian cobordism $L$ between $\Lambda_+$ and $\Lambda_-$, we choose a spin structure on $L$ that restricts on each component of $\Lambda_+$ and $\Lambda_-$ to the null-cobordant spin structure. Note that there are $|H_1(\widehat{L};\Z_2)|$ such spin structures, where $\widehat{L}$ is the closed surface obtained from $L$ by gluing a disk to each boundary component, and any such spin structure will do. Besides a spin structure on $L$, the other pieces of auxiliary data that Karlsson uses to define the cobordism maps are systems of capping operators for $\Lambda_+$ and $\Lambda_-$ satisfying certain technical conditions. These capping operators are used by Karlsson to define the signs in the DGAs $(\SA_{\Lambda_+},\dd_+^\NC)$ and $(\SA_{\Lambda_-},\dd_-^\NC)$, where $\dd_\pm^\NC$ are the differentials associated to the null-cobordant spin structures on $\Lambda_\pm$, as well as the signs in the cobordism map between the DGAs.

In our setting, for any Legendrian $\Lambda$ with the Lie group spin structure, a suitable system of capping operators has been constructed in \cite[Section 4.5]{EkholmEtnyreSullivan05c}, compare \cite[Remark 2.9]{Karlsson-cob}. These capping operators give precisely the signs for the DGA differential on $\Lambda$ that we have presented combinatorially in Section~\ref{ssec:dga-def} above and written as $\dd^\comb$, see Remark~\ref{rmk:spin}. However, for the cobordism maps we need the signs from the null-cobordant rather than the Lie group spin structure. As explained in Remark~\ref{rmk:spin}, we can express this combinatorially by choosing a set $\cS$ of marked points on $\Lambda$ with an odd number of marked points on each component, resulting in a differential $\dd^\cS$ on $\SA_\Lambda$ such that we have an isomorphism
\[
\phi^\cS :\thinspace (\SA_\Lambda,\dd^\cS) \stackrel{\cong}{\longrightarrow} (\SA_\Lambda,\dd^\NC).
\]
To return to the setting of a cobordism $L$ between $\Lambda_+$ and $\Lambda_-$, 
Theorem 2.5 in \cite{Karlsson-cob} gives a DGA map over $\Z$, $\Phi_L :\thinspace (\SA_{\Lambda_+},\dd^\NC) \to (\SA_{\Lambda_-},\dd^\NC)$. If we choose sets of marked points $\cS_\pm$ on $\Lambda_\pm$ with an odd number of marked points on each component of $\Lambda_\pm$, then $\Phi_L$ induces a DGA map from $(\SA_{\Lambda_+},\dd^{\cS_+})$ to $(\SA_{\Lambda_-},\dd^{\cS_-})$. We also denote this map by $\Phi_L$, and it satisfies that the following diagram commutes:
\[
\xymatrix{
(\SA_{\Lambda_+},\dd^{\cS_+}) \ar[r]^{\phi^{\cS_+}}_\cong \ar[d]_{\Phi_L} &
(\SA_{\Lambda_+},\dd^\NC) \ar[d]^{\Phi_L} \\
(\SA_{\Lambda_-},\dd^{\cS_-}) \ar[r]^{\phi^{\cS_-}}_\cong &
(\SA_{\Lambda_-},\dd^\NC).
}
\]

Furthermore, the cobordism maps $\Phi_L$ constructed by Karlsson are functorial. To state this property, suppose that $L_1$ and $L_2$ are exact Lagrangian cobordisms that go from $\Lambda_0$ to $\Lambda_1$ and from $\Lambda_1$ to $\Lambda_2$ (from bottom to top), respectively. We can concatenate these to produce an exact cobordism $L_1 \# L_2$ from $\Lambda_0$ to $\Lambda_2$. As before, equip $L_1,L_2$ with spin structures that restrict to the null-cobordant spin structures on their boundaries. Choices of capping operators on $\Lambda_0,\Lambda_1,\Lambda_2$ now produce DGA maps $\Phi_{L_1} :\thinspace (\SA_{\Lambda_1},\dd^\NC) \to (\SA_{\Lambda_0},\dd^\NC)$, $\Phi_{L_2} :\thinspace (\SA_{\Lambda_2},\dd^\NC) \to (\SA_{\Lambda_1},\dd^\NC)$, and $\Phi_{L_1 \# L_2} :\thinspace (\SA_{\Lambda_2},\dd^\NC) \to (\SA_{\Lambda_0},\dd^\NC)$, and \cite[Theorem 2.6]{Karlsson-cob} states that:
$$
\Phi_{L_1} \circ \Phi_{L_2} = \Phi_{L_1 \# L_2}.
$$

Let us choose collections of marked points $\cS_0,\cS_1,\cS_2$ on $\Lambda_0,\Lambda_1,\Lambda_2$ such that each component has an odd number of marked points (as usual). Then, we can use the isomorphisms between $(\SA_{\Lambda_i},\dd^{\cS_i})$ and $(\SA_{\Lambda_i},\dd^\NC)$ to produce DGA maps $\Phi_{L_1},\Phi_{L_2},\Phi_{L_1 \# L_2}$ between the DGAs $(\SA_{\Lambda_i},\dd^{\cS_i})$ such that the following diagram commutes:
\begin{equation}
\begin{split}
\xymatrix{
(\SA_{\Lambda_2},\dd^{\cS_2}) \ar[r]^{\phi^{\cS_2}}_\cong \ar[d]_{\Phi_{L_2}} 
\ar@/_4pc/[dd]_{\Phi_{L_1\# L_2}} &
(\SA_{\Lambda_2},\dd^\NC) \ar[d]^{\Phi_{L_2}}\ar@/^4pc/[dd]^{\Phi_{L_1\# L_2}}  \\
(\SA_{\Lambda_1},\dd^{\cS_1}) \ar[r]^{\phi^{\cS_1}}_\cong \ar[d]_{\Phi_{L_1}} &
(\SA_{\Lambda_1},\dd^\NC) \ar[d]^{\Phi_{L_1}} \\
(\SA_{\Lambda_0},\dd^{\cS_0}) \ar[r]^{\phi^{\cS_0}}_\cong &
(\SA_{\Lambda_0},\dd^\NC).
}
\end{split}
\label{eq:functorial}
\end{equation}
Note that all of the horizontal maps in this diagram are algebra maps over the relevant coefficient ring. Colloquially, they send each homology coefficient $s_i$ to $s_i$, and not to $-s_i$.

This discussion above is summarized in the following result.

\begin{prop}
Given an exact Lagrangian cobordism $L$ between $\Lambda_+$ and $\Lambda_+$, and choices of marked points $\cS_\pm$ on $\Lambda_\pm$ with an odd number on each component, we can write the cobordism map $\Phi_L$ as a DGA map from $(\SA_{\Lambda_+},\dd^{\cS_+})$ to $(\SA_{\Lambda_-},\dd^{\cS_-})$. 
\label{prop:functorial}
If we have exact cobordisms $L_1$ from $\Lambda_0$ to $\Lambda_1$ and $L_2$ from $\Lambda_1$ to $\Lambda_2$, and marked points $\cS_0,\cS_1,\cS_2$ on $\Lambda_0,\Lambda_1,\Lambda_2$ with an odd number on each component, then the cobordism maps for $L_1$, $L_2$, and their concatenation $L_1 \# L_2$ satisfy $\Phi_{L_1} \circ \Phi_{L_2} = \Phi_{L_1 \# L_2}$.\hfill$\Box$
\end{prop}

\subsection{System of augmentations for a decomposable filling}
\label{ssec:system}

All the Lagrangian fillings that we consider in this paper are decomposable in the sense of \cite{EHK} (see Section~\ref{ssec:background}).
For a decomposable filling, one can explicitly construct the corresponding system of augmentations by composing the cobordism maps induced by each of the elementary cobordisms. These elementary cobordism maps are described in Sections~\ref{ssec:isotopy} and~\ref{ssec:EHK-map} below. To combine them into the desired system of augmentations, we additionally need to keep track of base points and discuss how they produce the parameters in the system of augmentations. This is the content of the discussion that now follows.\footnote{We note that a similar treatment of base points on Lagrangian cobordisms and the induced DGA maps (over $\Z_2$) appears in \cite[section 2]{GSW}.}

First, consider a general exact Lagrangian cobordism $L$ between Legendrians $\Lambda_+$ and $\Lambda_-$, inducing a chain map $\Phi_L$ between the DGAs of $\Lambda_+$ and $\Lambda_-$. Recall from Section~\ref{ssec:dga-def} that in the setting of the DGA of a Legendrian $\Lambda$, it is convenient to choose base points on $\Lambda$ and use these points to keep track of the homology classes of the boundaries of the holomorphic disks that contribute to the differential. In a similar manner, we will keep track of homology classes contributing to $\Phi_L$ by placing arcs on $L$ and counting intersections of holomorphic disks with these arcs.

To this end, suppose that we have a collection of oriented arcs and circles on $L$, such that all circles lie in the interior of $L$, the endpoints of all arcs lie on $\Lambda_+ \cup \Lambda_-$, and the arcs are transverse to $\Lambda_+ \cup \Lambda_-$ at their endpoints. Label these arcs $\gamma_1,\ldots,\gamma_k$. Some subset $\{\gamma_{i_1},\ldots,\gamma_{i_p}\}$ has at least one endpoint on $\Lambda_+$, and we view these endpoints as base points on $\Lambda_+$; similarly some subset $\{\gamma_{j_1},\ldots,\gamma_{j_q}\}$ has at least one endpoint on $\Lambda_-$, and we view these endpoints as base points on $\Lambda_-$. The chain map $\Phi_L$ between the DGAs of $\Lambda_+$ and $\Lambda_-$ is defined by counting a finite collection of holomorphic disks with boundary on $L$ and boundary punctures mapping to Reeb chords for $\Lambda_+$ and $\Lambda_-$; we make the (generic) assumption that our curves $\gamma_i$ intersect the boundaries of these disks transversely, and that no endpoint of an arc $\gamma_i$ lies at the endpoint of a Reeb chord of $\Lambda_+$ or $\Lambda_-$.

In this setting, $\Phi_L$ is a map of algebras over the coefficient ring $\Z[s_1^{\pm 1},\ldots,s_k^{\pm 1}]$. More precisely, the DGA for $\Lambda_+$ equipped with the base points from $\gamma_{i_1},\ldots,\gamma_{i_p}$ has coefficient ring $\Z[s_{i_1}^{\pm 1},\ldots,s_{i_p}^{\pm 1}]$, and we can tensor this DGA over $\Z[s_{i_1}^{\pm 1},\ldots,s_{i_p}^{\pm 1}]$ with $\Z[s_1^{\pm 1},\ldots,s_k^{\pm 1}]$ to obtain a DGA over $\Z[s_1^{\pm 1},\ldots,s_k^{\pm 1}]$, which we write as $(\SA_{\Lambda_+},\dd_+)$. Similarly we can define the DGA $(\SA_{\Lambda_-},\dd_-)$ over $\Z[s_1^{\pm 1},\ldots,s_k^{\pm 1}]$.
Then we can define the chain map $\Phi_L :\thinspace \SA_{\Lambda_+} \to \SA_{\Lambda_-}$ as a map of $\Z[s_1^{\pm 1},\ldots,s_k^{\pm 1}]$-algebras: each holomorphic disk $\Delta$ contributing to $\Phi_L$ is given the coefficient $s_1^{n_1(\Delta)}\cdots s_k^{n_k(\Delta)} \in \Z[s_1^{\pm 1},\ldots,s_k^{\pm 1}]$, where $n_i(\Delta)$ counts the number of signed intersections of $\partial \Delta$ with the curve $\gamma_i$.

We now apply this discussion to describe how to concretely construct a system of augmentations for a Legendrian link $\Lambda$ associated to a connected, decomposable exact Lagrangian filling $L$ of $\Lambda$. Let $m$ denote the number of components of $\Lambda$ and $g$ the genus of $L$. By assumption, $L$ is a union of $0$-handles (minimum cobordisms) and $1$-handles (saddle cobordisms); let $k$ denote the number of $0$-handles, and note that it follows that there are $2g+m+k-2=: \ell$ $1$-handles. We can cut off a small neighborhood of each minimum of $L$ to produce a new cobordism $L'$ whose top end is $\Lambda$ and whose bottom end is a $k$-component unlink $\Lambda_0$, such that $L'$ is assembled out of just the $1$-handles of $L$.

We can view $L'$ through slices from top to bottom, so that it becomes a movie of embedded Legendrian links (except at finitely many times) starting with $\Lambda$, at the top, and ending with the $k$-component unlink $\Lambda_0$, at the bottom. In the Lagrangian projection, each saddle move is then represented by replacing a (contractible) crossing by its $0$-resolution. We can now add base points to this movie as follows. Place base points $t_1,\ldots,t_m$ on the $m$ components of $\Lambda$. Each time we pass through a saddle, add two more base points labeled $s_i$ and $-s_i^{-1}$. All base points persist to the bottom of the cobordism, $\Lambda_0$. See Figure~\ref{fig:decomp-filling}.

\begin{center}
	\begin{figure}[h!]
		\centering
						\includegraphics[scale=1]{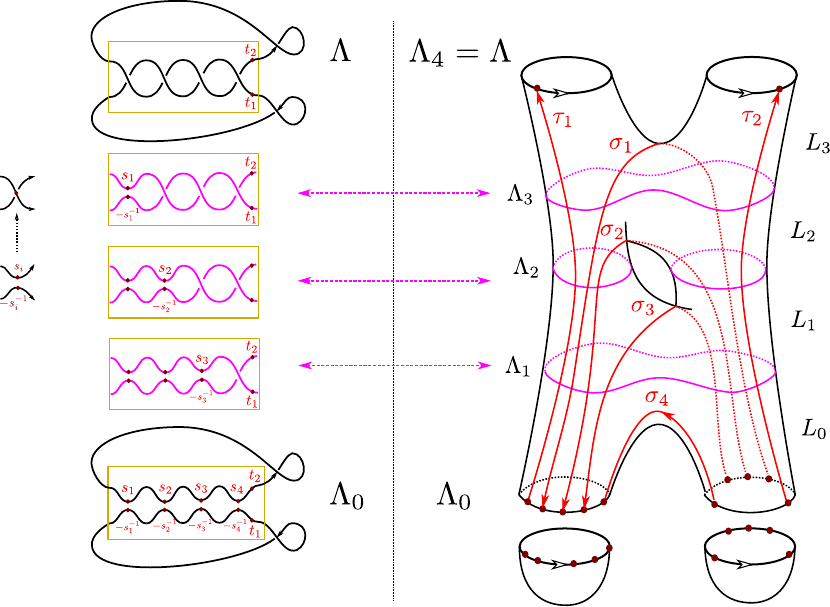}
		\caption{
		On the left, placing a pair of base points at the bottom of a saddle cobordism, representing opposite sides of an arc passing through the saddle point in the cobordism. On the right,
dividing a decomposable filling of $\Lambda$ into elementary pieces: from top to bottom, a sequence of saddle cobordisms ending at an unlink $\Lambda_0$, and then filling in each unknot component. 
}
		\label{fig:decomp-filling}
	\end{figure}
\end{center}

On the Lagrangian cobordism $L'$, the base points $t_1,\ldots,t_m$ trace out arcs joining $\Lambda$ to $\Lambda_0$, while for each $i=1,\ldots,\ell$, the base points $s_i,-s_i^{-1}$ together trace out an arc joining $\Lambda_0$ to itself. We call these arcs $\tau_1,\ldots,\tau_m$ and $\sigma_1,\ldots,\sigma_\ell$, respectively. Orient the $\tau_i$ arcs upwards, and orient the $\sigma_i$ arcs so that in each slice the arc is oriented upwards at the point labeled $s_i$ and downwards at the point labeled $-s_i^{-1}$. This places the decomposable cobordism between $\Lambda_0$ and $\Lambda$ in the general picture described above of a cobordism equipped with oriented arcs.

Label the slices of $L'$ from bottom to top by $\Lambda_0,\Lambda_1,\ldots,\Lambda_{\ell}=\Lambda$, and divide $L'$ into saddle cobordisms $L_1,\ldots,L_{\ell}$, where $L_j$ is the piece of $L'$ between $\Lambda_{j-1}$ and $\Lambda_j$; note that the saddle of $L_j$ is associated to the arc $\sigma_{\ell+1-j}$. Each $\Lambda_j$ is equipped with a collection of base points each labeled by either $t_i$ or $\pm s_i^{\pm 1}$. For $j=1,\ldots,\ell$, let $(\SA_{\Lambda_j},\dd^\comb)$ denote the DGA of $\Lambda_j$ over $\Z[t_1^{\pm 1},\ldots,t_m^{\pm 1},s_1^{\pm 1},\ldots,s_\ell^{\pm 1}]$ with the differential $\dd^\comb$ defined combinatorially as in Section~\ref{ssec:dga-def} (note that some of the $s_i$ parameters may not correspond to base points of $\Lambda_j$ and thus may not appear in the definition of $\dd^\comb$).

We next relate the DGAs $(\SA_{\Lambda_j},\dd^\comb)$ to the discussion from Section~\ref{ssec:natural}. To this end, for each $j=0,\ldots,\ell$, we identify a subset $\cS_j$ of the base points on $\Lambda_j$ such that each component of $\Lambda_j$ contains an odd number of points in $\cS_j$; we abbreviate this condition by calling such a subset \textit{odd-cardinality}. We define $\cS_j$ by backwards induction on $j$. Let $\cS_\ell$ be the collection of all of the base points $t_1,\ldots,t_m$ on $\Lambda_\ell=\Lambda$, and note that this is odd-cardinality. Given $\cS_j$, each base point on $\Lambda_j$ descends to a corresponding base point on $\Lambda_{j-1}$, and so we may view $\cS_j$ as a collection of base points on $\Lambda_{j-1}$. On $\Lambda_{j-1}$, we can add to $\cS_j$ one more base point, from the two new base points labeled by $\pm s_{\ell+1-j}^{\pm 1}$, such that the resulting collection $\cS_{j-1}$ is odd-cardinality: if $\Lambda_{j-1}$ has one more component than $\Lambda_j$, then the choice of this extra base point is forced by the odd-cardinality condition, while if $\Lambda_{j-1}$ has one fewer component than $\Lambda_j$, then we can choose either.

The choice of base points $\cS_j$ on $\Lambda_j$ produces a differential $\dd^{\cS_j}$ on $\SA_{\Lambda_j}$ as follows: first remove the $-$ signs at the front of any base points on $\Lambda_j$ labeled by $-s_i^{-1}$, so that all base points are labeled by $t_i$ or $s_i^{\pm 1}$; then negate any base point in $\cS_j$, and let $\dd^{\cS_j}$ be the resulting combinatorial differential as in Remark~\ref{rmk:spin}. Note that each $t_i$ is negated in this process, while exactly one of $s_i$ or $s_i^{-1}$ is negated, depending on which of these base points is in $\cS_j$. Thus we can define a $\Z$-algebra isomorphism
\[
\psi :\thinspace \Z[t_1^{\pm 1},\ldots,t_m^{\pm 1},s_1^{\pm 1},\ldots,s_\ell^{\pm 1}] \to
\Z[t_1^{\pm 1},\ldots,t_m^{\pm 1},s_1^{\pm 1},\ldots,s_\ell^{\pm 1}]
\]
by $\psi(t_i) = -t_i$ for all $i=1,\ldots,m$ and $\psi(s_i) = \pm s_i$ for each $i=1,\ldots,\ell$ (with the sign determined by whether $s_i$ or $s_i^{-1}$ is in $\cS_0$), which extends to a map $\psi :\thinspace \SA_{\Lambda_j} \to \SA_{\Lambda_j}$ by specifying $\psi(a)=a$ for all Reeb chords $a$. This map $\psi$ now intertwines the differentials $\dd^\comb$ and $\dd^{\cS_j}$:
\[
\psi :\thinspace (\SA_{\Lambda_j},\dd^\comb) \stackrel{\cong}{\longrightarrow} (\SA_{\Lambda_j},\dd^{\cS_j}).
\]
We can combine this with the isomorphism $\phi^{\cS_j} :\thinspace (\SA_{\Lambda_j},\dd^{\cS_j}) \stackrel{\cong}{\longrightarrow} (\SA_{\Lambda_j},\dd^\NC)$ from Section~\ref{ssec:natural} to obtain an isomorphism $\phi^{\cS_j} \circ \psi$ from $(\SA_{\Lambda_j},\dd^\comb)$ to the DGA $(\SA_{\Lambda_j},\dd^\NC)$ with the null-cobordant spin structure.

Recall from Section~\ref{ssec:natural} that since each $\cS_j$ is odd-cardinality, each cobordism $L_j$ induces a cobordism map $\Phi_{L_j} :\thinspace (\SA_{\Lambda_j},\dd^{\cS_j}) \to (\SA_{\Lambda_{j-1}},\dd^{\cS_{j-1}})$. By combining this with the isomorphism $\psi$, we can view the cobordism map as a map
$(\SA_{\Lambda_j},\dd^\comb) \to (\SA_{\Lambda_{j-1}},\dd^\comb)$, which we also write as $\Phi_{L_j}$, so that the following diagram commutes:
\[
\xymatrix{
(\SA_{\Lambda_j},\dd^\comb) \ar[r]^{\psi}_{\cong} \ar[d]_{\Phi_{L_j}} & (\SA_{\Lambda_j},\dd^{\cS_j}) \ar[d]^{\Phi_{L_j}} \\
(\SA_{\Lambda_{j-1}},\dd^\comb) \ar[r]^{\psi}_{\cong} & (\SA_{\Lambda_{j-1}},\dd^{\cS_{j-1}}).
}
\]
Similarly, we can view the cobordism map $\Phi_{L'}$ as a DGA map $(\SA_{\Lambda},\dd^\comb) \to (\SA_{\Lambda_0},\dd^\comb)$. By the functoriality property from Proposition~\ref{prop:functorial}, we have
\[
\Phi_{L'} = \Phi_{L_1} \circ \cdots \circ \Phi_{L_\ell} :\thinspace (\SA_{\Lambda},\dd^\comb) \to (\SA_{\Lambda_0},\dd^\comb).
\]

We obtain the filling $L$ of $\Lambda$ from the cobordism $L'$ by filling in the $k$ components of the unlink $\Lambda_0$ with disjoint Lagrangian disks. Each disk filling produces a unique augmentation, as we record in the following statement.

\begin{prop}
Let $U$ denote the standard Legendrian unknot with a collection of base points with labels $l_1,\ldots,l_p$ (where typically each label is of the form $\pm s_i^{\pm 1}$ or $\pm t_i^{\pm 1}$). If $l_1\cdots l_p = -1$ then the DGA $(\SA_U,\dd_U)$ has a unique augmentation.
\label{prop:unknot}
\end{prop}

\begin{proof}
Let $a$ denote the Reeb chord of $U$.  If $l_1,\ldots,l_q$ are the base points on one lobe of the figure eight in $\Pi_{xy}(U)$ and $l_{q+1},\ldots,l_p$ are the base points on the other, then $\dd_U(a) = l_1\cdots l_q + l_p^{-1}\cdots l_{q+1}^{-1}$. The condition for $\varepsilon$ to be an augmentation is that $\varepsilon(\dd_U(a)) = 0$, in which case $\varepsilon$ is uniquely determined since $\varepsilon(a) = 0$ for grading reasons.
\end{proof}

Now let $w_1,\ldots,w_k$ denote the product of the labels of the base points on each of the $k$ components of the unlink $\Lambda_0$, and write $R$ for the ring $$R=(\Z[t_1^{\pm 1},\ldots,t_m^{\pm 1},s_1^{\pm 1},\ldots,s_{\ell}^{\pm 1}])/(w_1=\cdots=w_k=-1).$$
Then the filling of $\Lambda_0$ by disks yields an augmentation
\[
\varepsilon_0 :\thinspace \SA_{\Lambda_0} \to R.
\]
Composing with $\Phi_{L'}$ now gives the augmentation of $\Lambda$ induced by $L$:
\[
\Phi_L = \varepsilon_0 \circ \Phi_{L'} :\thinspace \SA_{\Lambda} \to R.
\]

We will already call this $\Phi_L$ the \textit{combinatorial} system of augmentations of $\Lambda$ induced by $L$, even though it will not become fully combinatorial until we present the combinatorial cobordism maps for isotopy cylinders and saddle cobordisms in Section~\ref{sec:elementary}. This is to temporarily distinguish $\Phi_L$ from the \textit{geometric} system of augmentations of $\Lambda$ from Proposition~\ref{prop:aug-system}. In fact, the two systems agree up to equivalence, as we will show next.

\subsection{The systems of augmentations agree}
\label{ssec:agree}

In this subsection, we prove that the combinatorial and geometric systems of augmentations of a decomposable filling $L$ are equivalent. This result generalizes a result of Y. Pan from \cite[section 3]{Pan-fillings}, which uses $\Z_2$ coefficients and treats the case where $\Lambda$ has a single component. We use the same notation as in the previous subsection: $\Phi_L$ is a map from $\SA_{\Lambda}$ to $R$, where $R$ is the ring $R=(\Z[t_1^{\pm 1},\ldots,t_m^{\pm 1},s_1^{\pm 1},\ldots,s_{\ell}^{\pm 1}])/(w_1=\cdots=w_k=-1)$ with $w_1,\ldots,w_k$ being words associated to the $k$ minima of $L$. The desired equivalence is shown in the following result, which will be proven momentarily:

\begin{prop}
Suppose that the filling $L$ of $\Lambda$ is connected.
Then we have $R \cong \Z[\Z^{2g+2m-2}]$ and consequently $\Phi_L$ is a $(2g+2m-2)$-system of augmentations of $\Lambda$. Furthermore, up to equivalence,
$\Phi_L$ agrees with the geometric system of augmentations from Proposition~\ref{prop:aug-system}.
\label{prop:decomp-system}
\end{prop}

The crucial consequence of Proposition~\ref{prop:decomp-system} is that since geometric systems of augmentations are invariant under Hamiltonian isotopy of the filling, the same is true of the combinatorial system of augmentations $\Phi_L$. This is the fact that will allow us to distinguish fillings through a combinatorial calculation of their augmentations. Indeed, the following result is a direct consequence of Propositions~\ref{prop:aug-system} and~\ref{prop:decomp-system}:

\begin{prop}
	Let $L$ be a connected filling of $\Lambda$, and suppose that all Reeb chords of $\Lambda$ have nonnegative degree. Then the combinatorial system of augmentations $\Phi_L$ of $\SA_\Lambda$ is invariant, up to equivalence, under exact Lagrangian isotopy of $L$.\hfill$\Box$
	\label{prop:invariant}
\end{prop}

The argument for Proposition \ref{prop:decomp-system} above occupies the remainder of this section.

\begin{proof}[Proof of Proposition \ref{prop:decomp-system}]
By functoriality, $\Phi_L$ and the system of augmentations from Proposition~\ref{prop:aug-system} agree over $\Z$. What we need to do is keep track of the homology coefficients that appear in the definitions of the two families of augmentations, and show that the two agree up to equivalence. Thus, we reduce mod $2$ and work with group rings over $\Z_2$. In the course of tracking the homology coefficients, we will see that the abelian group generated multiplicatively by $t_1,\ldots,t_m,s_1,\ldots,s_\ell$ with relations $w_1=\cdots=w_k=1$ is isomorphic to a free abelian group with $2g+2m-2$ generators, whence it will follow that $R \cong \Z_2[\Z^{2g+2m-2}]$.

As in Section~\ref{ssec:system}, let $\tau_i$ and $\sigma_i$ denote the oriented arcs on $L'$ corresponding to $t_i$ and $s_i$. The map $\Phi_L$ counts intersections with $\tau_i$ and $\sigma_i$; what we will show is that these counts keep track of homology classes in $H_1(L)$ along with link automorphisms.
If $a$ is a degree-$0$ Reeb chord of $\Lambda$, let $\cM(a)$ denote the moduli space of (rigid) holomorphic disks $\Delta$ with boundary on $L$ and a single positive boundary puncture mapping to $a$. We may assume that $L$ is generic, so that none of the minima of $L$ lies on the boundary of a holomorphic disk in any of the $\cM(a)$. Recall that $L'$ is obtained from $L$ by removing a neighborhood of each minimum of $L$. By making these neighborhoods sufficiently small, we may assume that the boundary of each of the holomorphic disks $\Delta\in\cM(a)$ lies entirely in $L'$ and does not intersect the negative boundary $\Lambda_0$ of $L'$: that is, $\partial \Delta$ is an oriented arc on $L'$ with endpoints at the endpoints of $a$.

The cobordism map $\Phi_{L'}$ is then given as follows, for all degree $0$ Reeb chords $a$ of $\Lambda$:
\[
\Phi_{L'}(a) = \sum_{\Delta\in\cM(a)} w(\Delta) \in \Z_2[t_1^{\pm 1},\ldots,t_m^{\pm 1},s_1^{\pm 1},\ldots,s_{\ell}^{\pm 1}],
\]
where 
\begin{equation}
w(\Delta) = \prod_{i=1}^m t_i^{\#(\partial\Delta\cap \tau_i)} \prod_{i=1}^\ell s_i^{\#(\partial\Delta\cap \sigma_i)}.
\label{eq:word}
\end{equation}
By the discussion preceding the proposition, the augmentation $\Phi_L$ is the composition of $\Phi_{L'}$ with the quotient map 
$$\Z_2[t_1^{\pm 1},\ldots,t_m^{\pm 1},s_1^{\pm 1},\ldots,s_{\ell}^{\pm 1}] \to \Z_2[t_1^{\pm 1},\ldots,t_m^{\pm 1},s_1^{\pm 1},\ldots,s_{\ell}^{\pm 1}]/(w_1=\cdots=w_k=1).$$

We want to compare $\Phi_L$ with the geometric setup from Section~\ref{ssec:geom-cob}. Recall from Theorem~\ref{thm:EHK} that $L$ induces an augmentation $\varepsilon_L :\thinspace \SA_{\Lambda} \to \Z_2[H_1(L)]$. This map agrees over $\Z_2$ with $\Phi_L$ but the group-ring coefficients are given by:
\[
\varepsilon_L(a) = \sum_{\Delta\in\cM(a)} \exp([\overline{\dd\Delta}]) \in \Z_2[H_1(L)].
\]
The notation here is as follows. Choose a capping path $\gamma_a$ for each Reeb chord $a$ of $\Lambda$: a path in the connected surface $L$ whose endpoints are the same as the endpoints of $\gamma_a$. For each disk $\Delta \in \cM(a)$, close up the arc $\partial \Delta$ by adding the reverse of $\gamma_a$ to give a closed loop $\overline{\partial\Delta} = (\partial\Delta) \cup (-\gamma_a)$. Then $\overline{\partial\Delta}$ represents a homology class in $H_1(L)$, and we denote this class in $\Z_2[H_1(L)]$ by $\exp([\overline{\dd\Delta}])$ (the exponential changes addition to multiplication).

We specify particular capping paths $\gamma_a$ as follows. For $i=0,\ldots,\ell-1$, let $L_{>i}:= L_{i+1} \cup \cdots \cup L_{\ell}$ denote the portion of $L$ above $\Lambda_i$, and $L_{>\ell}:= \Lambda$. Note that $L_{>\ell}$ has $m$ components while $L_{>0}$ has $1$ component, and there are exactly $m-1$ values of $i$ for which $L_{>(i-1)}$ has $1$ fewer component than $L_{>i}$. For notational simplicity we will assume that these are the largest possible values: $i=\ell-m+2,\ldots,\ell$. (A similar argument holds in general.) In this case the first $m-1$ saddle moves from the top are all cobordisms that merge components. The arcs $\sigma_1,\ldots,\sigma_{m-1}$ are the cores of these $1$-handle attachments, and we write $\sigma_1^\vee,\ldots,\sigma_{m-1}^\vee$ for the corresponding cocores. (More explicitly, begin at the $i$-th saddle, place one point on each strand of the crossing above this saddle, and trace this pair of points upwards through the cobordism to $\Lambda$ to produce $\sigma_i^\vee$.) The paths $\sigma_1^\vee,\ldots,\sigma_{m-1}^\vee$ join the $m$ components of $\Lambda$ to each other. For each Reeb chord $a$ of $\Lambda$, we can now choose the capping path $\gamma_a$ to lie on $\Lambda \cup \sigma_1^\vee \cup \cdots \cup \sigma_{m-1}^\vee$ and to avoid the base points $t_1,\ldots,t_\ell$ on $\Lambda$. By construction, among the arcs $\tau_1,\ldots,\tau_m,\sigma_1,\ldots,\sigma_\ell$, the only ones that $\gamma_a$ intersects are some subset of $\sigma_1,\ldots,\sigma_{m-1}$ determined by which components of $\Lambda$ contain the endpoints of $a$.

Since the arcs $\sigma_1^\vee,\ldots,\sigma_{m-1}^\vee$ form a tree connecting the components of $\Lambda$, we can find units $u_1,\ldots,u_m \in \Z_2[s_1^{\pm 1},\ldots,s_{m-1}^{\pm 1}]$ such that for each $i=1,\ldots,m-1$, if $\sigma_i^\vee$ ends on component $r(i)$ and begins on component $c(i)$ of $\Lambda$, then $s_i = u_{r(i)} u_{c(i)}^{-1}$; furthermore, $(u_1,\ldots,u_m)$ are well-defined once we specify $u_m=1$, and the induced map $\Z_2[s_1^{\pm 1},\ldots,s_{m-1}^{\pm 1}] \to \Z_2[u_1^{\pm 1},\ldots,u_{m-1}^{\pm 1}]$ is an isomorphism. 
(Concretely, for any $i$ there is a unique path from component $m$ to component $i$ that traverses $\pm \sigma_{j_1}^\vee, \ldots, \pm \sigma_{j_k}^\vee$, where the $\pm$ signs denote orientation, and then $u_i$ is given by $s_{j_1}^{\pm 1} \cdots s_{j_k}^{\pm 1}$.)
It now follows by the construction of the capping paths $\gamma_a$ that if $a$ is any Reeb chord of $\Lambda$ and $r(a),c(a) \in \{1,\ldots,m\}$ are the components of the ending and beginning points of $a$, then
\[
\prod_{i=1}^{m-1} s_i^{\#(\gamma_a \cap \sigma_i)} = u_{r(a)} u_{c(a)}^{-1}.
\]

Suppose for now that $L$ has exactly one minimum; the general case will be considered afterward. Then $\Lambda_0$ is a single-component unknot $U$, and the product of the labels of the base points on $\Lambda_0$ is $t_1\cdots t_m$ since the $s_i$ base points cancel in pairs. Note that the abelian group generated by $t_1,\ldots,t_m,s_1,\ldots,s_\ell$ with a single relation $t_1\cdots t_m=1$ is free on $m+\ell-1=2g+2m-2$ generators; $\Phi_L$ maps to the ring $\Z_2[t_1^{\pm 1},\ldots,t_m^{\pm 1},s_1^{\pm 1},\ldots,s_{\ell}^{\pm 1}]/(t_1\cdots t_m = 1) \cong \Z_2[\Z^{2g+2m-2}]$. 

When $L$ has one minimum, the relative homology $H_1(L,\Lambda)$ is generated by $\sigma_m,\ldots,\sigma_\ell$ and $\tau_2-\tau_1,\ldots,\tau_m-\tau_1$. (Strictly speaking all of these arcs end on $\Lambda_0$; we extend these arcs by adding arcs in the disk filling $\Lambda_0$, so that any endpoint on $\Lambda_0$ is replaced by an endpoint at the minimum of $L$.) Since $H_1(L)$ is dual to $H_1(L,\Lambda)$, we can compute the homology class $[\overline{\partial\Delta}]$ for $\Delta\in\cM(a)$ by counting intersections with the generating set of $H_1(L,\Lambda)$. To be precise, we can identify $\Z_2[H_1(L)] \cong \Z_2[t_2^{\pm 1},\ldots,t_m^{\pm 1},s_m^{\pm 1},\ldots,s_\ell^{\pm 1}]$, and under this isomorphism we have
\[
\exp [\overline{\partial\Delta}] = \prod_{i=2}^m t_i^{\#(\overline{\partial\Delta}\cap\tau_i)-\#(\overline{\partial\Delta}\cap\tau_1)} \prod_{i=m}^\ell s_i^{\#(\overline{\partial\Delta}\cap\sigma_i)} =
(t_2\cdots t_m)^{-\#(\partial\Delta\cap\tau_1)} \prod_{i=2}^m t_i^{\#(\partial\Delta\cap\tau_i)} \prod_{i=m}^\ell s_i^{\#(\partial\Delta\cap\sigma_i)}.
\]

We now compare this to the formula for $w(\Delta)$ in equation~\eqref{eq:word}:
\begin{align*}
w(\Delta)|_{t_1=(t_2\cdots t_m)^{-1}} &= \left.\left(\prod_{i=1}^m t_i^{\#(\partial\Delta\cap\tau_i)}\right)\right|_{t_1=(t_2\cdots t_m)^{-1}} \prod_{i=1}^{m-1} s_i^{\#(\partial\Delta\cap\sigma_i)}\prod_{i=m}^\ell s_i^{\#(\partial\Delta\cap\sigma_i)} \\
&= (t_2\cdots t_m)^{-\#(\partial\Delta\cap\tau_1)} \prod_{i=2}^m t_i^{\#(\partial\Delta\cap\tau_i)}
\prod_{i=1}^{m-1} s_i^{\#(\gamma_a\cap\sigma_i)}\prod_{i=m}^\ell s_i^{\#(\partial\Delta\cap\sigma_i)} \\
&= u_{r(a)} u_{c(a)}^{-1} \exp [\overline{\partial\Delta}],
\end{align*}
where in the second equality we use the fact that $\partial\Delta$ and $\gamma_a$ have the same endpoints and $\sigma_i$ is a separating curve in $L$. Now, we extend $$\varepsilon_L :\thinspace \SA_\Lambda \to \Z_2[H_1(L)] \cong \Z_2[t_2^{\pm 1},\ldots,t_m^{\pm 1},s_m^{\pm 1},\ldots,s_\ell^{\pm 1}]$$ by a link automorphism to 
$$\tilde{\varepsilon}_L :\thinspace \SA_\Lambda \to \Z_2[H_1(L) \oplus \Z^{m-1}] \cong \Z_2[t_2^{\pm 1},\ldots,t_m^{\pm 1},s_m^{\pm 1},\ldots,s_\ell^{\pm 1},u_1^{\pm 1},\ldots,u_{m-1}^{\pm 1}]$$ defined by
$\tilde{\varepsilon}_L(a):= u_{r(a)} u_{c(a)}^{-1} \varepsilon_L(a)$, as in Section~\ref{ssec:geom-cob}. Then, we have
\[
\tilde{\varepsilon}_L(a) = \sum_{\Delta\in\cM(a)} u_{r(a)} u_{c(a)}^{-1} \exp([\overline{\partial\Delta}])
= \sum_{\Delta\in\cM(a)} w(\Delta)|_{t_1=(t_2\cdots t_m)^{-1}}.
\]
That is, the following diagram commutes:
\[
\xymatrix{
\SA_\Lambda \ar[rr]^<<<<<<<<<<{\tilde{\varepsilon}_L} \ar@/_1pc/[drr]_<<<<<<<{\Phi_L} && \Z_2[t_2^{\pm 1},\ldots,t_m^{\pm 1},s_m^{\pm 1},\ldots,s_\ell^{\pm 1},u_1^{\pm 1},\ldots,u_{m-1}^{\pm 1}] \ar[d]^\cong \\
&& \Z_2[t_1^{\pm 2},\ldots,t_m^{\pm 1},s_1^{\pm 1},\ldots,s_{\ell}^{\pm 1}]/(t_1\cdots t_m=1).
}
\]
This shows that the combinatorial cobordism map $\Phi_L$ and the geometric cobordism map $\tilde{\varepsilon}_L$ agree up to isomorphism when $L$ has one minimum.

Now suppose that $L$ has $k>1$ minima. We claim that we can reduce to the above case of a single minimum. The arcs $\sigma_i$ have endpoints at the minima; since $L$ is connected, there is a spanning tree of $k-1$ arcs that connects all of the minima to each other. For notational simplicity, we assume that these arcs are $\sigma_{\ell-k+2},\ldots,\sigma_\ell$. Now imagine deforming $L$ by homotopy equivalence by successively contracting each arc $\sigma_{\ell-k+2},\ldots,\sigma_\ell$ to a point. The result is a new surface $\widetilde{L}$ with a single minimum, which inherits the arcs $\tau_1,\ldots,\tau_m$ and $\sigma_i$, $i\leq \ell-k+1$. The geometric cobordism map $\varepsilon_L$ is defined homologically and does not change when we replace $L$ by $\widetilde{L}$.

We now examine what happens to the cobordism map $\Phi_L$ as we pass from $L$ to $\widetilde{L}$. Recall that $\Phi_L$ is an augmentation taking values in the ring 
$$\Z_2[t_1^{\pm 1},\ldots,t_m^{\pm 1},s_1^{\pm 1},\ldots,s_{\ell}^{\pm 1}]/(w_1=\cdots=w_k=1),$$ where for $j=1,\ldots,k$, $w_j$ is the word given by the product of the arcs having an endpoint at the $j$-th minimum (each endpoint contributes $t_i^{\pm 1}$ or $s_i^{\pm 1}$ depending on the orientation of the corresponding arc at the minimum). At the step where we contract $\sigma_i$, note that $s_i$ appears in exactly two words $w_{i_1}$ and $w_{i_2}$ corresponding to the endpoints of $\sigma_i$. We use the relation for one of these words, $w_{i_1}=1$, to solve for $s_i$, and substitute into $w_{i_2}=1$; the result is exactly the relation corresponding to the new minimum given by contracting $\sigma_i$. Once we have contracted all of $\sigma_{\ell-k+2},\ldots,\sigma_\ell$, we are left with a single word $w$ for the unique remaining minimum, and this process gives an isomorphism between the coefficient ring $\Z_2[t_1^{\pm 1},\ldots,t_m^{\pm 1},s_1^{\pm 1},\ldots,s_{\ell}^{\pm 1}]/(w_1=\cdots=w_k=1)$ for $L$ and the ring $\Z_2[t_1^{\pm 1},\ldots,t_m^{\pm 1},s_1^{\pm 1},\ldots,s_{\ell-k+1}^{\pm 1}]/(w=1)$ for $\widetilde{L}$. In particular, note that the abelian group generated by $t_1,\ldots,t_m,s_1,\ldots,s_\ell$ with relations $w_1=\cdots=w_k=1$ is again free on $2g+2m-2$ generators, just as in the case where $L$ has one minimum.

\begin{center}
	\begin{figure}[h!]
		\centering
								\includegraphics[scale=0.8]{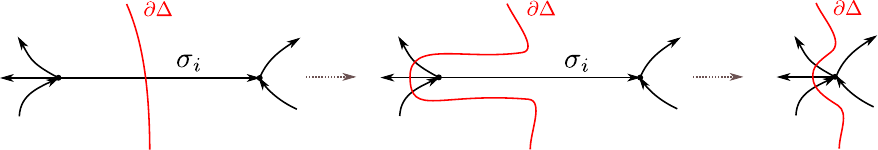}
		\caption{
Sliding the arc $\partial\Delta$ to avoid intersections with $\sigma_i$, and then contracting $\sigma_i$.
}
		\label{fig:perturb}
	\end{figure}
\end{center}

Now in $\widetilde{L}$, the boundaries $\partial\Delta$ of some holomorphic disks may pass through the minimum. To restore transversality, we perturb each $\partial\Delta$ as follows: at the step where we contract $\sigma_i$, we homotop $\partial\Delta$ near any intersection with $\sigma_i$ so that it wraps around one of the endpoints of $\sigma_i$ instead; see Figure~\ref{fig:perturb}. This removes any intersections of $\partial\Delta$ with $\sigma_i$, and it does not change the word $w(\partial\Delta)$ as given in \eqref{eq:word} because of the relations $w_j=1$. The end result is the surface $\widetilde{L}$ where all boundaries $\partial\Delta$ are disjoint from the minimum of $\widetilde{L}$, and we have reduced to the case of a single minimum. This completes the proof.
\end{proof}

\begin{remark}
\label{rmk:Caitlin}
The above proof shows that the augmentation/cobordism map $\Phi_L:\thinspace \SA_\Lambda \to \Z_2[t_1^{\pm 1},\ldots,t_m^{\pm 1},s_1^{\pm 1},\ldots,s_\ell^{\pm 1}]/(w_1=\cdots=w_k=1)$ sends the product $t_1\cdots t_m$ to $1$, since the product $w_1\cdots w_k$ is equal to $t_1\cdots t_m$: each $\sigma$ arc contributes endpoints that cancel, and each $\tau$ arc ends at exactly one of the minima.

We can lift this statement to $\Z$ coefficients: if $\SA_\Lambda$ is the DGA of $\Lambda$ with the Lie group spin structure, then $\Phi_L :\thinspace \SA_\Lambda \to \Z[t_1^{\pm 1},\ldots,t_m^{\pm 1},s_1^{\pm 1},\ldots,s_\ell^{\pm 1}]/(w_1=\cdots=w_k=1)$ sends $t_1\cdots t_m$ to $(-1)^m$. This follows from a result of Leverson \cite{Leverson-PJM} that any augmentation of $\SA_\Lambda$ to a field (whether or not it comes from a filling) must send $t_1\cdots t_m$ to $(-1)^m$, whence this must be true of $\Phi_L$.\hfill$\Box$
\end{remark}


\section{Cobordism Maps for Elementary Cobordisms}
\label{sec:elementary}

Given a decomposable Lagrangian filling $L$ of a Legendrian link $\Lambda$, we have described in Section~\ref{sec:cobordism} the general theory of how to build a system of augmentations for $L$. In order to apply this theory, we will use combinatorial formulas for cobordism maps corresponding to elementary cobordisms, which we can then compose to produce a formula for the cobordism map of an arbitrary decomposable filling. Of the three elementary cobordisms in Section \ref{ssec:background}, we have already discussed the DGA map for a minimum cobordism; see Proposition~\ref{prop:unknot}. In this section we present combinatorial formulas for the cobordism maps for the other two elementary cobordisms: isotopy cobordisms and saddle cobordisms. 

The map for an isotopy cobordism (Section~\ref{ssec:isotopy}) is not new and dates back originally to work of K\'alm\'an \cite{Kalman}. The map for a saddle cobordism (Section~\ref{ssec:EHK-map}) occupies the bulk of Section~\ref{sec:elementary}, with some technical details postponed to Appendix~\ref{sec:saddle-map}. It builds on work of Ekholm--Honda--K\'alm\'an \cite{EHK}, but introduces two new features:

\begin{itemize}
\item[1.] A combinatorial lift to integer coefficients $\Z$,
\item[2.] A formula that (even) over $\Z_2$ works for some saddle cobordisms (where the combinatorial EHK map over $\Z_2$ does not).
\end{itemize}
In order to lift the saddle cobordism map to $\Z$, rather than directly constructing explicit orientations of the relevant moduli spaces, we use an ad hoc argument that allows us to deduce signs for a particularly simple saddle cobordism from the fact, due to work of Karlsson \cite{Karlsson-cob}, that the map must be a chain map over $\Z$. In fact we conclude a slightly weaker result: namely, we show that the cobordism map agrees with our combinatorial formula up to a link automorphism. Nevertheless, this additional choice of link automorphism will not affect our computations, and the statement we obtain is sufficient for the purposes of calculating augmentations for fillings. This is explained in Section~\ref{ssec:linkaug-filling}.

\subsection{The cobordism map for a Legendrian isotopy}

\label{ssec:isotopy}

In this subsection we review the cobordism map for an isotopy cobordism. Suppose that $\Lambda_+$ and $\Lambda_-$ are Legendrian links related by a Legendrian isotopy. There is then a quasi-isomorphism between the DGAs $(\SA_{\Lambda_+},\dd)$ and $(\SA_{\Lambda_-},\dd)$, as first constructed by Chekanov \cite{Chekanov} over $\Z_2$ and then lifted to $\Z$ in \cite{ENS}. More precisely, these quasi-isomorphisms are DGA maps that are constructed for certain elementary Legendrian isotopies, to be described below. Any general Legendrian isotopy can be broken down into a sequence of elementary isotopies, and we compose the DGA maps for the elementary pieces to produce a DGA map for the isotopy.

This picture fits in a natural way with cobordism maps. Given a Legendrian isotopy between $\Lambda_+$ and $\Lambda_-$, let $L$ denote the corresponding Lagrangian cobordism between $\Lambda_+$ and $\Lambda_-$. Then Ekholm--Honda--K\'alm\'an \cite[section 6.3]{EHK} show that over $\Z_2$, the cobordism map $\Phi_L :\thinspace (\SA_{\Lambda_+},\dd) \to (\SA_{\Lambda_-},\dd)$ agrees with the DGA map associated to the isotopy; note that by functoriality, it suffices to show this when $L$ is the cobordism for an elementary isotopy. 
This result was subsequently upgraded to $\Z$ coefficients by the combined work of K\'alm\'an \cite{Kalman}, who showed that the map of DGAs over $\Z$ associated to an isotopy (a path in the space of Legendrian links) is invariant under homotopy of the path; Ekholm--K\'alm\'an \cite{EK}, who showed that over $\Z_2$, this DGA map gives the differential for the Legendrian contact DGA of the Legendrian surface given by the lift of $L$; and Karlsson \cite[section 6]{Karlsson-cob}, who showed that one can assign signs to the differential of this Legendrian surface to induce signs for the cobordism map $\Phi_L$. For our purposes, we summarize this work as follows.

\begin{prop}[\cite{Kalman,EK,EHK,Karlsson-cob}]
Suppose that $\Lambda_+$ and $\Lambda_-$ are related by an elementary Legendrian isotopy, with corresponding Lagrangian cobordism $L$. Choose base points, a spin structure, and capping operators on $\Lambda_+$; these induce, via the isotopy, a corresponding choice of base points, spin structure, and capping operators on $\Lambda_-$. Then, the cobordism map
$\Phi_L :\thinspace (\SA_{\Lambda_+},\dd) \to (\SA_{\Lambda_-},\dd)$ is equal to the DGA map for the isotopy as constructed in \cite{Chekanov,ENS}.
\end{prop}

By ``elementary Legendrian isotopy'', we will mean one of the following three isotopies between Legendrian links with base points, all described in terms of their $xy$ projections:
\begin{itemize}
\item
Base point moves: fix the $xy$ projection and move a base point across a crossing,
\item
Reidemeister III moves (triple point moves),
\item
Reidemeister II moves.
\end{itemize}
Any Legendrian isotopy can be decomposed into these elementary isotopies, along with planar isotopies of the $xy$ projection in $\R^2$.

In the remainder of this subsection, we review the combinatorial formulas from \cite{Chekanov,ENS} for the DGA maps for elementary isotopies. As usual, to compute the cobordism map for a general Legendrian isotopy, we can divide the isotopy into elementary isotopies and compose the resulting cobordism maps.

\subsubsection{Base point moves}
\label{sssec:base-point}

Suppose that $\Lambda$ and $\Lambda'$ are Legendrian links that are related by a base point move: outside of a neighborhood of a Reeb chord $a$, their $xy$ projections agree, and inside this neighborhood, a base point moves across the crossing. See Figure~\ref{fig:basepoint}. Then the DGA map for this move is $\Psi :\thinspace (\SA_\Lambda,\dd) \to (\SA_{\Lambda'},\dd')$ defined as follows: $\Psi$ acts as the identity on all Reeb chords besides $a$ and on all base point variables including $s$, and
\begin{align*}
\Psi(a) &= sa \quad \text{(left diagram)} &&&
\Psi(a) &= as^{-1} \quad \text{(right diagram)}.
\end{align*}
Note that $\Psi$ is an isomorphism, and the DGA map for the reverse of one of these base point moves is $\Psi^{-1}$.

\begin{center}
	\begin{figure}[h!]
		\centering
								\includegraphics[scale=1.4]{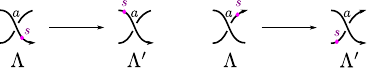}

		\caption{
A base point move.}
		\label{fig:basepoint}
	\end{figure}
\end{center}

We observe that if we move a base point (or collection of base points) all the way around a component of $\Lambda$ until it returns to where it started, the corresponding automorphism of $(\SA_\Lambda,\dd)$ is the identity map. (This uses the fact that the variable $s$ associated to the base point commutes with Reeb chord generators of $\SA_\Lambda$; in the fully noncommutative setting where $s$ does not commute with Reeb chords, the automorphism is conjugation by $s$.) As a consequence, when calculating the cobordism map for an isotopy cobordism $L$,
we do not need to specify an arc on $L$ joining corresponding base points on the ends of $L$, since any two choices of such an arc will yield the same map.

\subsubsection{Reidemeister III moves}
\label{sssec:RIII}

Suppose that $\Lambda$ and $\Lambda'$ are related by a Reidemeister III move: see Figure~\ref{fig:RIII}. There are two types of Reidemeister III moves, III$_a$ (left diagram) and III$_b$ (right diagram); these are called ``Move II'' and ``Move I'' in \cite{ENS}, respectively, and ``L1a'' and ``L1b'' in \cite{EHK}. There is a one-to-one correspondence between the Reeb chords of $\Lambda$ and $\Lambda'$, with the correspondence between the three crossings involved in the move shown in Figure~\ref{fig:RIII}. Under this identification, $\SA_\Lambda$ and $\SA_{\Lambda'}$ are identical.

\begin{center}
	\begin{figure}[h!]
		\centering
								\includegraphics[scale=1]{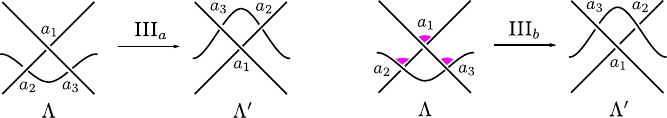}

		\caption{
The two types of Reidemeister III moves.}
		\label{fig:RIII}
	\end{figure}
\end{center}

The DGA map for Reidemeister III$_a$, which we will actually not need in this paper, is simply the identity map on $\SA_\Lambda$. To describe the DGA map for Reidemeister III$_b$, let $\sigma \in \{\pm 1\}$ denote the product of the orientation signs of the three quadrants of $\Pi_{xy}(\Lambda)$ indicated in Figure~\ref{fig:RIII}: this is $+1$ or $-1$ depending on whether an even or odd number of those quadrants are shaded. Then the DGA map $\Psi :\thinspace (\SA_\Lambda,\d) \to (\SA_{\Lambda'},\d')$ is defined to be the identity on all Reeb chords except for $a_1$ and on all base point variables, and
\[
\Psi(a_1) = a_1 + \sigma a_3 a_2.
\]

\subsubsection{Reidemeister II moves}

\begin{center}
	\begin{figure}[h!]
		\centering
								\includegraphics[scale=1.2]{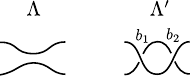}

		\caption{
A Reidemeister II move.}
		\label{fig:RII}
	\end{figure}
\end{center}

The DGA maps for a Reidemeister II move are more involved than for the other elementary isotopies.
Suppose that $\Lambda$ and $\Lambda'$ are related by a Reidemeister II move, with $\Pi_{xy}(\Lambda')$ having two more crossings than $\Pi_{xy}(\Lambda)$, as shown in Figure~\ref{fig:RII}. 

Let $a_1,\ldots,a_r$ be the Reeb chords of $\Lambda$, and let $b_1,b_2$ denote the two new Reeb chords of $\Lambda'$. Write $(\SA_{\Lambda},\dd)$ and $(\SA_{\Lambda'},\dd')$ for the DGAs of $\Lambda$ and $\Lambda'$. Let $|b_1|=i=|b_2|+1$ in $\SA_{\Lambda'}$, and construct the stabilization $(S(\SA_{\Lambda}),\dd)$ by adding two generators $e_1,e_2$ with $|e_1|=i=|e_2|+1$ to $\SA_{\Lambda}$ and extending the differential $\dd$ by $\dd(e_1)=e_2$, $\dd(e_2) = 0$. There is a chain isomorphism $\Psi :\thinspace \SA_{\Lambda'} \to S(\SA_{\Lambda})$ whose definition we recall below. We can then compose $\Psi^{-1}$ with the inclusion map $i :\thinspace \SA_\Lambda \to S(\SA_{\Lambda})$ to get a chain map $\Psi^{-1} \circ i :\thinspace \SA_\Lambda \to \SA_{\Lambda'}$. In the other direction, we can compose $\Psi$ with the projection map $p :\thinspace S(\SA_{\Lambda}) \to \SA_\Lambda$ sending each generator of $\SA_\Lambda$ to itself and sending $e_1,e_2$ to $0$, to get a chain map $p \circ \Psi :\thinspace \SA_{\Lambda'} \to \SA_\Lambda$:
\[
\xymatrix{
\SA_\Lambda \ar@<0.5ex>[r]^>>>>i \ar@/^2pc/[rr]^{\Psi^{-1} \circ i} & S(\SA_\Lambda) \ar@<0.5ex>[l]^<<<<p & \SA_{\Lambda'} \ar[l]_<<<<{\Psi}^<<<<{\cong} \ar@/^2pc/[ll]^{p\circ\Psi}.
}
\]
Then $\Psi^{-1}\circ i$ and $p\circ\Psi$ are the cobordism maps for the cobordisms from $\Lambda'$ to $\Lambda$ and from $\Lambda$ to $\Lambda'$, respectively, induced by the Reidemeister II isotopy.

We will need the precise definition of $\Psi$ from \cite{ENS}, and we recall it now. Let $a_1,\ldots,a_r$ be the Reeb chords of $\Lambda$, ordered in increasing height. Inductively construct a sequence of algebra isomorphisms $\Psi_1,\Psi_2,\ldots,\Psi_r :\thinspace \SA_{\Lambda'} \to S(\SA_\Lambda)$ as follows. By inspecting Figure~\ref{fig:RII}, we see that there is a bigon for $\Lambda'$ with $+$ corner at $b_1$ and $-$ corner at $b_2$, and so we can write $\dd'(b_1) = \sigma b_2 + v$ where $\sigma \in \{\pm 1\}$ and $v$ counts disks with $+$ corner in the leftmost quadrant at $b_1$.
The map $\Psi_1$ (written as $\Phi_0$ in \cite{ENS}) is defined by
\begin{align*}
\Psi_1(b_1) &= e_1 & \Psi_1(b_2) &= \sigma(e_2-v) & \Psi_1(a_\ell) &= a_\ell.
\end{align*}
Given $\Psi_{\ell-1}$, we define $\Psi_\ell = g_\ell \circ \Psi_{\ell-1}$, where $g_\ell :\thinspace S(\SA_{\Lambda}) \to S(\SA_{\Lambda})$ is the identity on all generators except $a_\ell$, and
\[
g_\ell(a_\ell) = a_\ell + H(\dd a_\ell - \Psi_{\ell-1} \dd' a_\ell).
\]
Here $H$ is the map on $S(\SA_{\Lambda})$ (a module map, not an algebra map) defined by
$H(w) = 0$ if $w$ is any word that either does not contain $e_1$ or $e_2$, or for which the leftmost $e_i$ appearing in $w$ is $e_1$, and
$H(w_1 e_2 w_2) = (-1)^{|w_1|+1} w_1 e_1 w_2$ if $w_1$ does not contain $e_1$ or $e_2$. Finally, $\Psi = \Psi_r$.

Noting that for each $\ell$, $\Psi(a_\ell) = \Psi_\ell(a_\ell) = g_\ell(a_\ell)$, we can restate the definition of $\Psi$ more succinctly as follows:
\begin{align}
\begin{split}
\Psi(b_1) &= e_1 \\
\Psi(b_2) &= \sigma(e_2-v) \\
\Psi(a_\ell) &= a_\ell - H(\Psi \dd' a_\ell).
\end{split}
\label{eq:RII}
\end{align}
This definition looks circular since $\Psi$ occurs on the right hand side of the definition of $\Psi(a_\ell)$, but in fact the height ordering and Stokes' Theorem imply that for any $\ell$, $\dd' a_\ell$ involves only $b_1,b_2,a_1,\ldots,a_{\ell-1}$ and not $a_{\ell+1},\ldots,a_r$, and so \eqref{eq:RII} can be used to recursively define $\Psi(a_\ell)$. Note that the height ordering does not appear explicitly in \eqref{eq:RII}; however, the existence of the height filtration means that the recursive definition \eqref{eq:RII} terminates and thus produces a well-defined result.

\begin{remark}
It follows from the definition of $\Psi$ that the chain map $p\circ\Psi :\thinspace \SA_{\Lambda'} \to \SA_{\Lambda}$ has the following simple form: 
\label{rmk:RII}
\begin{align*}
(p\circ\Psi)(b_1) &= 0 & (p\circ\Psi)(b_2) &= -\sigma v & (p\circ\Psi)(a_\ell)  &= a_\ell.
\end{align*}\hfill$\Box$
\end{remark}

This concludes our description of the DGA maps associated to isotopy cobordisms.

\subsection{The cobordism map for a saddle cobordism}
\label{ssec:EHK-map}

We now address the cobordism map associated to a saddle cobordism. Let $\Lambda_+$ be a Legendrian link with a contractible Reeb chord $a$ of degree $0$; contractible chords of even degree can be similarly treated with suitable modification to the grading. 
In the $xy$ projection, replacing the crossing $a$ by its oriented resolution yields a Legendrian link $\Lambda_-$, and we write $L_a$ for the saddle cobordism between $\Lambda_-$ and $\Lambda_+$. 

Our goal in this subsection is to write down a combinatorial formula for the cobordism map $\Phi_{L_a} :\thinspace (\SA_{\Lambda_+},\dd) \to (\SA_{\Lambda_-},\dd)$. In \cite{EHK}, Ekholm--Honda--K\'alm\'an describe such a formula for this map over $\Z_2$, subject to the assumption that the Reeb chord $a$ is what they call ``simple''. Our goal here is to describe the EHK map over $\Z$ and for what we call ``proper chords'', which are a different (and apparently larger) class of contractible Reeb chords than simple chords. The proof that our map is indeed the geometric cobordism map $\Phi_{L_a}$ (stated as Proposition~\ref{prop:EHK} below) is deferred to Appendix~\ref{sec:saddle-map}.

Recall from Section~\ref{ssec:dga-def} that the Legendrian contact differential $\dd$ for a Legendrian $\Lambda$ counts immersed disks with a single $+$ corner, where ``immersed disk'' in our terminology includes the condition that all punctures are mapped to single quadrants (i.e., all corners are convex). We will now need to consider more general disks, which we call \textit{immersed disks with concave corners}. These are immersed disks where each boundary puncture is again mapped to a crossing of the Lagrangian projection $\Pi_{xy}(\Lambda)$, but where we now allow a neighborhood of each boundary puncture to be mapped to either a single quadrant at the crossing (a convex corner) or the union of three quadrants (a concave corner). As with convex corners, we can label each concave corner as positive or negative, depending on whether 2 of the 3 quadrants covered by the corner are positive or negative, respectively.

\begin{definition}
A contractible Reeb chord $a$ of $\Lambda_+$ is \textit{proper} if the following condition holds. For any immersed disk $\Delta$, possibly with concave corners, such that:
\begin{itemize}
\item
$\Delta$ has a positive convex corner at some Reeb chord besides $a$,
\item
$\Delta$ has at least one positive convex corner at $a$,
\item
all other convex corners of $\Delta$ are negative, and
\item
the only possible concave corners of $\Delta$ are positive concave corners at $a$,
\end{itemize}
then it must be the case that the (closure of the) boundary of $\Delta$ in $\Pi_{xy}(\Lambda_+)$ passes through the crossing $\Pi_{xy}(a)$ only once. 
\label{def:proper}
That is, any immersed disk $\Delta$ with the given properties must have no concave corners, $\Delta$ must have exactly one positive convex corner at $a$, and the boundary of $\Delta$ never passes through $a$ except at that corner. \hfill$\Box$
\end{definition}

\begin{remark}
We remark that being proper and being simple in the sense of \cite{EHK} are not the same. 
\label{rmk:proper-simple}
We refer to \cite[Definition 6.16]{EHK} for the index condition that defines the latter property. In our language, the condition for a Reeb chord $a$ to be simple can be restated as follows: for any immersed disk $\Delta$ such that $\Delta$ has $k$ positive convex corners at $a$, a positive convex corner at some Reeb chord besides $a$, and all other corners (including concave corners) being negative, it must be the case that $\Delta$ has at least $k$ concave corners.\footnote{In conversation with T.\ Ekholm, it emerged that there is a typo in \cite[Definition 6.16]{EHK}: the inequality $\operatorname{ind}(u) \geq k$ in that definition should be $\operatorname{ind}(u) \geq k+1$. The revised inequality corresponds to our condition of having at least $k$ concave corners.} 
The contractible Reeb chord $a_9$ in Figure~\ref{fig:proper-ex} below is proper but not simple; the necessity of considering saddle moves at Reeb chords like $a_9$ in this paper is what motivated our definition of proper chords. We do not know if all simple contractible chords must be proper.\hfill$\Box$
\end{remark}

All of the Reeb chords that we use in this paper to perform saddle cobordisms are contractible and proper. This is a consequence of the following result.

\begin{prop}
If $\beta\in\Br_N^+$ is an admissible braid and $a$ is a crossing of $\beta$ such that $\beta\setminus\{a\}$ contains a half-twist, then as a Reeb chord of the $(-1)$-closure $\Lambda(\beta)$, $a$ is contractible and proper.
\label{prop:halftwist3}
\end{prop}

\begin{proof}
Contractibility has already been shown in Proposition~\ref{prop:halftwist2}; we need to show properness. Suppose that $\Delta$ is a disk as in Definition~\ref{def:proper}. Because $\Delta$ has a positive convex corner at $a$, it must be ``thin'' in the sense that it lies in the neighborhood of the Legendrian unknot that contains the satellite $\Lambda(\beta)$. The presence of the half-twist, and the fact that $\Delta$ has no concave corners in the half-twist, prevents $\Delta$ from passing through the half-twist. This forces $\Delta$ to be embedded in the neighborhood of the unknot, and so its boundary only passes through $a$ once.
\end{proof}

We will next present a formula for the map for a saddle cobordism at a Reeb chord $a$ when $a$ is contractible and proper. As in \cite{EHK}, the key is to consider immersed disks with two $+$ corners, one of which is at $a$. We break these into two types. 

\begin{center}
	\begin{figure}[h!]
		\centering
				\includegraphics[scale=1.5]{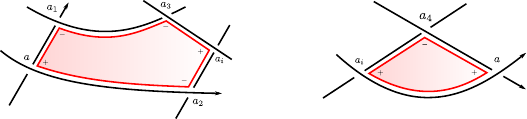}
		\caption{
A disk in $\Delta^\rightarrow_a(a_i)$ (left) and a disk in $\Delta^\leftarrow_a(a_i)$ (right). For the disk on the left, $w_1(\Delta) = a_3a_1$ and $w_2(\Delta) = a_2$. For the disk on the right, $w_1(\Delta) =1$ and $w_2(\Delta) = a_4$.}
		\label{fig:leftright}
	\end{figure}
\end{center}

Let $a_i$ be a Reeb chord of $\Lambda_+$ not equal to $a$. Define $\Delta^\rightarrow_a(a_i)$, respectively $\Delta^\leftarrow_a(a_i)$, to be the set of immersed disks for $\Lambda_+$, such that:
\begin{itemize}
\item
all corners are convex, and there are exactly two positive corners, one at $a_i$ and one at $a$;
\item
at the corner at $a$, the orientation of $\Lambda$ points toward, respectively away from (for $\Delta^\leftarrow_a(a_i)$), the disk.
\end{itemize}
See Figure~\ref{fig:leftright}. For any $\Delta \in \Delta^\rightarrow_a(a_i) \cup \Delta^\leftarrow_a(a_i)$, we can define three quantities. One is the sign $\sgn(\Delta) \in \{\pm 1\}$, which is the product of the orientation signs over all corners of $\Lambda$, multiplied by the signs of any base points traversed by the boundary of the disk. The other are two words $w_1(\Delta),w_2(\Delta) \in \SA_{\Lambda_+}$, defined as follows: $w_1(\Delta)$ is the product of the $-$ corners and base points that we encounter as we traverse the boundary of $\Delta$ counterclockwise from $a_i$ to $a$, and $w_2(\Delta)$ is the analogous product as we traverse the boundary counterclockwise from $a$ to $a_i$. See Figure~\ref{fig:leftright} for an example.

\begin{definition}
The combinatorial cobordism map, denoted $\Phi_{L_a}^\comb:\SA_{\Lambda_+} \to \SA_{\Lambda_-}$, is the composition of three algebra maps:
\label{def:cob-comb}
\[
\Phi_{L_a}^\comb:= \Phi^\leftarrow \circ \Phi^\rightarrow \circ \Phi^0,
\]
where $\Phi_0  \thinspace: \SA_{\Lambda_+} \to \SA_{\Lambda_-}$ is defined by $\Phi_0(a) = s$ and $\Phi_0(a_i) = a_i$ for any Reeb chord $a_i$ besides $a$, and $\Phi^\rightarrow,\Phi^\leftarrow :\thinspace \SA_{\Lambda_-} \to \SA_{\Lambda_-}$ are defined as follows.
Let $a_i$ be a generator of $\SA_{\Lambda_-}$, that is, a Reeb chord of $\Lambda_-$, which is then also a Reeb chord of $\Lambda_+$. Then,
\begin{align*}
\Phi^\rightarrow(a_i) &:= a_i + \sum_{\Delta\in\Delta^\rightarrow_a(a_i)} (-1)^{|w_1(\Delta)|}\sgn(\Delta) \Phi^\rightarrow(w_1(\Delta)) s^{-1} w_2(\Delta) \\
\Phi^\leftarrow(a_i) &:= a_i + \sum_{\Delta\in\Delta^\leftarrow_a(a_i)} (-1)^{|w_1(\Delta)|}\sgn(\Delta) \Phi^\leftarrow(w_1(\Delta)) s^{-1} w_2(\Delta).
\end{align*}\hfill$\Box$
\end{definition}

\begin{remark}
As with the definition of the Reidemeister II cobordism map, equation \eqref{eq:RII} in Section~\ref{ssec:isotopy}, these definitions may appear circular but can be used to recursively define $\Phi^\rightarrow$ and $\Phi^\leftarrow$. The reason is that if we order the Reeb chords $a_1,\ldots,a_r$ in increasing order of height, then all disks with positive punctures at $a$ and $a_{j}$ can only have negative corners at $a_{1},\ldots,a_{j-1}$ and not at $a_{j+1},\ldots,a_{i_r}$: in particular, if $\Delta \in \Delta^\rightarrow_a(a_{j}) \cup \Delta^\leftarrow_a(a_{j})$ then $w_1(\Delta)$ only involves $a_{1},\ldots,a_{j-1}$.\hfill$\Box$
\end{remark}

The key result in this subsection is the relation between $\Phi_{L_a}^\comb$, as defined above, and $\Phi_{L_a}$. This is the content of the following result:

\begin{prop}
If $a\in\SA_{\Lambda_+}$ is a proper contractible Reeb chord, then the cobordism map $\Phi_{L_a} :\thinspace \SA_{\Lambda_+} \to \SA_{\Lambda_-}$ is equal to the combinatorial map $\Phi_{L_a}^\comb$, up to a link automorphism of $\Lambda_-$. That is, there is a link automorphism $\Omega :\thinspace \SA_{\Lambda_-} \to \SA_{\Lambda_-}$ such that
$\Phi_{L_a} = \Omega \circ \Phi_{L_a}^\comb$.
\label{prop:EHK}
\end{prop}

Proposition~\ref{prop:EHK} is proved in Appendix~\ref{sec:saddle-map} below. Let us illustrate how to compute $\Phi_{L_a}^\comb$ in an explicit example, which will also appear as part of our later computations.

\begin{center}
	\begin{figure}[h!]
		\centering
				\includegraphics[scale=1.2]{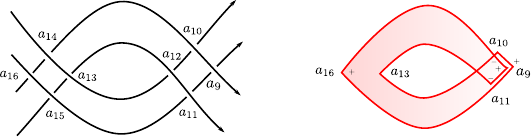}
		\caption{
Calculating the cobordism map for a saddle cobordism at $a_9$. Left, the Legendrian link at the top of the cobordism; right, an immersed disk showing that $a_9$ is not simple.}
		\label{fig:proper-ex}
	\end{figure}
\end{center}

\begin{ex}
Consider the configuration shown in Figure~\ref{fig:proper-ex}, this appears as part of our calculations for the $\widetilde{D}_4$-Legendrian in Section~\ref{ssec:AffineD4}. 
\label{ex:proper-ex}
The first step in that calculation is a saddle move at $a_9$, and we calculate the corresponding map $\Phi^\leftarrow$ here. By inspection we see that $\Delta^\leftarrow_{a_9}(a_i) = \emptyset$ for $10\leq i\leq 13$, while $\Delta^\leftarrow_{a_9}(a_{14})$ and $\Delta^\leftarrow_{a_9}(a_{15})$ each contain one disk apiece, with negative corners at $a_{13},a_{11}$ and $a_{10},a_{13}$ respectively. For $a_{16}$, $\Delta^\leftarrow_{a_9}(a_{16})$ contains three disks, one with no negative corners, one with negative corners at $a_{10},a_{14}$, and one with negative corners at $a_{15},a_{11}$. It follows from this that $\Phi^\leftarrow(a_i) = a_i$ for $10\leq i\leq 13$ and
\begin{align*}
\Phi^\leftarrow(a_{14}) &= a_{14} - \Phi^\leftarrow(a_{13}a_{11})s^{-1} = a_{14} - a_{13}a_{11}s^{-1} \\
\Phi^\leftarrow(a_{15}) &= a_{15} - \Phi^\leftarrow(1)s^{-1}a_{10}a_{13} = a_{15} - s^{-1}a_{10}a_{13} \\
\Phi^\leftarrow(a_{16}) &= a_{16} - s^{-1} - s^{-1}a_{10}a_{14}-\Phi^\leftarrow(a_{15}a_{11})s^{-1} =\\
&= a_{16}- s^{-1} - s^{-1}a_{10}a_{14}-(a_{15} - s^{-1}a_{10}a_{13})a_{11}s^{-1}.
\end{align*}
For the complete Legendrian that we study in Section~\ref{ssec:AffineD4}, an inspection of Figure~\ref{fig:Proof_AffineD4Braid} shows that $\Delta_{a_9}^\rightarrow(a_i) = \emptyset$ and thus $\Phi^\rightarrow(a_i) = a_i$  for $10\leq i\leq 16$. It follows that the map $\Phi_{L_{a_9}}^\comb$ sends $a_9$ to $s$ and agrees with $\Phi^\leftarrow$ for $a_i$, $10\leq i\leq 16$.

We note that in this example, $a_9$ is contractible and proper but not simple, and thus even over $\Z_2$ we cannot directly apply the combinatorial formula from \cite{EHK}. The reason $a_9$ is not simple (cf.\ Remark~\ref{rmk:proper-simple}) is the disk shown in Figure~\ref{fig:proper-ex}, which has $2$ positive corners at $a_9$, $1$ positive corner at $a_{16}$, and a single concave corner at $a_{13}$.\hfill$\Box$
\end{ex}

\begin{remark}
If $a$ is not just proper but also simple, then our definition of $\Phi_{L_a}^\comb$ can be stated in an easier way, to match \cite{EHK}. In this case, write $\Delta_a(a_i) = \Delta^\rightarrow_a(a_i) \cup \Delta^\leftarrow_a(a_i)$. If $\Delta$ is any disk in $\Delta_a(a_i)$ and $a_j$ is a negative corner of $\Delta$, then it must be the case that $\Delta_a(a_j) = \emptyset$; otherwise the union of $\Delta$ and a disk $\Delta' \in \Delta_a(a_j)$ is an immersed disk with concave corner at $a_j$ and two positive (convex) corners at $a$, violating the simplicity condition. Then we can drop the $\Phi^\rightarrow$ and $\Phi^\leftarrow$ in $\Phi^\rightarrow(w_1(\Delta))$ and $\Phi^\leftarrow(w_1(\Delta))$, and conclude directly that for all $a_i$,
\[
\Phi_{L_a}^\comb(a_i) = a_i + \sum_{\Delta\in\Delta_a(a_i)} \sgn(\Delta) w_1(\Delta) s^{-1} w_2(\Delta).
\]
If we set $s=1$ and reduce mod $2$, this recovers the formula for the cobordism map from \cite[Proposition 6.18]{EHK}.\hfill$\Box$
\end{remark}

\subsection{Assembling elementary cobordism maps}
\label{ssec:linkaug-filling}

Having described the cobordism maps for elementary cobordisms, we can calculate the map associated to any decomposable cobordism by composing the maps for its elementary pieces, and indeed this is what we do in Sections~\ref{sec:DGA} and~\ref{sec:Monodromy} below. There is a possible difficulty with this approach: we have only calculated the saddle cobordism map up to a link automorphism (see Proposition~\ref{prop:EHK}). However, for a filling, the extra flexibility provided by the basepoint parameters gets rid of this problem, as we explain in this subsection.

Let $L$ be a connected decomposable genus-$g$ filling of an $m$-component Legendrian link $\Lambda$. As in Section~\ref{ssec:system}, we decorate $L$ with arcs corresponding to base points $t_1,\ldots,t_m,s_1,\ldots,s_\ell$.
Divide $L$ into elementary cobordisms $L_1,\ldots,L_k$, where:
\begin{itemize}
\item
$L_j$ is a cobordism between Legendrians $\Lambda_{j-1}$ and $\Lambda_j$, with $\Lambda_0 = \emptyset$ and $\Lambda_k = \Lambda$;
\item
$L_1$ is a disjoint union of minimum cobordisms;
\item
for $j=2,\ldots,k$, $L_j$ is either an isotopy cobordism or a saddle cobordism.
\end{itemize}
Note that this decomposition differs slightly from our simplified setup in Section~\ref{ssec:system}, where we suppressed isotopy cobordisms.

As in Section~\ref{ssec:system}, let $R$ be the ring
$$R:=(\Z[t_1^{\pm 1},\ldots,t_m^{\pm 1},s_1^{\pm 1},\ldots,s_{\ell}^{\pm 1}])/(w_1=\cdots=w_k=-1) \cong \Z[H_1(L) \oplus \Z^{m-1}],$$
where $w_1,\ldots,w_k$ are words coming from the minima of $L$. For each $j$, let $(\SA_{\Lambda_j},\dd^\comb)$ denote the DGA for $\Lambda_j$ over $R$, with $\SA_\Lambda = \SA_{\Lambda_k}$. Then each elementary cobordism gives a map $\Phi_{L_j} :\thinspace (\SA_{\Lambda_j},\dd^\comb) \to (\SA_{\Lambda_{j-1}},\dd^\comb)$, and their composition is a $(2g+2m-2)$-system of augmentations for $\Lambda$:
\[
\Phi_L = \Phi_{L_1} \circ \cdots \circ \Phi_{L_k} :\thinspace (\SA_{\Lambda},\dd^\comb) \to (R,0).
\]

Now suppose that for $j=1,\ldots,k$, $\Omega_j :\thinspace \SA_{\Lambda_{j-1}} \to \SA_{\Lambda_{j-1}}$ is a link automorphism of $\Lambda_{j-1}$, and define $\widetilde{\Phi}_{L_j} = \Omega_j \circ \Phi_{L_j}$.

\begin{prop}
The maps $\Phi_L = \Phi_{L_1} \circ \cdots \circ \Phi_{L_k}$ and $\widetilde{\Phi}_L = \widetilde{\Phi}_{L_1} \circ \cdots \circ \widetilde{\Phi}_{L_k}$ are equivalent systems of augmentations of $\Lambda$.
\label{prop:linkaug-filling}
\end{prop}

\begin{proof}
We prove by induction that for $j=1,\ldots,k$, $\Phi_{L_1} \circ \cdots \circ \Phi_{L_j}$ and $\widetilde{\Phi}_{L_1} \circ \cdots \circ \widetilde{\Phi}_{L_j}$ are equivalent as maps $(\SA_{\Lambda_j},\dd^\comb) \to (R,0)$. The base case $j=1$ is true since $\Phi_{L_1}$, and thus $\widetilde{\Phi}_{L_1}$, are both the zero map on Reeb chords of $\Lambda_1$.

For the induction step, assume that $\Phi_{L_1} \circ \cdots \circ \Phi_{L_j}$ and $\widetilde{\Phi}_{L_1} \circ \cdots \circ \widetilde{\Phi}_{L_j}$ are equivalent, so that there is an automorphism $\psi_j$ of $R$ such that $\widetilde{\Phi}_{L_1} \circ \cdots \circ \widetilde{\Phi}_{L_j} = \psi_j \circ (\Phi_{L_1} \circ \cdots \circ \Phi_{L_j})$. Since the map $\Phi_{L_1} \circ \cdots \circ \Phi_{L_j}$ agrees with the geometric system of augmentations for $L_1 \cup \cdots \cup L_j$ by Proposition~\ref{prop:decomp-system}, and the geometric system incorporates link automorphisms of $\Lambda_j$, the link automorphism $\Omega_j$ of $\Lambda_j$ induces an automorphism $\omega_j$ of $R$ such that
$$(\Phi_{L_1} \circ \cdots \circ \Phi_{L_j}) \circ \Omega_j = \omega_j \circ (\Phi_{L_1} \circ \cdots \circ \Phi_{L_j}).$$
(Note that Proposition~\ref{prop:decomp-system} assumes that $L_1\cup\cdots\cup L_j$ is connected; however, the argument here extends to the disconnected case as well, since the system of augmentations of a disconnected filling annihilates any Reeb chord with endpoints on different components.)

We conclude that the following diagram commutes:
\[
\xymatrix{
&&\SA_{\Lambda_{j+1}} \ar[dl]_>>>>>{\Phi_{L_{j+1}}} \ar[dr]^>>>>>{\widetilde{\Phi}_{L_{j+1}}} && \\
& \SA_{\Lambda_j} \ar[rr]^{\Omega_j}_\cong \ar[dl]_>>>>{\Phi_{L_1} \circ \cdots \circ \Phi_{L_j}} && \SA_{\Lambda_j} \ar[dl]_>>>{\Phi_{L_1} \circ \cdots \circ \Phi_{L_j}} \ar[dr]^>>>>{\widetilde{\Phi}_{L_1} \circ \cdots \circ \widetilde{\Phi}_{L_j}} & \\
R \ar[rr]^{\omega_j}_\cong && R \ar[rr]^{\psi_j}_\cong && R.
}
\]
It follows that $\Phi_{L_1} \circ \cdots \circ \Phi_{L_j}\circ \Phi_{L_{j+1}}$ and $\widetilde{\Phi}_{L_1} \circ \cdots \circ \widetilde{\Phi}_{L_j} \circ \widetilde{\Phi}_{L_{j+1}}$ are equivalent since one is the composition of the other with $\psi_j \circ \omega_j$, and this completes the induction.
\end{proof}

By Proposition~\ref{prop:linkaug-filling}, when we build systems of augmentations for fillings by composing elementary cobordism maps, we can replace any elementary cobordism map by its composition with a link automorphism. In particular, Proposition~\ref{prop:EHK} implies that we can use the combinatorial saddle map $\Phi^\comb$ as the cobordism map for a saddle cobordism, and this is what we will do in subsequent sections.

\section{Legendrian Contact DGA and Cobordism Maps for $(-1)$-closures
}\label{sec:DGA}

In this section we present an algebraically amenable description of the Legendrian contact DGA for the $(-1)$-closure of an admissible braid, and detail the effect of the $\vartheta$-loops and saddle cobordisms on the DGA.

\subsection{The DGA of the $(-1)$-closure of an admissible braid}\label{ssec:DGAClosure}

Let $\sigma_{k_1} \cdots \sigma_{k_r} \in \mbox{Br}^+_N$ be an admissible positive braid. Henceforth we will write $\Lambda(\sigma_{k_1} \cdots \sigma_{k_r})$ for the $(-1)$-closure of this braid in the sense of Definition~\ref{def:closure}.\footnote{Note that this differs from the notation $\Lambda(\beta)$ in Section~\ref{ssec:LegendrianLinks}, but by Proposition~\ref{prop:admissible}, the two notations represent links that are Legendrian isotopic.} We decorate the $xy$ projection of $\Lambda(\sigma_{k_1} \cdots \sigma_{k_r})$ as follows; see Figure~\ref{fig:closure}. Place a column of base points on the $n$ strands of the braid between braid crossings, as well as on either end of the braid, and label these base points $t_{\ell,i}$, $1\leq i\leq n$, $0\leq \ell\leq r$. (In practice we may only need some small subset of these base points; in that case we formally set $t_{\ell,i}=1$ for all of the other base points and then remove them.)
The Reeb chords for $\Lambda(\sigma_{k_1} \cdots \sigma_{k_r})$ consist of:
\begin{itemize}
	\item
	$a_1,\ldots,a_r$, of degree $0$, corresponding to the crossings of the braid, and labeled in the obvious way;
	\item
	$c_{ij}$, $1\leq i,j\leq n$, of degree $1$, corresponding to the Reeb chord of the standard Legendrian unknot $U$.
\end{itemize}
Recall from Section~\ref{ssec:dga-def} that in order to calculate degrees of Reeb chords, we need to choose a base point on each component of the link; any subset of the $t_{\ell,i}$ will do and produces the degrees given above.

\begin{center}
	\begin{figure}[h!]
		\centering
				\includegraphics[scale=1]{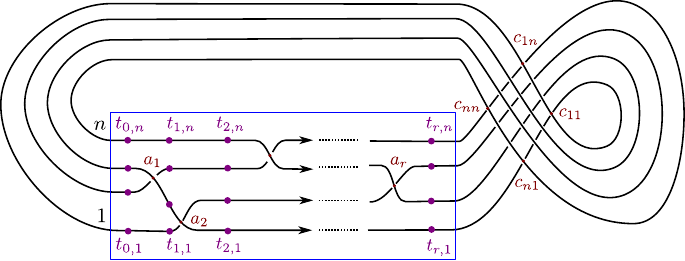}
		\caption{The Lagrangian projection of the Legendrian link $\Lambda(\sigma_{k_1} \cdots \sigma_{k_r})$, with crossings and base points labeled. The braid itself is in the blue box. Arrows represent the orientation of the link.}
		\label{fig:closure}
	\end{figure}
\end{center}

The differential on the Legendrian contact DGA of $\Lambda(\sigma_{k_1} \cdots \sigma_{k_r})$ can be expressed in a compact way using the path matrices of K\'alm\'an \cite{Kalman-braid}.\footnote{Note that we number our braid strands in increasing order from bottom to top, while K\'alm\'an numbers braid strands from top to bottom. We also incorporate base points while K\'alm\'an does not.}
For $k=1,\ldots,n-1$, define an $n\times n$ matrix $P_k(a)$ (as a function of an input $a$) as follows:
\[
(P_k(a))_{ij} = \begin{cases} 1 & i=j \text{ and } i\neq k,k+1 \\
1 & (i,j) = (k,k+1) \text{ or } (k+1,k) \\
a & i=j=k+1 \\
0 & \text{otherwise;}
\end{cases}
\]
that is, $P_k(a)$ is the identity matrix except for the $2\times 2$ submatrix given by rows and columns $k$ and $k+1$, which is $\left( \begin{smallmatrix} 0 & 1 \\ 1 & a \end{smallmatrix} \right)$.
(These are the path matrices considered in \cite{Kalman-braid}, but note that we number our braid strands in increasing order from bottom to top, while K\'alm\'an numbers braid strands from top to bottom.) Also define $\mathbf{t}_\ell = (t_{\ell,1},\ldots,t_{\ell,n})$ and write
$D(\mathbf{t}_\ell)$ for the diagonal $n\times n$ matrix with $t_{\ell,1},\ldots,t_{\ell,n}$ along the diagonal.

\begin{definition}
Let $\beta = \sigma_{k_1}\cdots\sigma_{k_r}$ be an $n$-stranded braid decorated with base points, with crossings and base points labeled as in Figure~\ref{fig:closure}. The \textit{path matrix} of $\beta$ is the $n\times n$ matrix
\[
P_\beta = D(\mathbf{t}_0) P_{k_1}(a_1) D(\mathbf{t}_1) P_{k_2}(a_2) D(\mathbf{t}_2) \cdots P_{k_r}(a_r) D(\mathbf{t}_r).
\]
\end{definition}

Colloquially, the $(i,j)$ entry of the path matrix $P(\beta)$ counts paths beginning at the left of $\beta$ on strand $i$, ending at the right on strand $j$, and at each crossing the path encounters, either passing straight through the crossing, or turning a corner if the path changes direction from southeast to northeast at the corner. Each path produces a word by reading the base points traversed and corners turned in order, and the $(i,j)$ entry of $P(\beta)$ is the sum of these words.

\begin{prop}
	The differential on the DGA $(\SA_{\Lambda(\sigma_{k_1} \cdots \sigma_{k_r})},\dd)$ for $\Lambda(\sigma_{k_1} \cdots \sigma_{k_r})$ is given as follows: $\dd(a_\ell) = 0$, and if we assemble the $c_{ij}$ into an $n\times n$ matrix $C = (c_{ij})$ and write $\mathbf{1}$ for the $n\times n$ identity matrix, then:
	\label{prop:differential}
	\[
	\dd(C) = \mathbf{1} + P_\beta.
	\]
\end{prop}

\begin{proof}
Each degree-$0$ generator $a_\ell$ has vanishing differential for degree reasons. For $c_{ij}$, there are two possible types of immersed disks (all of which are in fact embedded) with $+$ corner at $c_{ij}$, depending on which $+$ quadrant at $c_{ij}$ is covered by the disk. There is an embedded disk with $+$ puncture at the right quadrant of $c_{ij}$ and no $-$ puncture if $i=j$, and otherwise there is no immersed disk with $+$ puncture at this right quadrant. This produces the $\mathbf{1}$ term in the formula. For embedded disks with $+$ puncture at the left quadrant of $c_{ij}$, we need to keep track of ways that the boundary of this disk can enter the braid from the left on strand $i$ and exit the braid to the right on strand $j$, with possible convex corners at some crossings $a_\ell$. The contribution of these disks to $\dd(c_{ij})$ is precisely the $(i,j)$ entry of the path matrix $P_\beta$.
\end{proof}

\subsection{$\vartheta$-monodromy action on the DGA} 
\label{ssec:purple-monodromy}
Consider a Legendrian link $\La=\La(\beta,k;\gamma)\sse(\R^3,\xi_\st)$ and its $\vartheta$-loop, as defined in Section \ref{ssec:PurpleBox}. Here we compute the morphism
$$\SA(\vartheta):\SA_\La\lr\SA_\La$$
induced by this Legendrian isotopy, which we call the $\vartheta$-monodromy or purple box monodromy. 
To be precise, any Legendrian isotopy between Legendrian links induces a chain isomorphism between the (suitably stabilized) DGAs of the links, as described in \cite{Chekanov,ENS}. In the case of the isotopy given by the $\vartheta$-loop, which consists entirely of Reidemeister III moves, it is not necessary to stabilize the DGAs, and as a result we obtain the aforementioned chain isomorphism $\SA(\vartheta)$, which we now compute explicitly.

The Lagrangian projection of $\La(\beta,k;\gamma)$ is given in the right diagram in Figure~\ref{fig:SatelliteUnknotComponent}.
Let $n$ denote the braid index of $\gamma$, and let $N$ be the number of braid strands in $\Lambda(\beta,k;\gamma)$, so that the braid index of $\beta$ is $N-n+1$. 
As in Section~\ref{ssec:DGAClosure}, the Reeb chords of $\La(\beta,i;\gamma)$, which generate the Legendrian contact DGA $\SA_\La$, come in two types: the degree $1$ chords $c_{ij}$, $i,j\in [1,N]$, and the degree $0$ chords in the braiding region. We can divide these Reeb chords into two types in another way. Call the sublink of $\La(\beta,i;\gamma)$ corresponding to the $i$-th strand of $\beta$ (and containing the purple box $\gamma$) the \textit{satellite sublink}; this is depicted in purple in Figure~\ref{fig:SatelliteUnknotComponent}. We call crossings of $\La(\beta,i;\gamma)$ \textit{satellite crossings} and \textit{non-satellite crossings} depending on whether or not they involve the satellite sublink.
Note that the satellite crossings of degree $1$ are precisely $c_{ij}$ with $i,j \in \{k,\ldots,k+n-1\}$, while the satellite crossings of degree $0$ come in groups of $n$, with each group coming from a single crossing of $\beta$.

We allow for the placement of arbitrarily many base points on $\La(\beta,k;\gamma)$, subject to the restriction that any base points lying on the satellite sublink actually lie in the purple box for $\gamma$. (In practice, there will be one base point per strand of $\Lambda(\beta,k;\gamma)$, and the base points in the purple box will lie on its right edge.)
Let $P_\gamma$ denote the $n\times n$ path matrix for $\gamma$ with its base points. Extend this to an $N\times N$ matrix $\wt P_\gamma$ by
\[
\wt P_\gamma = \left( \begin{matrix} \mathbf{1} & 0 & 0 \\
0 & P_\gamma & 0 \\
0 & 0 & \mathbf{1}
\end{matrix} \right)
\]
where the central matrix $P_\gamma$ corresponds to rows and columns $k,\ldots,k+n-1$.

\begin{prop}
The purple-box monodromy map $\SA(\vartheta):\SA_\La\lr\SA_\La$ is given on generators as follows. Assemble the degree $1$ generators $c_{ij}$ into an $N\times N$ matrix: then
\[
\SA(\vartheta)(C) = \wt P_\gamma C \wt P_\gamma^{-1}.
\]
For degree $0$ generators, $\SA(\vartheta)$ fixes all non-satellite crossings, while its action on degree $0$ satellite crossings is as follows:
\begin{align*}
\SA(\vartheta) \left( \begin{smallmatrix} h_1 \\ \vdots \\ h_n \end{smallmatrix} \right) &= P_\gamma \left( \begin{smallmatrix} h_1 \\ \vdots \\ h_n \end{smallmatrix} \right) \\
\SA(\vartheta) \left( \begin{smallmatrix} h_1' \\ \vdots \\ h_n' \end{smallmatrix} \right) &= \left( P_\gamma^T \right)^{-1}  \left( \begin{smallmatrix} h_1' \\ \vdots \\ h_n' \end{smallmatrix} \right).
\end{align*}
Here $h_1,\ldots,h_n$ is any group of satellite crossings coming from a crossing of $\beta$ where the $i$-th strand is the overcrossing, while $h_1',\ldots,h_n'$ is any group of satellite crossings coming from a crossing of $\beta$ where the $i$-th strand is the undercrossing. See Figure~\ref{fig:satellite-crossings}.
\label{prop:purple-monodromy}
\end{prop}

\begin{center}
	\begin{figure}[h!]
		\centering
		\includegraphics[scale=0.75]{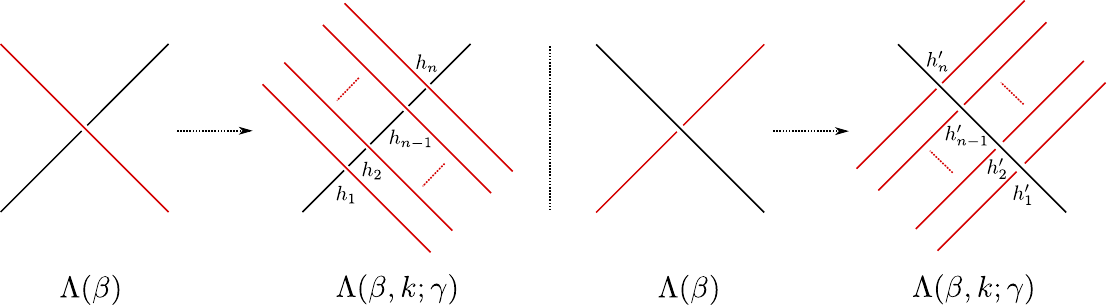}
		\caption{A group of satellite crossings coming from an overcrossing (left) and an undercrossing (right).}
		\label{fig:satellite-crossings}
	\end{figure}
\end{center}

\begin{proof}
The $\vartheta$-loop consists of a sequence of Reidemeister III moves that push the purple box around, and consequently the map $\SA(\vartheta)$ is the composition of a sequence of algebra isomorphisms corresponding to these Reidemeister III moves, as given concretely in Section~\ref{sssec:RIII}. In particular, any non-satellite crossing does not participate in any of the Reidemeister III moves and so it is fixed by $\SA(\vartheta)$.

Next consider a group of degree $0$ satellite crossings $h_1,\ldots,h_n$ as in the statement of the proposition (the argument for $h_1',\ldots,h_n'$ is similar and will be omitted). The $\vartheta$-loop pushes the purple box containing $\gamma$ through $h_1,\ldots,h_n$ from right to left. Since the path matrix $P_\gamma$ is a product of path matrices for individual crossings and columns of base points, and we can factor the action of $\SA(\vartheta)$ on $h_1,\ldots,h_n$ by pushing each individual crossing and base point column across $h_1,\ldots,h_n$ from right to left and composing the results, the key is to observe what happens when we push a single crossing or base point column of $\gamma$ across $h_1,\ldots,h_n$.

\begin{center}
	\begin{figure}[h!]
		\centering
		\includegraphics[scale=0.75]{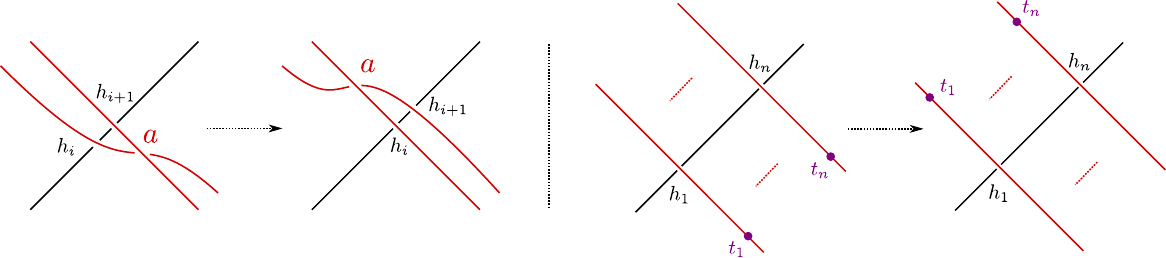}
		\caption{Pushing a crossing (left) or a column of base points (right) across a group of satellite crossings.}
		\label{fig:RIII-monodromy}
	\end{figure}
\end{center}

If we push a crossing $a$ from $\gamma$ across $h_1,\ldots,h_n$ by a single Reidemeister III move as shown in Figure~\ref{fig:RIII-monodromy} (left), then from Section~\ref{sssec:RIII}, the associated isomorphism sends $a \mapsto a$, $h_i \mapsto h_{i+1}$, and $h_{i+1} \mapsto h_i + a h_{i+1}$. (Note that compared to Figure~\ref{fig:RIII}, the crossings $h_i$ and $h_{i+1}$ have switched places after the Reidemeister III move.) This is precisely the matrix map
\[
\left( \begin{matrix} h_i \\ h_{i+1} \end{matrix} \right) \mapsto \left( \begin{matrix} 0 & 1 \\ 1 & a \end{matrix} \right) \left( \begin{matrix} h_i \\ h_{i+1} \end{matrix} \right).
\]
Thus pushing the crossing $a$ across $h_1,\ldots,h_n$ acts on $\left( \begin{smallmatrix} h_1 \\ \vdots \\ h_n \end{smallmatrix} \right)$ by left multiplication by the path matrix for $a$. If instead we push a column of base points across $h_1,\ldots,h_n$ as shown in Figure~\ref{fig:RIII-monodromy} (right), then from Section~\ref{sssec:base-point}, the associated isomorphism sends $h_i$ to $t_i h_i$ for $i \in [1,n]$, which corresponds to left multiplication by the diagonal matrix with diagonal entries $t_1,\ldots,t_n$. Composing the individual isomorphisms, we conclude that the purple-box monodromy indeed acts on  $h_1,\ldots,h_n$ by left multiplication by the path matrix $P_\gamma$, as desired.

\begin{center}
	\begin{figure}[h!]
		\centering
		\includegraphics[scale=0.85]{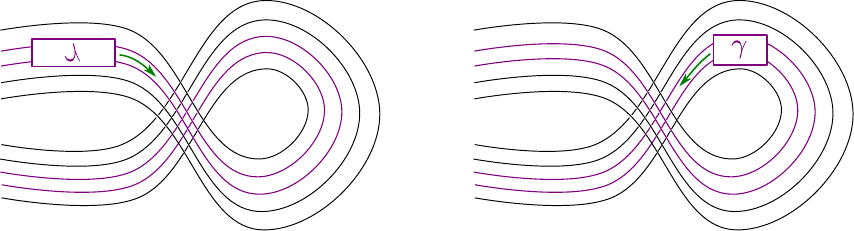}
		\caption{Pushing the purple box through the pigtail.}
\label{fig:deg1}
	\end{figure}
\end{center}

Finally, we consider the degree $1$ crossings. As we perform the $\vartheta$-loop, the purple box passes through the ``pigtail'' region with the degree $1$ crossings twice; see Figure~\ref{fig:deg1}. On the first pass, when the purple box pushes through the strand containing crossings $c_{1j},\ldots,c_{nj}$, this yields the map that sends $\left( \begin{smallmatrix} c_{1j} \\ \vdots \\ c_{nj} \end{smallmatrix} \right) \mapsto P_\gamma \left( \begin{smallmatrix} c_{1j} \\ \vdots \\ c_{nj} \end{smallmatrix} \right)$, by the same argument as for the degree $0$ crossings $h_1,\ldots,h_n$ above. Thus the first pass of the purple box through the degree-$1$ region cumulatively has the effect of sending $C$ to $\widetilde{P}_\gamma C$. Similarly the second pass (from upper right to lower left) sends $C$ to $C \widetilde{P}_\gamma^{-1}$. Together, the two passes send $C$ to $\widetilde{P}_\gamma C \widetilde{P}_\gamma^{-1}$.
\end{proof}

\

\subsection{The K\'alm\'an loop}
\label{ssec:Kalman}
The techniques of this section can be applied to compute the monodromy of other loops of Legendrian links besides $\vartheta$-loops. One case where it is particularly simple to calculate the monodromy in our setting is the loop of Legendrian $T(p,q)$-torus links originally studied by K\'alm\'an in \cite{Kalman}. For concreteness we focus here on the most basic example of the K\'alm\'an loop, which involves the max-tb Legendrian right handed trefoil ($p=2,q=3$). K\'alm\'an constructs a loop in the space of these Legendrian trefoils and proves that the induced action on the degree-$0$ Legendrian contact homology has order $5$. Here we reinterpret this result in our setting.

In our language, the Legendrian trefoil is the $(-1)$-closure of the admissible $2$-stranded braid $\sigma_1^5$; in other words, it is $\La(\beta,1;\gamma)\sse(\R^3,\xi_\st)$ where $\beta\in\Br_1^+$ is the 1-stranded braid and $\gamma = \sigma_1^5 \in \Br_2^+$. We label the crossings of $\gamma$ $a_1,\ldots,a_5$ and place base points $t_1,t_2$ to the right of $\gamma$, as shown in Figure~\ref{fig:trefoil} (left). The $\vartheta$-loop moves the entire braid $\sigma_1^5$ around the standard unknot $\La(\beta)=U$ until it returns to its starting point. We can factor this loop as the fifth power of another loop $\delta$, which moves the single leftmost crossing of $\sigma_1^5$ around the unknot until it returns to $\gamma$ as the rightmost crossing. Note that this move shifts the position of the base points $t_1,t_2$; we then slide $t_1,t_2$ along the knot until they return to their original positions. See Figure~\ref{fig:trefoil}. The combination of the crossing move and the base point move forms a loop beginning and ending at $\La(\beta,1;\gamma)$, which is the K\'alm\'an loop and which we denote by $\delta$.

\begin{center}
	\begin{figure}[h!]
		\centering
		\includegraphics[scale=1.25]{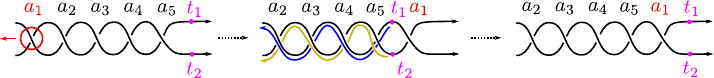}
		\caption{The Legendrian trefoil is the $(-1)$-closure of the depicted braid $\sigma_1^5$. Left, the braiding region with base points; middle, the result of moving the leftmost crossing to the right; right, the result of sliding the base points back to their original position (with the slides shown in the middle diagram).}
		\label{fig:trefoil}
	\end{figure}
\end{center}

The action $\SA(\delta)$ of the loop $\delta$ on the Legendrian contact DGA $\SA(\La(\beta,1;\gamma))$ is easy to describe. Moving $a_1$ to the right simply permutes the $a_i$: $a_1\mapsto a_2, \ldots, a_4\mapsto a_5, a_5\mapsto a_1$. From Section~\ref{sssec:base-point}, sliding the base points as indicated in Figure~\ref{fig:trefoil} fixes $a_1$ and sends $a_i$ to $t_1^{-1}a_it_2$ for $i=3,5$ and $t_2^{-1}a_it_1$ for $i=2,4$. Thus $\SA(\delta)$ acts on the DGA as follows:
\begin{align*}
a_1 &\mapsto t_2^{-1}a_2t_1 &
a_2 &\mapsto t_1^{-1}a_3t_2 &
a_3 &\mapsto t_2^{-1}a_4t_1 &
a_4 &\mapsto t_1^{-1}a_5t_2 &
a_5 &\mapsto a_1.
\end{align*}
(The degree $1$ generators are fixed by $\SA(\delta)$.) By inspection we see that $\SA(\delta)$ has order $5$, in agreement with K\'alm\'an's result: $\SA(\delta)^5 = \SA(\vartheta)$ is the identity map.

This argument readily generalizes to $(p,q)$-torus links for arbitrary positive $p,q$. In the general case, the Legendrian link is the $(-1)$-closure of the admissible braid $(\sigma_1\ldots\sigma_{p-1})^{p+q} \in \Br_p^+$, and the K\'alm\'an loop $\delta$ moves the leftmost $p-1$ crossings around the unknot. As in the case of the trefoil, we immediately see that $\delta^{p+q}$ acts as the identity on the Legendrian contact DGA.

\begin{remark}
The original proof in \cite{Kalman} that $\SA(\delta)$ has order $p+q$ uses the Legendrian link given by the resolution of the rainbow closure of the braid $(\sigma_1\cdots \sigma_{p-1})^q$. The DGA for this link has $(p-1)q$ generators in degree $0$, and K\'alm\'an's computation of the monodromy of $\delta$ on this DGA is rather nontrivial, both because Reidemeister II moves are involved and because the DGA differential itself is quite complicated due to the presence of non-embedded disks. K\'alm\'an then performs an intricate computation to show that this monodromy has order $p+q$. The mere fact that $p+q$ appears here, e.g. instead of $p$ or $q$, is rather mysterious from the geometric viewpoint. By contrast, in our setup with $(-1)$-closures, the DGA has $(p-1)(p+q)$ generators in degree $0$, the monodromy of $\delta$ simply cyclically permutes these generators, and it is evident without computation that this action has order $p+q$.\hfill$\Box$
\end{remark}

\subsection{The saddle cobordism map for $(-1)$-closures}
\label{ssec:alg-resolution}

So far we have discussed the $\vartheta$-monodromy. We now turn to the other principal computational ingredient in our calculations for the upcoming Section~\ref{sec:Monodromy}, namely
the calculation of saddle cobordism maps: we will consider augmentations corresponding to specific decomposable fillings, and these augmentations are the composition of a number of saddle maps.

In Section~\ref{ssec:EHK-map}, we defined a combinatorial cobordism map $\Phi^\comb:\thinspace \SA_{\Lambda_+} \to \SA_{\Lambda_-}$ associated to a saddle cobordism at any proper contractible Reeb chord. This combinatorial formula allows us in Section~\ref{sec:Monodromy} to calculate the augmentations corresponding to particular fillings of $(-1)$-closures, and the reader may skip ahead to that section at this point. In the present subsection, we take a slight detour and discuss what the formula for $\Phi^\comb$ looks like for saddle cobordisms of $(-1)$-closures, in terms of the matrix formula for the DGA of a $(-1)$-closure from Section~\ref{ssec:DGAClosure}. In particular, this will allow us to see combinatorially that $\Phi^\comb$ is indeed a chain map in this case, without going through the general theory.
The interested reader may want to compare our discussion here with \cite[section 3.3]{GSW}, which presents an independent but rather similar matrix treatment of saddle cobordism maps.

Consider a saddle cobordism whose top end is a Legendrian $(-1)$-closure $\Lambda_+ = \Lambda(\sigma_{k_1}\cdots\sigma_{k_r})$, and whose bottom end is the Legendrian link $\Lambda_-$ obtained by resolving a contractible proper crossing of $\Lambda_+$. For ease of notation, we will assume that the crossing is $a_1$, corresponding to the braid generator $\sigma_{k_1}$, and so $\Lambda_- = \Lambda(\sigma_{k_2}\cdots\sigma_{k_r})$. (The case of a saddle resolving an arbitrary crossing $a_\ell$ is easy to deduce from this; just perform the cyclic-permutation isotopy sending $\Lambda_+=\Lambda(\sigma_{k_1}\cdots\sigma_{k_r})$ to $\Lambda(\sigma_{k_\ell}\cdots\sigma_{k_r}\sigma_{k_1}\cdots\sigma_{k_{\ell-1}})$ and similarly for $\Lambda_-$.)

From Section~\ref{ssec:DGAClosure} above, we can write down the differentials $\dd_\pm$ on $\Lambda_\pm$ in matrix form. Specifically,
as in Section~\ref{ssec:DGAClosure}, we place base points $t_{\ell,i}$, $1\leq i\leq n$, $1\leq \ell\leq r$, next to the crossings of $\Lambda_+$. Then $\Lambda_-$ inherits this same array of base points, along with two new base points in place of the crossing $a_1$, one on strand $k_1+1$ labeled by $s_1$ and one on strand $k_1$ labeled by $-s_1^{-1}$. By Proposition~\ref{prop:differential}, in the notation from Section~\ref{ssec:DGAClosure}, the differentials $\dd_+$ and $\dd_-$ for the DGAs of $\Lambda_+$ and $\Lambda_-$ are given by the matrix formulas:
\begin{align*}
\dd_+(C) &= \mathbf{1} + P_{k_1}(a_1) D(\mathbf{t}_1) P_{k_2}(a_2) D(\mathbf{t}_2) \cdots P_{k_r}(a_r) D(\mathbf{t}_r) \\
\dd_-(C) &= \mathbf{1} + D(\mathbf{t}_0)D(\mathbf{t}_1) P_{k_2}(a_2) D(\mathbf{t}_2) \cdots P_{k_r}(a_r) D(\mathbf{t}_r)
\end{align*}
where $\mathbf{t}_0=(1,\ldots,1,-s_1^{-1},s_1,1,\ldots,1)$ (with $-s_1^{-1}$ and $s_1$ in the $k_1$ and $k_1+1$ components respectively). 

Let $\Phi^\comb = \Phi^\leftarrow \circ \Phi^\rightarrow \circ \Phi^0 :\thinspace \SA_{\Lambda_+} \to \SA_{\Lambda_-}$ be the cobordism map from Proposition~\ref{prop:EHK}. We first note that the action of $\Phi^\comb$ on degree-$1$ Reeb chords $c_{ij}$ is easy to write down. Indeed, write $T_{k_1}^\leftarrow(s_1)$ for the $n\times n$ matrix equal to the identity
matrix except with $(k_1,k_1+1)$ entry given by $s_1^{-1}$. Then we have
\begin{equation}
\Phi^\comb(C) =  T_{k_1}^\leftarrow(s_1) C ( T_{k_1}^\leftarrow(s_1))^{-1}.
\label{eq:PhiC}
\end{equation}
This can be seen directly from an inspection of Figure~\ref{fig:closure}, using the fact that $\Delta_{a_1}^\rightarrow(c_{ij}) = \emptyset$, while the only possible disks in $\Delta_{a_1}^\leftarrow(c_{ij})$ are thin disks heading left from their $+$ corner at $a_1$, following the figure eight, and ending in the region containing the $c_{ij}$'s. 
We omit the details here.

The explicit nature of this algebraic model allows us to sketch a direct argument for why $\Phi^\comb$ is a chain map. Note that this argument is mainly provided for context and is not needed in the rest of the paper,\footnote{The computation in the proof of Proposition~\ref{prop:chainmap-matrix} does contribute to the implementation of the program \cite{Ng-program}, in the code calculating the augmentation associated to a filling of a $(-1)$-closure.} and so we do not provide full details; see also \cite[section 3.3]{GSW} for a related discussion with more details.

\begin{prop}
$\Phi^\comb \circ \dd_+ = \dd_- \circ \Phi^\comb$. \label{prop:chainmap-matrix}
\end{prop} 

\begin{proof} In order to show that $\Phi^\comb$ is a chain map, it suffices to show that $\Phi^\comb(\dd_+(C)) = \dd_-(\Phi^\comb(C))$.
Note that
\[
\Phi^\comb(P_{k_1}(a_1)) = P_{k_1}(s_1) = T_{k_1}^{\leftarrow}(s_1) D(\mathbf{t}_0) T_{k_1}^\rightarrow(s_1)
\]
where $T_{k_1}^\rightarrow(s_1)$ is the identity matrix except with $(k_1+1,k_1)$ entry given by $s_1^{-1}$. Since $\Phi$ acts on $C$ by conjugation by $T_{k_1}^\leftarrow(s_1)$, showing that $\Phi$ is a chain map reduces to verifying the following:
\begin{equation}
\Phi^\comb\left(T_{k_1}^\rightarrow(s_1)D(\mathbf{t}_1) P_{k_2}(a_2) D(\mathbf{t}_2) \cdots P_{k_r}(a_r) D(\mathbf{t}_r)T_{k_1}^\leftarrow(s_1)\right) \\
= D(\mathbf{t}_1) P_{k_2}(a_2) D(\mathbf{t}_2) \cdots P_{k_r}(a_r) D(\mathbf{t}_r).
\label{eq:chainmap}
\end{equation}

Call a matrix \textit{lower-unipotent} if it is of the form $\mathbf{1}+N$ where $N$ is a strictly lower triangular matrix; that is, a lower-unipotent matrix is a lower triangular matrix with $1$'s along the diagonal. Note in particular that $T_{k_1}^\rightarrow(s_1)$ is lower-unipotent. Next we observe that if $T$ is lower-unipotent then
\[
T' = \left(P_{k_\ell}(a_\ell+T_{k_{\ell}+1,k_\ell})\right)^{-1} T P_{k_\ell}(a_\ell)
\]
is again lower-unipotent: this follows from the identity of $2\times 2$ matrices
\[
\left( \begin{matrix} 0 & 1 \\ 1 & a_\ell+T_{k_\ell+1,k_\ell} \end{matrix} \right)^{-1} \left( 
\begin{matrix} 1 & 0 \\ T_{k_\ell+1,k_\ell} & 1 \end{matrix} \right) \left( \begin{matrix} 0 & 1 \\ 1 & a_{\ell} \end{matrix} \right) = 
\left( \begin{matrix} 1 & 0 \\ 0 & 1 \end{matrix} \right).
\]
We can thus inductively define a sequence of lower-unipotent matrices $T_1',T_1,T_2',T_2,\ldots,T_r',T_r$ as follows:
\begin{align*}
T_1' &= T_{k_1}^\rightarrow(s_1), \\
T_\ell &= D(\mathbf{t}_\ell)^{-1} T_\ell' D(\mathbf{t}_\ell), \\
T_\ell' &= \left(P_{k_\ell}(a_\ell+(T_{\ell-1})_{k_{\ell}+1,k_\ell})\right)^{-1} T P_{k_\ell}(a_\ell).
\end{align*}
Then we have
\[
T_{\ell-1} P_{k_\ell}(a_\ell) D(\mathbf{t}_\ell) = P_{k_\ell}(a_\ell+(T_{\ell-1})_{k_{\ell}+1,k_\ell}) D(\mathbf{t}_\ell) T_\ell.
\]
Write 
$x_{\ell}:=(T_{\ell-1})_{k_{\ell}+1,k_\ell}$ for short; we now have
\begin{align*}
T_{k_1}^\rightarrow(s_1) D(\mathbf{t}_1) &P_{k_2}(a_2) D(\mathbf{t}_2) \cdots P_{k_r}(a_r) D(\mathbf{t}_r) \\
&=
D(\mathbf{t}_1) P_{k_2}(a_2+x_2) D(\mathbf{t}_2)
P_{k_3}(a_3+x_3) \cdots P_{k_r}(a_r+x_{r})D(\mathbf{t}_r) T_r.
\end{align*}
The key fact now, whose proof (and precise statement) we omit here, is that the matrices $T_\ell$ have geometric meaning: for $i>j$, the $(i,j)$ entry in $T_\ell$ counts embedded disks whose leftmost end is a positive corner at $a_1$ and whose rightmost end is a vertical line segment connecting strands $i$ and $j$ just to the right of crossing $a_\ell$. (In particular, $T_r = \mathbf{1}$.) Furthermore, the map $\Phi^\rightarrow$ from Section~\ref{ssec:EHK-map} is constructed exactly to satisfy
\[
\Phi^\rightarrow(a_\ell + x_\ell) = a_\ell
\]
for all $\ell = 2,\ldots,r$. As a consequence, we have
\[
\Phi^\rightarrow \left(T_{k_1}^\rightarrow(s_1) D(\mathbf{t}_1) P_{k_2}(a_2) D(\mathbf{t}_2) \cdots P_{k_r}(a_r) D(\mathbf{t}_r)\right) 
= D(\mathbf{t}_1) P_{k_2}(a_2) D(\mathbf{t}_2)
P_{k_3}(a_3) \cdots P_{k_r}(a_r)D(\mathbf{t}_r).
\]

Similarly, $\Phi^\leftarrow$ satisfies
\[
\Phi^\leftarrow \left(D(\mathbf{t}_1) P_{k_2}(a_2) D(\mathbf{t}_2) \cdots P_{k_r}(a_r) D(\mathbf{t}_r) T_{k_1}^\leftarrow(s_1) \right)
= D(\mathbf{t}_1) P_{k_2}(a_2) D(\mathbf{t}_2) P_{k_3}(a_3) \cdots P_{k_r}(a_r)D(\mathbf{t}_r).
\]
Combining this equation and the previous equation now yields \eqref{eq:chainmap}, whence $\Phi^\comb$ is a chain map.
\end{proof}



\section{Proof of Infinitely Many Fillings}\label{sec:Monodromy}

In this section we prove Theorem \ref{thm:main}. First, we describe the scheme of proof that we will use for all the Legendrian links $\La\in\SH$. The cases $\La(\wt D_4),\La_1,\La_2$ and $\La(\beta_{11}),\La(\beta_{12}),\La(\beta_{21})$ are then proven directly using this strategy. The general cases $\La(\wt D_n),\La_n$ are concluded from Proposition \ref{prop:LoopCommutativity} and the proofs we give for the two cases $\La(\wt D_4),\La_2$.


\subsection{The argument}\label{ssec:Argument} Let $\La\sse(\R^3,\xi_\st)$ be a Legendrian link $\La=\La(\beta,i;\gamma)$, $\beta\in\Br_N^+,\gamma\in\Br_M^+$, and consider its $\vartheta$-loop, as introduced in Section \ref{ssec:PurpleBox}. The general structure of our proofs can be described in three steps, as follows:

\begin{itemize}
	\item[(i)] First, choose an ordered sequence of crossings for $\beta$ and $\gamma$ such that resolving these crossings yields an orientable exact Lagrangian filling $L\sse(\R^4,\la_\st)$ of the Legendrian link $\La$.
	\\
	
	\item[(ii)] Second, compute the augmentation $\varepsilon_L:\SA_\La\lr\Z[H_1(L) \oplus \Z^{m-1}]$ associated to the exact Lagrangian filling $L$ (where $m$ is the number of components of $\La$) and the induced maps $\vartheta^k:\SA_\La\lr\SA_\La$, $k\in\N$.
	We note that all crossings chosen in (i) will have the property that their complement contains a half-twist, and consequently they are contractible and proper by Proposition~\ref{prop:halftwist3}.
	Thus we may apply the combinatorial formulas from Section~\ref{ssec:EHK-map} in this step.\\
	
	\item[(iii)] Third, fix a crossing $a$ for the braid word $\beta$ associated to the Legendrian link $\La$, which we consider as one of the generators $a\in\SA_\La$ of the Legendrian contact DGA. Consider the invariant
	$$E(k,a):=\max_{\eta:R\lr\Z}|(\eta\circ\varepsilon_L\circ\vartheta^k)(a)|,\quad k\in\N,$$
	where $R=\Z[H_1(L)\oplus\Z^{m-1}]$ and $\eta:R\lr\Z$ runs over all possible unital ring morphisms. Note that the set of such morphisms is finite, as the first Betti number $b_1(L)$ is finite, and thus $E(k,a)$ is a well-defined maximum over a finite set of integers. Finally, show that $E(k,a)$ is a strictly increasing function of $k\in\N$.
\end{itemize}

The different choices for the Lagrangian filling (and thus the augmentation $\varepsilon_L$) and crossing $a\in\SA_\La$ influence the computation of the invariant $E(k,a)$. Finding the maximum over a set whose cardinality grows exponentially in $l(\beta)+l(\gamma)$ makes brute force computation a difficult (though not unfeasible) route. Thus, particular care must be devoted in choosing the augmentation $\varepsilon_L$ and the crossing $a\in\SA_\La$: we will find crossings $a\in\SA_\La$ and Lagrangian fillings whose augmentations satisfy that $(\varepsilon_L\circ\vartheta^k)(a)$ is a {\it positive} Laurent polynomial in $\Z[H_1(L)]$, for all $k\in\N$, making the invariant $E(k,a)$ readily computable.

\begin{remark} Executing the argument laid out here for specific Legendrian links, including all of the ones that we consider in this section, is readily amenable to calculation by computer. 
\label{rmk:program}
A \textit{Mathematica} notebook that performs the calculations contained in the remainder of this section, and is suitable for calculations for general $(-1)$-closures, is available at the second author's web page \cite{Ng-program}. We will work out the argument for $\widetilde{D}_4$ in detail, without recourse to the computer program, in Section~\ref{ssec:AffineD4} below; we provide fewer details for subsequent computations and refer the reader to the program.
\hfill$\Box$
\end{remark}


\subsection{Augmentations for $\Lambda(\widetilde{D}_4)$}\label{ssec:AffineD4}

We now turn to proving
Theorem \ref{thm:main} for the Legendrian link $\La(\wt D_n)$, $n\geq4$. In this subsection we present the argument for $n=4$; the general $n \geq 4$ case is deduced from this in Section~\ref{ssec:generalcase}.

As stated in the introduction, the Legendrian link $\La(\wt D_4)\sse(\R^3,\xi_\st)\sse (\S^3,\xi_\st)$ is defined to be the rainbow closure of the positive braid $(\sigma_2 \sigma_1 \sigma_3 \sigma_2)^2$, which is also the
$(-1)$-closure of the braid $(\sigma_2 \sigma_1 \sigma_3 \sigma_2)^4 \sigma_3^2 \sigma_1^2 = (\sigma_2 \sigma_1 \sigma_3 \sigma_2)^2\Delta_4^2$. A Lagrangian projection for $\La(\wt D_4)$ is depicted in Figure \ref{fig:Proof_AffineD4Braid}. Let us prove the following result:

\begin{thm}[The $\widetilde{D}_4$--Legendrian]\label{thm:AffineD4} 
	Let $\vartheta:\S^1\lr \SL(\La(\wt D_4))$ be the purple-box Legendrian loop. Then there exists a Lagrangian filling $L\sse\bD^4$ of $\La(\wt D_4)$ such that the $\vartheta$-orbit of the system of augmentations $\varepsilon_L$ is entire.
\end{thm}

In order to prove Theorem~\ref{thm:AffineD4}, we set some notation and lay out the pieces that go into the proof.
	Let us label the crossings of the positive braid $(\sigma_2 \sigma_1 \sigma_3 \sigma_2)^4 \sigma_3^2 \sigma_1^2$ from left to right as
	$$a_{16},\ldots,a_1,a_{17},a_{18},a_{19},a_{20}.$$
	Figure \ref{fig:Proof_AffineD4Braid} shows the Lagrangian projection of $\La(\widetilde{D}_4)$ that we use for the proof, where the labeled crossings are also depicted. These crossings are the degree-$0$ Reeb chords of a Legendrian front for $\Lambda(\widetilde{D}_4)$. The $\vartheta$-monodromy is obtained by carrying around the purple box containing the two crossings $a_{19},a_{20}$ and the two base points $t_1,t_2$, as shown in
 Figure~\ref{fig:Proof_AffineD4Braid}, cf.\ Figure \ref{fig:BraidsIntroLoops}. 
 	\begin{center}
		\begin{figure}[h!]
			\centering
			\includegraphics[scale=0.7]{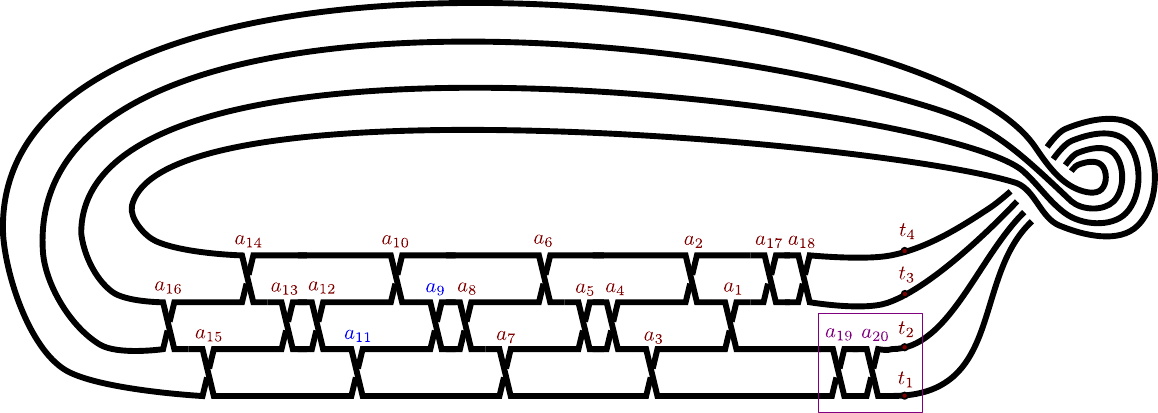}
			\caption{Lagrangian projection for the Legendrian link $\La(\wt D_4)$, as used in the proof of Theorem \ref{thm:AffineD4}. The crossings $a_{11},a_9$, in blue, are used to detect the infinite order of the $\vartheta$-monodromy. In this case, the $\vartheta$-monodromy is obtained by moving the crossings $a_{19},a_{20}$ around this projection.}
			\label{fig:Proof_AffineD4Braid}
		\end{figure}
	\end{center}
	
The filling $L$ of $\La(\wt D_4)$ that we will consider is the decomposable filling constructed as follows. Resolve the following crossings of $\La(\wt D_4)$ in order:	
	$$a_9,a_{10},a_{11},a_{12},a_{13},a_{14},a_{15},a_{16}.$$
Note that at each step the remaining braid is admissible in the sense of Definition~\ref{def:NestedPigTail}: this follows from Proposition~\ref{prop:halftwist} and the fact that the crossings $a_4,a_3,a_2,a_1,a_{17},a_{19}$ comprise a half-twist. Thus each step produces a legal Lagrangian projection of a Legendrian link, and each resolved crossing is contractible. The result of resolving these $8$ crossings is the $(-1)$-closure of a full positive twist, which we write as $\Lambda_0$ and is precisely the standard $4$-component Legendrian unlink. We then fill in each of the $4$ component unknots. This gives the desired filling $L$ of $\La(\wt D_4)$, expressed as $8$ saddle cobordisms and $4$ minimum cobordisms.

Following the discussion in Section~\ref{ssec:system}, we use $\Z[t_1^{\pm 1},\ldots,t_4^{\pm 1},s_9^{\pm 1},\ldots,s_{16}^{\pm 1}]$ as the coefficient ring for the DGA $\SA(\La(\wt D_4))$. 
We will need two maps on $\SA(\Lambda(\widetilde{D}_4))$, induced by the $\vartheta$-monodromy and the filling $L$. The former map is an automorphism $\vartheta:\thinspace\SA(\La(\wt D_4)) \to \SA(\La(\wt D_4))$. For the latter, as in Section~\ref{ssec:system}, $L$ induces an augmentation 
$\varepsilon_L:\thinspace \SA(\La(\wt D_4))\to R$. Here $R = \Z[t_1^{\pm 1},\ldots,t_4^{\pm 1},s_9^{\pm 1},\ldots,s_{16}^{\pm 1}]/(w_1=w_2=w_3=w_4=-1)$, where $w_1,w_2,w_3,w_4$ are the product of the labels of the base points on each unknot in $\Lambda_0$. Since by inspection $t_1,t_2,t_3,t_4$ appear on all distinct components of $\Lambda_0$, the quotient allows us to solve for the $t_i$'s, and we conclude that
$R \cong \Z[s_9^{\pm 1},\ldots,s_{16}^{\pm 1}]$.

Our aim is to pairwise distinguish the iterates $\varepsilon_L\circ \vartheta^k :\thinspace \SA(\La(\wt D_4)) \to R$, $k\in\N$, even up to automorphisms of $R$. We will do this by computing the image of $a_9$ under each of these maps. In order to perform this computation, we need to partially compute the maps $\vartheta$ and $\varepsilon_L$.

We first consider the monodromy automorphism $\vartheta$, which we compute using Proposition~\ref{prop:purple-monodromy}. First note that $\vartheta$ fixes the variables $a_{19},a_{20},t_1,t_2$ that appear inside the purple box. We will be interested in what $\vartheta$ does to the two Reeb chords $a_9,a_{11}$, which are depicted in blue in Figure~\ref{fig:Proof_AffineD4Braid}. The path matrix associated to the purple box is given by
	$$M = \left(\begin{array}{cc}
	0 & 1\\
	1 & a_{19}
	\end{array}\right)\cdot
	\left(\begin{array}{cc}
	0 & 1\\
	1 & a_{20}
	\end{array}\right)\cdot
	\left(\begin{array}{cc}
	t_1 & 0\\
	0 & t_2
	\end{array}\right)= \left( \begin{matrix}
	t_1 & t_2a_{20} \\ t_1a_{19} & t_2(1+a_{19}a_{20})
	\end{matrix} \right).
	$$
	By Proposition \ref{prop:purple-monodromy}, the effect of the DGA automorphism $\vartheta\in\Aut(\SA(\La(\wt D_4)))$ on the two crossings $a_{11},a_9$, which are depicted in blue in Figure \ref{fig:Proof_AffineD4Braid}, is
	$$	\left( \begin{matrix} a_{11} \\ a_9 \end{matrix} \right) \longmapsto \vartheta\left( \begin{matrix} a_{11} \\ a_9 \end{matrix} \right)=M \left(  \begin{matrix} a_{11} \\ a_9 \end{matrix} \right).$$

Next consider the augmentation $\varepsilon_L$, which we can explicitly compute using the formulas from Section~\ref{sec:elementary}. We will only need the following partial computation:

\begin{lemma}
We have
\label{lem:AffineD4}
$\varepsilon_L(a_9) = s_9$, $\varepsilon_L(a_{11}) = s_{11}$, and
\begin{align*}
\varepsilon_L(t_1) &=  -s_{11}s_{15}, \\
\varepsilon_L(t_2) &= -\frac{s_9s_{12}s_{13}s_{16}}{s_{11}s_{15}}, \\
\varepsilon_L(a_{19}) &=
\frac{s_{9}}{s_{11}}-\frac{s_{9} s_{12} s_{13}}{s_{11}^2 s_{15}}, \\
\varepsilon_L(a_{20}) &=  -\frac{s_{11}^3s_{15}^2}{s_{9}^2 s_{12} s_{13} s_{16}}+\frac{s_{11}^2s_{15}}{s_{9} s_{12} s_{13}}-\frac{s_{11}^2s_{15}^2}{s_{9} s_{10} s_{13} s_{16}}+\frac{s_{11}^2s_{14} s_{15}^2}{s_{9} s_{12} s_{13}^2 s_{16}}.
\end{align*}
\end{lemma}

\begin{proof}
For $i=9,\ldots,16$, let $\Phi_i = \Phi^\comb_{L_{a_i}}$ denote the combinatorial cobordism map associated to the saddle cobordism at $a_i$, as described in Section~\ref{ssec:EHK-map}; also let $\varepsilon_0 :\thinspace \SA_{\Lambda_0} \to R \cong \Z[s_9^{\pm 1},\ldots,s_{16}^{\pm 1}]$ denote the augmentation associated to the disk filling of $\Lambda_0$. We have
\[
\varepsilon_L = \varepsilon_0 \circ \Phi_{16} \circ \cdots \circ \Phi_9.
\]

We begin by computing $\e_0$. Note that all Reeb chords of $\Lambda_0$ either have degree $1$ (for the $4^2$ crossings on the right) or connect different components of $\Lambda_0$ (for the crossings $a_i$ for $1\leq i\leq 8$ and $17\leq i\leq 20$). Since the filling of $\Lambda_0$ consists of four disjoint disks, it follows that $\e_0$ sends all Reeb chords to $0$. As for the $t_i$ parameters, an inspection of Figure~\ref{fig:Proof_AffineD4Braid} yields that the unknot components of $\Lambda_0$ containing $t_1$ and $t_2$ contain the following base points in order:
$-s_{15}^{-1},-s_{11}^{-1},t_1$ and $-s_{16}^{-1},s_{15},-s_{13}^{-1},-s_{12}^{-1},s_{11},-s_9^{-1},t_2$, respectively. Setting each of the products of these base points equal to $-1$ gives $t_1 = -s_{11}s_{15}$ and $t_2 = -\frac{s_9s_{12}s_{13}s_{16}}{s_{11}s_{15}}$, and these are the respective images of $t_1$ and $t_2$ under $\e_0$ (and thus under $\e_L$ as well).

We now proceed to compute $\e_L$ for $a_9,a_{11},a_{19},a_{20}$. The sequence of saddle moves has been chosen to simplify the computation of $\varepsilon_L(a_9)$ and $\varepsilon_L(a_{11})$: indeed, $\varepsilon_L(a_9) = \Phi_9(a_9) = s_9$, while $\Phi_{9}$ and $\Phi_{10}$ fix $a_{11}$ and so $\varepsilon_L(a_{11}) = \Phi_{11}(a_{11}) = s_{11}$. 

	\begin{center}
		\begin{figure}[h!]
			\centering
									\includegraphics[scale=1.2]{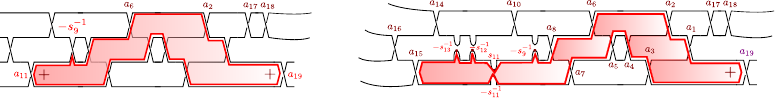}
			\caption{Disks with positive corners at $a_{11},a_{19}$ (left) and $a_{15},a_{19}$ (right), contributing to $\Phi_{11}(a_{19})$ and $\Phi_{15}(a_{19})$ respectively. For these disks, the positive corner at $a_{19}$ has positive orientation sign, while the positive corner at $a_{11}$ and $a_{15}$ has negative orientation sign.
			}
			\label{fig:AffineD4comp}
		\end{figure}
	\end{center}
	
For $a_{19}$, we keep track of disks with two positive punctures, one at $a_{19}$ and one at the crossing being resolved. There are no such disks when we resolve $a_9$ and $a_{10}$. When we resolve $a_{11}$, there is one disk $\Delta \in \Delta_{a_{11}}^\rightarrow(a_{19})$ passing through $-s_9^{-1}$ (against the orientation) with no negative corners; see Figure~\ref{fig:AffineD4comp}. From Definition~\ref{def:cob-comb}, we read off $\Phi^\rightarrow(a_{19}) = a_{19} +s_9s_{11}^{-1}$ and $\Phi^\leftarrow(a_{19}) = a_{19}$, and so $\Phi_{11}(a_{19}) = a_{19} + s_9s_{11}^{-1}$. As we successively resolve $a_{12},\ldots,a_{16}$, the only additional relevant disk with two positive punctures comes when we resolve $a_{15}$ and is shown in Figure~\ref{fig:AffineD4comp}; this gives
$\Phi_{15}(a_{19}) = a_{19} - s_9s_{11}^{-1}s_{12}s_{13}s_{15}^{-1}s_{11}^{-1}$. We conclude that
\[
\e_L(a_{19}) = \e_0(\Phi_{15}(\Phi_{11}(a_{19})) = \e_0\left(a_{19}+\frac{s_{9}}{s_{11}}-\frac{s_{9} s_{12} s_{13}}{s_{11}^2 s_{15}}\right) = \frac{s_{9}}{s_{11}}-\frac{s_{9} s_{12} s_{13}}{s_{11}^2 s_{15}}.
\]

The computation of $\e_L(a_{20})$ is similar but slightly more involved. We compute that
\begin{align*}
\Phi_9(a_{20}) &= a_{20}-t_1a_{11}s_9^{-1}t_2^{-1} \\
\Phi_{10}(a_{20}) &= a_{20}-t_1a_{12}s_{10}^{-1}t_2^{-1} \\
\Phi_{13}(a_{20}) &= a_{20} + t_1s_{13}^{-1}a_{14}t_2^{-1} \\
\Phi_{15}(a_{20}) &= a_{20}+t_1s_{15}^{-1}a_{16}t_2^{-1};
\end{align*}
piecing these together, along with $\Phi_i(a_i) = s_i$ and the values computed above for $\e_0(t_1)$ and $\e_0(t_2)$, gives the desired expression for $\e_L(a_{20})$.
\end{proof}

We are now in position to prove Theorem~\ref{thm:AffineD4}.

\begin{proof}[Proof of Theorem~\ref{thm:AffineD4}]
	Consider the following matrices with entries in $\Z[s_9^{\pm 1},\ldots,s_{16}^{\pm 1}]$:
	$$M_0:=\varepsilon_L(M),\quad
	v_0 :=  \varepsilon_L  \left( \begin{matrix} a_{11} \\ a_9 \end{matrix} \right)=\left( \begin{matrix} s_{11} \\ s_9 \end{matrix} \right),\quad
	N:=\left( \begin{matrix} s_{11} & 1 \\ s_9 & 0\end{matrix} \right),\quad
	M_1:= N^{-1} M_0 N.$$
For $k\in\N$, the augmentation $\e_L \circ \vartheta^k$ sends the column vector $\left( \begin{smallmatrix} a_{11} \\ a_9 \end{smallmatrix} \right)$ to
	$$
	\varepsilon_L(M)^k\cdot\varepsilon_L\left( \begin{matrix} a_{11} \\ a_9 \end{matrix} \right)=M_0^k v_0 
	= N(N^{-1}M_0N)^k \cdot \left( \begin{matrix} 1 \\ 0 \end{matrix} \right)= NM_1^k \left( \begin{matrix} 1 \\ 0 \end{matrix} \right).
	$$
	
We can explicitly write down $M_1$ using Lemma~\ref{lem:AffineD4}.  This leads to the following observation: if we replace
	$s_{11},s_{12},s_{15},s_{16}$ by their negatives $-s_{11},-s_{12},-s_{15},-s_{16}$, the matrix $M_1$ becomes
	$$(M_1)|_{\{s_{j}\to-s_{j},j=11,12,15,16\}}=\left( \begin{matrix} m_{11} & m_{12} \\ m_{21} & m_{22}\end{matrix} \right),$$
	where the entries are
	\begin{align*}
	m_{11} &= \frac{s_9 s_{13} s_{12}^2}{s_{10} s_{11}}+\frac{s_9 s_{14} s_{12}}{s_{11}}+\frac{s_9 s_{15} s_{12}}{s_{10}}+\frac{s_9 s_{14} s_{15}}{s_{13}}+s_9 s_{16} \\
	m_{12} &= \frac{s_{12} s_{13}}{s_{11}}+s_{15} \\
	m_{21} &=
	\frac{s_9 s_{13} s_{12}^2}{s_{10}}+s_9 s_{14} s_{12} \\
	m_{22} &= s_{12} s_{13}.
	\end{align*}	
	Note that all the coefficients are {\it positive} Laurent polynomials in the variables $s_9,\ldots,s_{16}$: this is the algebraic reason for our choice of augmentation $\varepsilon_L$, and the change of signs for the variables $s_{11},s_{12},s_{15},s_{16}$. 
	
	Let us now finally conclude that the iterates $\varepsilon_L\circ\vartheta^k$ are pairwise distinct. We do this by studying the quantity
	
	$$E(k,a_9):=\max_{\eta:R\to\Z}|(\eta\circ\varepsilon_L\circ\vartheta^k)(a_9)|,$$
	
	where $\eta:R\to\Z$ runs over all possible $2^8$ unital ring morphisms. This is an integer-valued invariant of an augmentation $\varepsilon_L:\SA_\La\to R$ even up to post-composition of an automorphism of $R$. That is, if $E(k,a_9)\neq E(l,a_9)$ then there exists no automorphism $\varphi\in\Aut(R)$ such that $\varphi(\varepsilon_L\circ\vartheta^k)=\varepsilon_L\circ\vartheta^l$, and thus the $k$-th and $l$-th $\vartheta$-iterates of $\varepsilon_L$ are distinct. In order to compute $E(k,a_9)$, we note that
	$$	|(\varepsilon_L \circ \vartheta^k)(a_9)| = \left| \left(\begin{matrix} s_9 & 0 \end{matrix} \right) M_1^k  \left( \begin{matrix} 1 \\ 0 \end{matrix} \right) \right|=\left| \left(\begin{matrix} 1 & 0 \end{matrix} \right) M_1^k  \left( \begin{matrix} 1 \\ 0 \end{matrix} \right) \right|
	$$
	is the absolute value of the upper-left entry of $M_1^k$. A unital ring morphism $\eta:R\to\Z$ is uniquely determined by specifying the values $s_9,\ldots,s_{16} \in \{\pm 1\}$ and since the entries $m_{11},m_{12},m_{21},m_{22}$ are positive Laurent polynomials, the value $|(\varepsilon_L \circ \vartheta^k)(a_9)|$ is maximized when $s_i=-1$ for $i=11,12,15,16$ and $s_i = 1$ for $i=9,10,13,14$. It follows that $E(k,a_9)$ is equal to the upper-left entry of $\left( \begin{smallmatrix} 5 & 2 \\ 2 & 1 \end{smallmatrix}\right)^k$, which is
a strictly increasing function of $k$. This proves that $E(k,a_9)\neq E(l,a_9)$ if $k\neq l$, as required.
\end{proof}

\subsection{Three Variations on the Affine $D_4$-braid}\label{ssec:VariationsAffineD4}

Let us next consider the following three Legendrian links from the Introduction:
$$\La(\beta_{12})=\La((\sigma_1\sigma_2\sigma_2\sigma_1)^2\sigma_1,1;\sigma_1^2),\quad \La(\beta_{21})=\La((\sigma_1\sigma_2\sigma_2\sigma_1)^2\sigma_1^2,1;\sigma_1),$$
$$\La(\beta_{11})=\La((\sigma_1\sigma_2\sigma_2\sigma_1)^2\sigma_1,1;\sigma_1).$$
These are obtained from the $\wt D_4$-braid by removing the crossing $a_{18}$, for $\beta_{12}$, the crossing $a_{20}$, for $\beta_{21}$ or the two crossings $a_{18},a_{20}$, for the braid $\beta_{11}$. See Figure \ref{fig:Proof_AffineD4Braid} for the notation on the crossings, we denote the crossings of these three braids by the same labels\footnote{That is, the crossing $a_i$ for the $\wt D_4$-braid is still denoted $a_i$ for the braids $\beta_{ij}$, $1\leq i,j\leq2$, where $\beta_{22}$ is precisely the $\wt D_4$-braid.} as in Figure \ref{fig:Proof_AffineD4Braid}. In these three cases, we can use the template given by the proof of Theorem \ref{thm:AffineD4}, again by studying the crossings $a_9,a_{11}$. We will omit the details and just give the choice of Lagrangian filling $L$, its corresponding augmentation $\e_L$ as computed from the formulas in Section~\ref{ssec:EHK-map}, and the augmented matrices $M_1$. These computations are also contained in the \textit{Mathematica} notebook \cite{Ng-program}.

\begin{itemize}
	\item[-] {\bf The link $\La(\beta_{12})$}. The Lagrangian filling $L$ is obtained by resolving the crossings $a_9,a_{10},a_{11},a_{12},a_{13},a_{15},a_{16}$ in order. The augmentation $\varepsilon_L$ sends 
	\[
\begin{gathered}
t_1\to -s_{11} s_{15},t_2\to -\frac{s_9 s_{12} s_{13} s_{16}}{s_{11} s_{15}},
	a_9 \to s_9, a_{11} \to s_{11}, a_{19}\to \frac{s_9}{s_{11}}+\frac{s_{12} s_{13} s_9}{s_{11}^2 s_{15}},\\
a_{20}\to -\frac{s_{15}^2 s_{11}^3}{s_9^2 s_{12} s_{13} s_{16}}-\frac{s_{15} s_{11}^2}{s_9 s_{12} s_{13}}-\frac{s_{15}^2 s_{11}^2}{s_9 s_{10} s_{13} s_{16}}.
	\end{gathered}
	\]
	The augmented matrix $M_1=N^{-1}M_0N$ satisfies
	$$M_1|_{\{s_{j}\to-s_{j},j=11,12,15,16\}} = \left(
	\begin{array}{cc}
	\frac{s_9 s_{13} s_{12}^2}{s_{10} s_{11}}+\frac{s_9 s_{15} s_{12}}{s_{10}}+s_9 s_{16} & \frac{s_{12} s_{13}}{s_{11}}+s_{15} \\
	\frac{s_9 s_{12}^2 s_{13}}{s_{10}} & s_{12} s_{13} \\
	\end{array}
	\right),$$
	whose entries are all positive Laurent polynomials. \\
	
	\item[-] {\bf The link $\La(\beta_{21})$}. The Lagrangian filling $L$ is obtained by resolving the crossings $a_9,a_{10},a_{11},a_{12},a_{13},a_{14},a_{16}$ in order. This augmentation $\varepsilon_L$ sends 
\[
t_1\to \frac{s_9 s_{12} s_{13}}{s_{11}},t_2\to -s_{11} s_{16},
a_9\to s_9,a_{11}\to s_{11},
a_{19}\to \frac{s_9 s_{13} s_{12}^2}{s_{10} s_{11}^2 s_{16}}+\frac{s_{13} s_{12}}{s_{11} s_{16}}-\frac{s_9 s_{14} s_{12}}{s_{11}^2 s_{16}}+\frac{s_9}{s_{11}}.
\]
The augmented matrix $M_1=N^{-1}M_0N$ satisfies
	$$M_1|_{\{s_{j}\to-s_{j},j=11,12,16\}} =\left(
	\begin{array}{cc}
	\frac{s_9 s_{13} s_{12}^2}{s_{10} s_{11}}+\frac{s_9 s_{14} s_{12}}{s_{11}}+s_9 s_{16} & \frac{s_{12} s_{13}}{s_{11}} \\
	\frac{s_9 s_{13} s_{12}^2}{s_{10}}+s_9 s_{14} s_{12} & s_{12} s_{13} \\
	\end{array}
	\right),$$
	whose entries are all positive Laurent polynomials. \\
	
	\item[-] {\bf The link $\La(\beta_{11})$}. The Lagrangian filling $L$ is obtained by resolving the crossings $a_9,a_{10},a_{11},a_{12},a_{13},a_{16}$ in order. This augmentation $\varepsilon_L$ sends 
	\[
	t_1\to \frac{s_9 s_{12} s_{13}}{s_{11}},t_2\to -s_{11} s_{16},a_9\to s_9,a_{11}\to s_{11},
	a_{19}\to \frac{s_9 s_{13} s_{12}^2}{s_{10} s_{11}^2 s_{16}}+\frac{s_{13} s_{12}}{s_{11} s_{16}}+\frac{s_9}{s_{11}}.
	\]
	The augmented matrix $M_1=N^{-1}M_0N$ satisfies
	$$M_1|_{\{s_{j}\to-s_{j},j=11,12,16\}} = \left(
	\begin{array}{cc}
	\frac{s_9 s_{13} s_{12}^2}{s_{10} s_{11}}+s_9 s_{16} & \frac{s_{12} s_{13}}{s_{11}} \\
	\frac{s_9 s_{12}^2 s_{13}}{s_{10}} & s_{12} s_{13} \\
	\end{array}
	\right),$$
whose entries are all positive Laurent polynomials. \\
\end{itemize}

This completes the proof of Theorem~\ref{thm:main} for the Legendrian links $\La(\beta_{11}),\La(\beta_{12}),\La(\beta_{21})$.
We emphasize that these links all have a stabilized component (or two, in the case of $\La(\beta_{11})$). In particular, these Legendrian links are not the rainbow closure of a positive braid, and our Floer-theoretic argument is presently the only known argument that shows the existence of infinitely many Lagrangian fillings for these Legendrian links.

\begin{center}
	\begin{figure}[h!]
		\centering
		\includegraphics[scale=0.8]{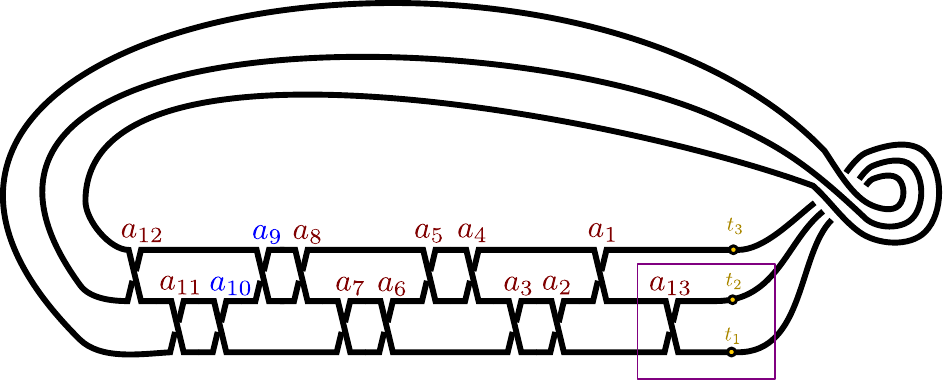}
		\caption{Lagrangian projection for the Legendrian link $\La_1$, as used in the proof of Theorem \ref{thm:Lambda1}. In this proof, the crossings $a_{10},a_9$, in blue, are used to detect the infinite order of the $\vartheta$-monodromy. The $\vartheta$-monodromy is obtained by moving the purple box around this projection. To construct $\Lambda_2$ instead, add an additional crossing labeled $a_{14}$ to the purple box, between $a_{13}$ and the base points $t_1,t_2$.}
		\label{fig:Proof_FamilyN1Case}
	\end{figure}
\end{center}

\subsection{Monodromy for the Braids $\La_1$ and $\La_2$} We next prove Theorem \ref{thm:main} for the Legendrian links $\La_1$ and $\La_2$ from the Introduction. Figure \ref{fig:Proof_FamilyN1Case} depicts a Lagrangian projection of $\La_1$; $\La_2$ comes from adding one additional crossing to the purple box.

\begin{thm}[The $\La_1$-- and $\La_2$--Legendrians]\label{thm:Lambda1} Let $\La_1$ and $\La_2$ be the $(-1)$-closures of the $3$-braids $(\sigma_2\sigma_1\sigma_1\sigma_2)^3\sigma_1$ and
$(\sigma_2\sigma_1\sigma_1\sigma_2)^3\sigma_1^2$, respectively.
	Let $\vartheta:\S^1\lr \SL(\La_1)$ be the purple-box Legendrian loop. Then for $n=1,2$, there exists a Lagrangian filling $L\sse\bD^4$ of $\La_n$ such that the $\vartheta$-orbit of the system of augmentations $\varepsilon_L$ is entire.
\end{thm}

\begin{proof}
As in Section~\ref{ssec:VariationsAffineD4}, this follows the proof of Theorem~\ref{thm:AffineD4} and we will simply specify the fillings and describe the corresponding augmentations and augmented matrices. The computation of the augmentations can be found in \cite{Ng-program}.

We begin with $\La_1$, whose Lagrangian projection is shown in Figure \ref{fig:Proof_FamilyN1Case}.
	We choose the filling $L$ given by resolving the crossings 
	$$a_{10},a_9,a_8,a_7,a_6,a_5,a_{12}$$
	in order. (As usual, this produces an unlink, and we then fill in each of the $3$ unknot components to complete the construction of $L$.)
	The augmentation $\e_L$ sends
	\[
	\begin{gathered}
		t_1\to \frac{s_5 s_8 s_9}{s_6 s_7 s_{10}},t_2\to -s_6 s_7 s_{10} s_{12},a_9\to s_9,a_{10}\to s_{10},\\
		a_{13} \to \frac{s_5}{s_6}-\frac{s_7}{s_8}+\frac{s_9}{s_{10}}+\frac{s_5 s_8}{s_6^2 s_7^2 s_{10} s_{12}}-\frac{s_5}{s_6^2 s_7^2 s_9 s_{10} s_{12}}+\frac{1}{s_6 s_7 s_8 s_9 s_{10} s_{12}}.
		\end{gathered}
		\]
Define $M_0 = \e_L \left( \begin{matrix} 0 & t_2 \\ t_1 & a_{13}t_2 \end{matrix} \right)$, 
$N = \e_L \left( \begin{matrix} a_{10} & 1 \\ a_9 & 0 \end{matrix} \right)$, and $M_1 = N^{-1} M_0 N$; then
$$M_1|_{\{s_{j}\to-s_{j},j=5,7,8,10\}} = \left( \begin{matrix} m_{11} & m_{12} \\ m_{21} & m_{22}\end{matrix} \right),$$ and the entries
\begin{align*}
	m_{11}&=\frac{s_6 s_{10} s_{12} s_7^2}{s_8}+s_6 s_9 s_{12} s_7+s_5 s_{10} s_{12} s_7+\frac{s_5}{s_6 s_7 s_9}+\frac{1}{s_8 s_9}\\
	m_{12}&=\frac{s_5 s_8}{s_6 s_7 s_{10}}\\
	m_{21}&=s_5 s_7 s_{12} s_{10}^2+\frac{s_6 s_7^2 s_{12} s_{10}^2}{s_8}+\frac{s_5 s_{10}}{s_6 s_7 s_9}+\frac{s_{10}}{s_8 s_9}\\
	m_{22}&=\frac{s_5 s_8}{s_6 s_7}
	\end{align*}
are all positive Laurent polynomials.

For $\La_2$, we start with the diagram for $\La_1$ in Figure~\ref{fig:Proof_FamilyN1Case}, and add one more crossing labeled $a_{14}$ directly to the right of $a_{13}$. Choose the filling $L$ of $\La_2$ given by resolving the crossings 
$$a_{10},a_9,a_8,a_7,a_6,a_5,a_{12},a_{11}$$
	in order. The augmentation $\e_L$ sends
	\[
	\begin{gathered}
	t_1\to -s_6 s_7 s_{10} s_{11},t_2\to -\frac{s_5 s_8 s_9 s_{12}}{s_6 s_7 s_{10} s_{11}},
	a_9 \to s_9, a_{10}\to s_{10},
	a_{13}\to \frac{s_5}{s_6}-\frac{s_7}{s_8}+\frac{s_9}{s_{10}}-\frac{s_5 s_8 s_9}{s_6^2 s_7^2 s_{10}^2 s_{11}}, \\
	a_{14} \to \frac{s_6^2 s_7^2 s_{10}^2 s_{11}}{s_5 s_8 s_9}
	-\frac{s_6^3 s_7^3 s_{10}^3 s_{11}^2}{s_5^2 s_8^3 s_9^3 s_{12}}
	+\frac{s_6^2 s_7^2 s_{10}^3 s_{11}^2}{s_5 s_8^2 s_9^3 s_{12}}
	-\frac{s_6^2 s_7^2 s_{10}^3 s_{11}^2}{s_5 s_8 s_9^2 s_{12}}.
	\end{gathered}
	\]
	Define $M_0 = \e_L \left( \begin{matrix} t_1 & a_{14} t_2 \\a_{13} t_1 & (1+a_{13}a_{14})t_2 \end{matrix} \right)$, 
$N = \e_L \left( \begin{matrix} a_{10} & 1 \\ a_9 & 0 \end{matrix} \right)$, and $M_1 = N^{-1} M_0 N$; then
$M_1|_{\{s_{j}\to-s_{j},j=5,7,8,10\}} = \left( \begin{matrix} m_{11} & m_{12} \\ m_{21} & m_{22}\end{matrix} \right)$, and the entries
	\begin{align*}
	m_{11}&=\frac{s_6^2 s_{10}^2 s_{11} s_7^3}{s_5 s_8^3 s_9^2}+\frac{2 s_6 s_{10}^2 s_{11} s_7^2}{s_8^2 s_9^2}+\frac{s_6^2 s_{10} s_{11} s_7^2}{s_5 s_8^2 s_9}+\frac{s_6 s_{10} s_{12} s_7^2}{s_8}+\frac{s_5 s_{10}^2 s_{11} s_7}{s_8 s_9^2}+\frac{s_6 s_{10} s_{11} s_7}{s_8 s_9}+\\
	&+s_6 s_9 s_{12} s_7+s_5 s_{10} s_{12} s_7+\frac{s_5}{s_6 s_7 s_9}+\frac{1}{s_8 s_9},\\
	m_{12}&=\frac{s_6 s_{10} s_{11} s_7^2}{s_8 s_9}+s_6 s_{11} s_7+\frac{s_5 s_{10} s_{11} s_7}{s_9}+\frac{s_5 s_8}{s_6 s_7 s_{10}},\\
	m_{21}&=\frac{s_5 s_7 s_{11} s_{10}^3}{s_8 s_9^2}+\frac{2 s_6 s_7^2 s_{11} s_{10}^3}{s_8^2 s_9^2}+\frac{s_6^2 s_7^3 s_{11} s_{10}^3}{s_5 s_8^3 s_9^2}+s_5 s_7 s_{12} s_{10}^2+\frac{s_6 s_7^2 s_{12} s_{10}^2}{s_8}+\frac{s_5 s_{10}}{s_6 s_7 s_9}+\frac{s_{10}}{s_8 s_9},\\
	m_{22}&=\frac{s_5 s_7 s_{11} s_{10}^2}{s_9}+\frac{s_6 s_7^2 s_{11} s_{10}^2}{s_8 s_9}+\frac{s_5 s_8}{s_6 s_7}.
	\end{align*}
	are all positive Laurent polynomials.
	\end{proof}

\begin{center}
	\begin{figure}[h!]
		\centering
		\includegraphics[scale=0.8]{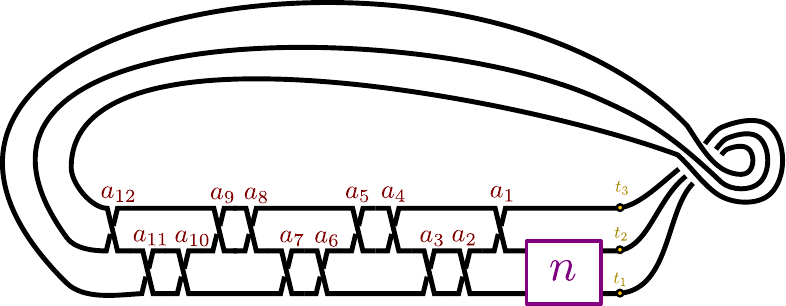}
		\caption{Lagrangian projection for the Legendrian link $\La_n$, $n\geq1$. 
		}
		\label{fig:Proof_FamilyNGeneralCase}
	\end{figure}
\end{center}

\begin{center}
	\begin{figure}[h!]
		\centering
		\includegraphics[scale=0.7]{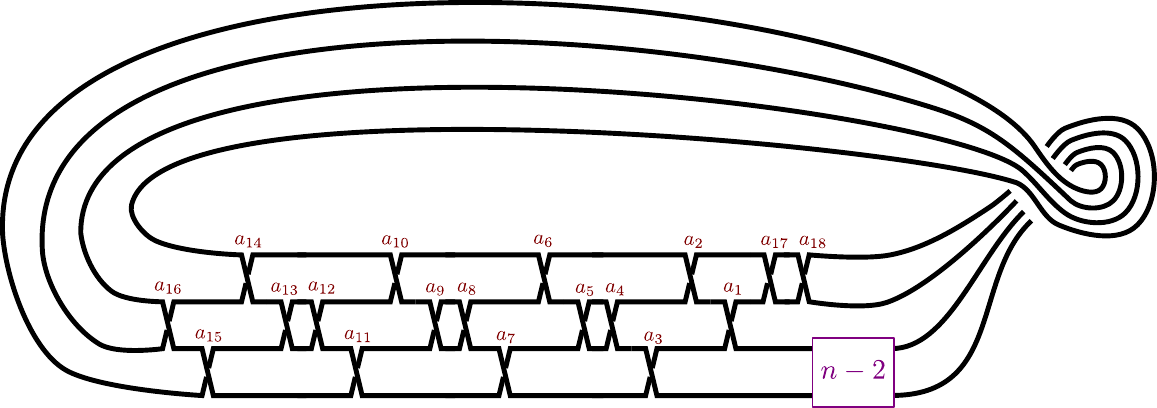}
		\caption{Lagrangian projection for the Legendrian link $\La(\wt D_n)$, $n\geq4$, as used in the proof of Theorem \ref{thm:AffineD4}. 
		}
		\label{fig:Proof_AffineDnBraid}
	\end{figure}
\end{center}

\subsection{The general case: 
$\La(\wt D_n)$ and $\La_n$.} 
\label{ssec:generalcase}
We now turn to the Legendrian links $\La(\wt D_n)$, $n\geq5$, and $\La_n$, $n\geq3$, which are depicted in Figures \ref{fig:Proof_FamilyNGeneralCase} and \ref{fig:Proof_AffineDnBraid}.
The action of the $\vartheta$-loops on the Legendrian contact DGA for $\La(\wt D_n)$, respectively $\La_n$,  can be studied directly thanks to our understanding of the $\vartheta$-loops for the Legendrian braids $\La(\wt D_4)$, respectively $\La_2$. The main ingredient that allows us to deduce the general cases from a particular case is the following:

\begin{prop}\label{prop:LoopCommutativity} Let $\beta\in\Br_N^+$ be an admissible braid, let $k_1,k_2\in\N$ with $2\leq k_1\leq k_2$, and let $L$ be 
an exact Lagrangian filling of $\La(\beta,1;\sigma_1^{k_1})$. Consider an exact Lagrangian cobordism $\Sigma$ from $\La(\beta,1;\sigma_1^{k_1})$ to $\La(\beta,1;\sigma_1^{k_2})$
	obtained by resolving any combination of $(k_2-k_1)$ crossings in the braid $\sigma_1^{k_2}$ which are not the initial nor the final crossings. Let
	$$\Phi_\Sigma:\SA(\La(\beta,1;\sigma_1^{k_2}))\lr\SA(\La(\beta,1;\sigma_1^{k_1}))$$
	denote the induced map between the Legendrian contact DGAs. Let $a$ denote any Reeb chord of $\La(\beta,1;\sigma_1^{k_1})$ not in $\sigma_1^{k_1}$, as well as the corresponding Reeb chord of $\La(\beta,1;\sigma_1^{k_2})$. Then for all $m\in\N$, we have
	$$(\varepsilon_L\circ\vartheta_1^m)(a)=(\varepsilon_L\circ \Phi_{\Sigma}\circ\vartheta_2^m)(a),$$
	where $\vartheta_i$ denotes the $\vartheta$-loop of $\La(\beta,1;\sigma_1^{k_i})$, $i=1,2$.
	
	Consequently, if the $\vartheta_1$-orbit of the augmentation $\varepsilon_L$ of $\La(\beta,1;\sigma_1^{k_2})$ is entire, then the $\vartheta_2$-orbit of the augmentation $\varepsilon_{L \circ \Sigma} = \varepsilon_L \circ \Phi_\Sigma$ of $\La(\beta,1;\sigma_1^{k_2})$ is entire.
\end{prop}

\begin{proof}
We begin by noting that $f_\Sigma$ fixes any Reeb chord of $\La(\beta,1;\sigma_1^{k_2})$ outside of the braid $\sigma_1^{k_2}$. This is because $f_\Sigma$ consists of a composition of saddle cobordism maps that count disks with two positive corners, and the only such disks with a positive corner at one of the resolved crossings must have its other positive corner at a crossing in the braid, by our assumption that we do not resolve the two extreme crossings of $\sigma_1^{k_2}$.

There are two types of crossings in $\La(\beta,1;\sigma_1^{k_2})$ besides the crossings in $\sigma_1^{k_2}$: the ones that come from crossings of $\La(\beta)$ involving the satellited strand of $\beta$, and the ones that do not. If $a$ is of the latter type, then $\vartheta_2$ and $\vartheta_1$ both fix $a$. Since $f_\Sigma(a) = a$, we are done in this case.

Now assume that $a$ is of the former type, and note that crossings of this type come in pairs corresponding to the two strands of $\sigma_1^{k_2}$. Recall that the action of the $\vartheta$-monodromy on such a pair of crossings is completely determined by the path matrix of the braid. If we write $P_1$ and $P_2$ for the path matrices for the braids $\sigma_1^{k_1}$ and $\sigma_1^{k_2}$ respectively, then it suffices to show that
\[
P_1 =  f_\Sigma(P_2).
\]
Note that these path matrices incorporate all base points in the braid region; in particular, the braid $\sigma_1^{k_1}$ includes base points in its interior, coming from the resolved crossings of $\sigma_1^{k_2}$.

\begin{center}
	\begin{figure}[h!]
		\centering
								\includegraphics[scale=1.2]{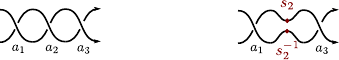}
		\caption{Resolving a crossing in $\sigma_1^3$ to produce $\sigma_1^2$.}
		\label{fig:PathMatrix}
	\end{figure}
\end{center}

To prove $P_1 = f_\Sigma(P_2)$, by functoriality we may assume that $\Sigma$ consists of a single saddle cobordism. Furthermore, since $f_\Sigma$ fixes any crossing besides the two crossings adjacent to the saddle, it suffices to check the equality when $k_1 = 3$ and $k_2 = 2$; see Figure~\ref{fig:PathMatrix}. In this case, if $a_1,a_2,a_3$ denote the crossings in $\sigma_{k_1}$ as shown in Figure~\ref{fig:PathMatrix}, we have $f_\Sigma(a_1) = a_1-s_2^{-1}$, $f_\Sigma(a_2) = s_2$, $f_\Sigma(a_3) = a_3-s_2^{-1}$, and we compute:
\[
f_\Sigma(P_2) = f_\Sigma\left(\left( \begin{matrix} 0 & 1 \\ 1 & a_1 \end{matrix} \right)
\left( \begin{matrix} 0 & 1 \\ 1 & a_2 \end{matrix} \right)
\left( \begin{matrix} 0 & 1 \\ 1 & a_3 \end{matrix} \right) \right)
= \left( \begin{matrix} 0 & 1 \\ 1 & a_1 \end{matrix} \right)
\left( \begin{matrix} -s_2^{-1} & 0 \\ 0 & s_2 \end{matrix} \right)
\left( \begin{matrix} 0 & 1 \\ 1 & a_3 \end{matrix} \right)
= P_1,
\]
as desired.

The final sentence of the proposition follows from the fact that $\Phi_\Sigma$ is surjective; see the proof of Proposition~\ref{prop:aug-infinite-condition} below.
\end{proof}

From Proposition~\ref{prop:LoopCommutativity}, and the fact that $\Lambda(\widetilde{D}_4)$ and $\Lambda_2$ have fillings for which the $\vartheta$-orbit of the associated augmentation is entire, it follows that the same is true for $\Lambda(\widetilde{D}_n)$, $n \geq 5$, and $\Lambda_n$, $n \geq 3$. This completes the proof of Theorem~\ref{thm:main}.

\begin{remark}
A consequence is that the Legendrian links $\Lambda(\widetilde{D}_n)$, $n \geq 5$, and $\Lambda_n$, $n \geq 3$ have infinitely many fillings. This conclusion also follows from just the existence of a cobordism from $\Lambda(\widetilde{D}_4)$ or $\Lambda_2$ to these links, using Proposition~\ref{prop:aug-infinite-condition} below. However, Theorem~\ref{thm:main} is stronger: we have actually constructed an infinite family of fillings of each of these links that are all provably distinct from each other.\hfill$\Box$
\end{remark}


\section{Proof of Corollaries and Concluding Remarks}\label{sec:CorRmks}

In this section we discuss some of the applications stated in the Introduction. First, we show that the smooth isotopy type of the Lagrangian fillings we construct is independent of the iteration of the $\vartheta$-loop. Then, we precisely state the notion of {\it aug-infinite} Legendrians (which implies the existence of infinitely many fillings) and prove some of its properties under exact Lagrangian cobordisms. We also conclude Proposition \ref{prop:linksample}, providing a gamut of small smooth knots with a max-tb Legendrian representative that admits infinitely many Lagrangian fillings. Finally, we prove Corollaries \ref{cor:Stein1} and \ref{cor:Stein2} regarding closed Lagrangians surfaces in certain Weinstein 4-manifolds.


\subsection{Smooth isotopy class of Lagrangian fillings} 
\label{ssec:smooth}
Let $L\sse(\D^4,\la_\st)$ be an exact Lagrangian filling of $\La\sse(\S^3,\xi_\st)$ and $\vartheta:\S^1\lr\SL(\La)$ a Legendrian loop. The isotopy cobordism $gr(\vartheta)\sse\S^3\times[0,1]$ associated to the Legendrian loop $\vartheta$ is an exact Lagrangian self-concordance of $\La$, which we can concatenate with $L$. This yields another exact Lagrangian filling $L_\vartheta=L \# gr(\vartheta)$ of $\La$. Theorem \ref{thm:main} shows that, for certain $\vartheta$, $L_\vartheta$ may {\it not} be Hamiltonian isotopic to $L$. For the Legendrian $\vartheta$-loops we use in this article, let us prove that $L_\vartheta$ is always {\it smoothly} isotopic to $L$. The argument is the same as in (the updated version of) \cite{CasalsHonghao}; we reproduce it here for convenience:

\begin{prop}
	Let $\Lambda\sse(\S^3,\xi_\st)$ be a Legendrian link of the form $\La=\La(\beta,i;\gamma)$, where $\beta\in\Br^+_N,\gamma\in\Br^+_M$. Let $L_\pi\sse(\D^4,\la_\st)$ be an exact Lagrangian filling obtained by a pinching sequence $\pi\in S_{|\beta|}$, and $\vartheta:\S^1\lr\SL(\La)$ a Legendrian $\vartheta$-loop. Then, the exact Lagrangian fillings $L$ and $L_\vartheta$ are smoothly isotopic relative to their boundary $\La$.
	\label{prop:smooth}
\end{prop}

\begin{proof}
	From the perspective of a positive braid representative $\beta\in\Br^+_N$ of $\La=\La(\beta)$, a Legendrian $\vartheta$-loop consists of two moves: Reidemeister III moves and conjugations. Let us denote the ordered crossings of $\beta$ by $(a_j)$, $j\in[1,|\beta|]$, with $a_{\pi(k)}$ being the $k$-th crossing to be resolved. First, {\it any} pinching (resolution) sequence $\pi\in S_{|\beta|}$ yields a surface which is smoothly isotopic to $L$. From the smooth perspective, resolving a crossing corresponds to an elementary surface cobordism of index 1 and thus two consecutive such cobordisms can be performed in either order without affecting the smooth type, see Figure \ref{fig:Prop71_1} below. Since any two different pinching sequences differ by a composition of transpositions, the smooth isotopy class of $L_\pi$ is equal for {\it any} pinching sequence $\pi\in S_{|\beta|}$. It thus suffices to consider the case of the identity permutation $\pi=e$, and show that $L_e\cup_\La gr(\vartheta)$ is smoothly isotopic to $L_e$.
	
		\begin{center}
	\begin{figure}[h!]
		\centering
		\includegraphics[scale=1.2]{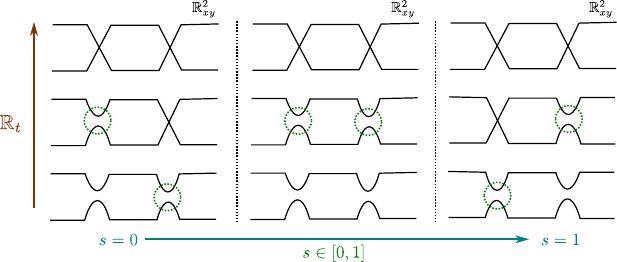}
		\caption{A compactly supported smooth isotopy between two (local) exact Lagrangian cobordisms which are not Hamiltonian isotopic. The coordinate $t\in\R$ represents the symplectization direction $\R_t\times R^3_{x,y,z}$ and the diagrams are Lagrangian projections in $\R^2_{x,y}$ as indicated. The variable $s\in[0,1]$ is the real coordinate associated to the isotopy itself.}
		\label{fig:Prop71_1}
	\end{figure}
\end{center}
	
	Consider a Reidemeister III move for three (consecutive) crossings $a_{i-1},a_{i},a_{i+1}$, which leads to $a_{i+1},a_{i},a_{i-1}$. For the Lagrangian filling $L_e$, these three crossings are resolved left to right: starting at $a_{i-1}$, then $a_i$ and $a_{i+1}$, in this order. Starting at $a_{i+1},a_{i},a_{i-1}$ we can describe {\it two} smooth cobordisms, both local to this piece of the braid (constant relative to its endpoints):
	
	\begin{itemize}
		\item[(i)] Apply a Reidemeister III move down to $a_{i-1},a_{i},a_{i+1}$ and then resolve according to $\pi=e$. Namely, first $a_{i-1}$, then $a_i$ and finally $a_{i+1}$.
		\item[(ii)] Directly resolve the three crossings $a_{i+1},a_{i},a_{i-1}$, using the transposition $\pi=(i,i+2)$. That is, we resolve $a_{i+1}$ first, then $a_i$ and lastly $a_{i-1}$.
	\end{itemize}
	
	These cobordisms are depicted in Figure \ref{fig:Prop71_2}. Both tangle cobordisms start at the tangle $a_{i+1},a_{i},a_{i-1}$ and end up in the trivial 3-stranded tangle. Since the crossings $a_{i-1}$ and $a_{i+1}$ are interchanged in a Reidemeister III move ($a_{i-1}$ before being geometrically the same as $a_{i+1}$), as are $a_{i+1},a_{i-1}$, the two tangle cobordisms are smoothly isotopic. Hence, concatenating with the graph of an isotopy given by a sequence of Reidemeister III moves, from $\La$ to itself, does not affect the smooth isotopy type of a Lagrangian filling $L$.
	
			\begin{center}
		\begin{figure}[h!]
			\centering
			\includegraphics[scale=1.2]{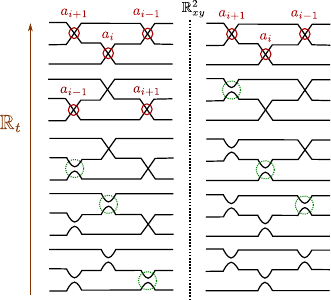}
			\caption{The two (local) exact Lagrangian cobordisms associated to a Reidemeister III move: the cobordism in Item (i) in the text is depicted on the left, whereas the cobordism in Item (ii) is drawn on the right. These cobordisms are to be compared smoothly, relative to their common ends.}
			\label{fig:Prop71_2}
		\end{figure}
	\end{center}
	
	The same occurs for conjugation of the given positive braid $\beta=(a_1,a_2,\ldots,a_{|\beta|-1},a_{|\beta|})$. Indeed, there are two smooth concordances starting with $(a_{|\beta|},a_1,a_2,\ldots,a_{|\beta|-1})$:
	
	\begin{itemize}
		\item[(i)] Apply the cyclic shift from $(a_{|\beta|},a_1,a_2,\ldots,a_{|\beta|-1})$ down to $(a_1,a_2,\ldots,a_{|\beta|-1},a_{|\beta|})$ and then resolve the crossings starting at $a_{|\beta|}$ and then left to right, that is, continuing with $a_1,a_2$ and resolving through $a_{|\beta|-1}$.
		\item[(ii)] Directly resolve the crossings of $(a_{|\beta|},a_1,a_2,\ldots,a_{|\beta|-1})$ according to $\pi=e$: starting at $a_{|\beta|}$, then $a_1,a_2$ through $a_{|\beta|-1}$.
	\end{itemize}
	
	These two concordances yield smoothly isotopic surfaces. In conclusion, starting with the Lagrangian filling $L_e\sse(\D^4,\la_\st)$, the concatenation $L_\vartheta=L\cup_\La gr(\vartheta)$ yields a Lagrangian filling of the form $L_\pi$, for a permutation $\pi\in S_{|\beta|}$. Since $L_\pi$ and $L_e$ are smoothly isotopic, the required statement follows.
\end{proof}

\subsection{Aug-infinite Legendrian links and cobordisms}
\label{ssec:aug-infinite}

Here we describe a method for starting with one Legendrian link known to have infinitely many fillings and producing others. First we need to define a condition that implies having infinitely many fillings and is in turn implied in our examples by the $\vartheta$-orbit being entire.

Suppose that $\Lambda$ is a Legendrian link with a (connected, orientable, exact Lagrangian) filling $L$ of Maslov number $0$. As discussed in Section~\ref{ssec:geom-cob}, $L$ induces a $(2g+2m-2)$-system of augmentations
$\varepsilon_L :\thinspace \SA_\Lambda \to \Z[s_1^{\pm 1},\ldots,s_{2g+2m-2}^{\pm 1}]$, where $g$ is the genus of $L$ and $m$ is the number of components of $\Lambda$. Furthermore, up to equivalence (automorphism of $\Z[s_1^{\pm 1},\ldots,s_{2g+2m-2}^{\pm 1}]$), this system is well-defined, independent of choices, and invariant under Hamiltonian isotopy of $L$.
Here, as in Section~\ref{ssec:dga-def}, we have assumed in defining the DGA $(\SA_\Lambda,\dd)$ that there is one base point on each component of $\Lambda$.

There are $2^{2g+2m-2}$ ring morphisms from $\Z[s_1^{\pm 1},\ldots,s_{2g+2m-2}^{\pm 1}]$ to $\Z$, each sending each $s_i$ to $\pm 1$. By composing $\varepsilon_L$ with these homomorphisms, we obtain $2^{2g+2m-2}$ augmentations from $\SA_\Lambda$ to $\Z$. In this way, the filling $L$ of $\Lambda$ induces finitely many $\Z$-valued augmentations $(\SA_\Lambda,\dd) \to (\Z,0)$.
Note that this continues to hold even if $L$ is not connected: the augmentations induced by a disconnected filling of $\Lambda$ necessarily annihilate any Reeb chord of $\Lambda$ whose endpoints lie on different components of the filling, and each component of the filling induces finitely many augmentations of the sublink of $\Lambda$ given by the boundary of the component.

\begin{definition}
A Legendrian link $\Lambda$ is \textit{aug-infinite} if the collection of all $\Z$-valued augmentations $(\SA_\Lambda,\dd) \to (\Z,0)$ induced by orientable exact Lagrangian fillings of $\Lambda$ of Maslov number $0$, ranging over all possible such fillings, is infinite.
\end{definition}

\noindent
Note that the aug-infinite condition is independent of the choices made along the way, including spin structure, capping paths and operators, and placement of base points. Adding extra base points also does not affect the condition; cf.\ the proof of Proposition~\ref{prop:aug-infinite-condition} below.

The following is an immediate consequence of the fact that each filling induces finitely many $\Z$-valued augmentations.

\begin{prop}\label{prop:aug-infinite}
If $\Lambda$ is aug-infinite then it has infinitely many fillings.
\end{prop}

In the conclusion of Proposition \ref{prop:aug-infinite}, {\it infinitely} many Lagrangian fillings refers to the fact that there are infinitely many Lagrangian fillings up to Hamiltonian isotopy. A prior, they might not be smoothly isotopic. Nevertheless, as proven in Proposition \ref{prop:smooth}, this is the case for the Lagrangian fillings we construct with $\vartheta$-loops.

Next we observe that our arguments from Section~\ref{sec:Monodromy} actually prove that the Legendrians $\Lambda(\widetilde{D}_n)$, $\Lambda_n$, and $\Lambda(\widetilde{A}_{2})$ satisfy this strengthened version of having infinitely many fillings.

\begin{prop}
The three classes of Legendrian links $\Lambda(\widetilde{D}_n)$ $(n\geq 4)$, $\Lambda_n$ $(n \geq 1)$, and $\La(\beta_{11})=\Lambda(\widetilde{A}_{2}),\La(\beta_{12}),\La(\beta_{21})$
are aug-infinite.
\label{prop:aug-infinite-ex}
\end{prop}

Now we claim that for particular decomposable Lagrangian cobordisms, if the bottom of the cobordism is aug-infinite, then the top is as well. To be precise, we have the following.

\begin{prop}
Let $\Lambda_+$ and $\Lambda_-$ be Legendrian links with rotation number $0$, and suppose that $\Lambda_-$ is an aug-infinite Legendrian link. Suppose that the following two properties hold:
\begin{itemize}
\item[-]
The $xy$ projection $\Pi_{xy}(\Lambda_-)$ is obtained from $\Pi_{xy}(\Lambda_+)$ by a sequence of saddle cobordisms at proper contractible Reeb chords of $\Lambda_+$ of degree $0$;
\item[-]
All Reeb chords of $\Lambda_-$ (and thus of $\Lambda_+$) are in nonnegative degree.
\end{itemize}
Then the Legendrian link $\Lambda_+$ is aug-infinite.
\label{prop:aug-infinite-condition}
\end{prop}

\begin{proof}
It suffices to consider the case where $\Lambda_+$ and $\Lambda_-$ are related by a single saddle move at a Reeb chord $a$ of $\Lambda_+$. Suppose that $\Lambda_+$ has $m$ components, and place a base point on each; these base points trace down to $\Lambda_-$. As in Section~\ref{ssec:system}, we place a pair of base points on $\Lambda_-$ coming from the saddle at $a$. Then both $\SA_{\Lambda_+}$ and $\SA_{\Lambda_-}$ are DGAs over $\Z[t_1^{\pm 1},\ldots,t_m^{\pm 1},s^{\pm 1}]$, and the cobordism gives a map $\Phi :\thinspace (\SA_{\Lambda_+},\dd) \to (\SA_{\Lambda_-},\dd)$. 

Any filling of $\Lambda_-$ produces a system of augmentations for $\Lambda_-$ as in the discussion in Section~\ref{ssec:system}; note that now one component of $\Lambda_-$ has more than one base point, but the construction from Section~\ref{ssec:system} works just as well in this case. From Remark~\ref{rmk:multiple}, adding each extra base point has the effect on $(\SA_{\Lambda_-},\dd)$ of replacing one generator $t$ of the coefficient ring by two generators $t',t''$ and setting $t=t't''$ in the differential. It follows that there is a two-to-one correspondence between $\Z$-valued augmentations of $(\SA_{\Lambda_-},\dd)$ after and before the extra base point is added, and so adding extra base points does not affect the aug-infinite condition. We conclude that $(\SA_{\Lambda_-},\dd)$ has infinitely many augmentations coming from fillings. Since there are finitely many choices for the images of $t_1,\ldots,t_m,s$ under such an augmentation, there exist $t_1^0,\ldots,t_m^0,s^0 \in \{\pm 1\}$ such that $(\SA_{\Lambda_-},\dd)$ has infinitely many augmentations from fillings that send $t_i$ to $t_i^0$ and $s$ to $s^0$.

Write $\SA_{\Lambda_+}^\Z$ and $\SA_{\Lambda_-}^\Z$ for the DGAs over $\Z$ obtained by setting $t_i=t_i^0$ and $s=s^0$. The cobordism map $\Phi$ induces a map $\Phi^\Z :\thinspace \SA_{\Lambda_+}^\Z \to \SA_{\Lambda_-}^\Z$ satisfying $\Phi^\Z(a) = \pm 1$ and for all other Reeb chords $a_i$ of $\Lambda_+$, $\Phi^\Z(a_i) = a_i + f(a_i)$ for some $f(a_i) \in \SA_{\Lambda_-}$ determined by the construction in Section~\ref{ssec:EHK-map}. As observed in Section~\ref{ssec:EHK-map}, this map respects the height filtration: for each $i$, $f(a_i)$ only involves Reeb chords of strictly smaller height than $a_i$. We conclude from this that $\Phi_\Z$ is surjective.

Since $\SA_{\Lambda_-}^\Z$ has infinitely many augmentations from fillings, there is some Reeb chord $a_i$ of $\Lambda_-$ that is sent to infinitely many values in $\Z$ under these augmentations. Now use the surjectivity of $\Phi^\Z$ and suppose that $x \in \SA_{\Lambda_+}$ satisfies $\Phi^\Z(x) = a_i$. Each augmentation of $\SA_{\Lambda_-}^\Z$ from a filling of $\Lambda_-$ produces an augmentation of $\SA_{\Lambda_+}^\Z$ from a filling of $\Lambda_+$ by composition with $\Phi^\Z$, and $x$ is sent to infinitely many values in $\Z$ under these augmentations. This shows that $\SA_{\Lambda_+}^\Z$ has infinitely many augmentations from fillings of $\Lambda_+$, and consequently that $\Lambda_+$ is aug-infinite.
\end{proof}

\begin{remark}
It is expected that Proposition~\ref{prop:aug-infinite-condition} should hold whenever there is an exact Lagrangian cobordism between $\Lambda_+$ and $\Lambda_-$, without the restriction of being composed strictly of saddle moves (and not isotopy cylinders) or even of being decomposable. One approach to proving the more general result is to show that exact cobordisms induce injective maps on the augmentation categories of Legendrian links (over $\Z$), in the spirit of previous work of Pan \cite{Pan-AGT} for Legendrian knots and the upcoming paper \cite{CLLMPT} for links.\hfill$\Box$
\end{remark}

As the first application of Proposition~\ref{prop:aug-infinite-condition}, we have Corollary \ref{cor:toruslinks}:

\begin{proof}[Proof of Corollary \ref{cor:toruslinks}] Observe that there is a decomposable Lagrangian cobordism to the Legendrian $(4,4)$ torus link $\Lambda(4,4)$, which is the $(-1)$-closure of the $4$-braid $(\sigma_1\sigma_2\sigma_3)^8 = (\sigma_2\sigma_1\sigma_3\sigma_2)^4\sigma_3^4\sigma_1^4$, from the link $\Lambda(\widetilde{D}_4) = \Lambda((\sigma_2\sigma_1\sigma_3\sigma_2)^4\sigma_3^2\sigma_1^2$, consisting of two saddle cobordisms at proper contractible degree $0$ Reeb chords. Since $\Lambda(\widetilde{D}_4)$ is aug-infinite, it follows that $\Lambda(4,4)$ is aug-infinite as well. In addition, there is an another such cobordism from $\Lambda(4,4)$ to the Legendrian $(n,m)$ torus link $\Lambda(n,m)$ for any $n,m\geq 4$, and so by Proposition~\ref{prop:aug-infinite-condition}, $\Lambda(n,m)$ is aug-infinite for any $n,m\geq 4$. Similarly, we can deduce that the Legendrian $(3,6)$ torus link $\Lambda(3,6)$, which is the $(-1)$-closure of the $3$-braid $(\sigma_1\sigma_2)^9=(\sigma_2\sigma_1^2\sigma_2)^3\sigma_1^6$, is aug-infinite because there is a cobordism to $\Lambda(3,6)$ from $\Lambda_1$; it then also follows that the $(3,m)$-torus link $\Lambda(3,m)$ is aug-infinite for all $m \geq 6$. The proof is complete.
\end{proof}

We can also apply Proposition~\ref{prop:aug-infinite-condition} to show that various other single-component Legendrian knots have infinitely many fillings.

\begin{prop}\label{prop:linksample}
The Legendrian knots given by the $(-1)$-closures of the following positive braids have infinitely many fillings:
\label{prop:knots}
\begin{itemize}
\item[(i)]
$(\sigma_2\sigma_1\sigma_3\sigma_2)^4\sigma_2\sigma_1\sigma_3$, which has smooth type $m(10_{145})$, Thurston--Bennequin number $3$ and genus $2$ fillings,\\
\item[(ii)]
$(\sigma_2\sigma_1\sigma_3\sigma_2)^3\sigma_2^2\sigma_1^2\sigma_3\sigma_2\sigma_1\sigma_3^2 \in B_4$ which has smooth type $10_{154}$, Thurston--Bennequin number $5$ and genus $3$ fillings,\\
\item[(iii)]
$\sigma_2^2\sigma_1^2\sigma_2^2\sigma_1^2\sigma_2^2\sigma_1^2\sigma_2\sigma_1 \in B_3$, which has smooth type $m(10_{161})$, Thurston--Bennequin number $5$ and genus $3$ fillings,\\
\item[(iv)]
$\sigma_2\sigma_1^2\sigma_2^2\sigma_1^2\sigma_2^2\sigma_1^2\sigma_2^2\sigma_1^3 \in B_3$, of smooth type $10_{139}$, Thurston--Bennequin number $7$ and genus $4$ fillings,\\
\item[(v)]
$\sigma_2\sigma_1^2\sigma_2^2\sigma_1^2\sigma_2^2\sigma_1\sigma_2\sigma_1^2\sigma_2^2\sigma_1 \in B_3$, of smooth type $m(10_{152})$, Thurston--Bennequin number $7$ and genus $4$ fillings.
\end{itemize}
\end{prop}

\begin{proof}
The $m(10_{145})$ and $10_{154}$ knots have a cobordism from the link $\La(\beta_{11})$, which is the $(-1)$-closure of $(\sigma_2\sigma_1\sigma_3\sigma_2)^4\sigma_1\sigma_3$. The other three knots have cobordisms from the link $\Lambda_1$, which is the $(-1)$-closure of $\sigma_2\sigma_1^2\sigma_2^2\sigma_1^2\sigma_2^2\sigma_1^2\sigma_2\sigma_1$.
\end{proof}

In light of Proposition~\ref{prop:aug-infinite-condition}, given two Legendrian links $\Lambda_+,\Lambda_-$ with infinitely many fillings, we might consider $\Lambda_-$ to be ``simpler'' than $\Lambda_+$ if there is a saddle cobordism from $\La_-$ to $\Lambda_+$. Since such a cobordism increases Thurston--Bennequin number as we go from bottom to top, a rough measure of the simplicity of a Legendrian link with infinitely many fillings is given by its Thurston--Bennequin number: the lower the $tb$, the simpler the link. (Alternatively, we could use $2g+m$ where $g$ is the genus of a connected filling and $m$ is the number of components of the link, since $tb = 2g+m-2$.) From this perspective, $m(10_{145})$ ($tb=3$) is the simplest knot that is known to us to have infinitely many fillings, while $\Lambda(\widetilde{A}_{2})$ ($tb=2$) is the simplest known link.

\begin{remark} We presently do not know of any Legendrian knots with infinitely many genus $1$ fillings, or of any Legendrian links with infinitely many planar (genus $0$) fillings. From the perspective of cluster algebras, the existence of the former would be somewhat unexpected if we restrict to the class of $(-1)$-closures of admissible braids.\hfill$\Box$\end{remark}

\subsection{Lagrangian surfaces in Weinstein 4-manifolds}
\label{ssec:Stein-surf}

Here we prove Corollaries \ref{cor:Stein1} and \ref{cor:Stein2}. Let $\Lambda\sse(\S^3,\xi_\st)$ be a Legendrian link with $m:=|\pi_0(\La)|$ components, and $W(\La)$ the Weinstein 4-manifold obtained by attaching $m$ Weinstein handles to $(\D^4,\la_\st)$, one along each component of the Legendrian $\Lambda\sse(\S^3,\xi_\st)\cong(\dd\D^4,\ker(\la_\st|_{\dd\D^4}))$. Given an embedded exact Lagrangian filling $L\sse(\D^4,\la_\st)$, we denote by $\overline{L}\sse W(\Lambda)$ the closed embedded exact Lagrangian surface in $W(\Lambda)$ given by the set-theoretic union $\overline{L}:=L\cup L_{cap}$, where $L_{cap}$ is the (disjoint) union of the Lagrangian cores of the $m$ Weinstein handles.

The augmentations $\varepsilon_L:\SA_\La\lr\Z[H_1(L)]$ of the Legendrian contact DGA used in this manuscript employ the system of coefficients $\Z[H_1(L)]$, geometrically keeping track of local systems in a Lagrangian filling $L\sse(\D^4,\la_\st)$. In the transition from $L$ to $\overline{L}\sse W(\Lambda)$, we must compare $\Z[H_1(L)]$ and $\Z[H_1(\overline{L})]$, which are not isomorphic unless $\Lambda$ has a single component. This motivates the following definition.

\begin{definition}\label{def:restricted}
Let $L\sse(\D^4,\la_\st)$ be a filling of a Legendrian link $\Lambda\sse(\S^3,\xi_\st)$, inducing the system of augmentations $\varepsilon_L :\thinspace \SA_\Lambda \to \Z[H_1(L)]$, where $\Lambda$ is equipped with the null-cobordant spin structure. The \textit{restricted system of augmentations} associated to $L$ is the composition
\[
\varepsilon_{\overline{L}} :\thinspace \SA_\Lambda \stackrel{\varepsilon_L}{\longrightarrow} \Z[H_1(L)] \longrightarrow \Z[H_1(\overline{L})],
\]
where the second map is induced by the quotient map $H_1(L) \to H_1(\overline{L})$.\hfill$\Box$
\end{definition}

If we place a single base point $t_i$ on each component $\Lambda_i$ of $\Lambda$, then $t_i$ represents the homology class of $\Lambda_i$ in both $H_1(\Lambda)$ and $H_1(L)$, and the quotient map in Definition~\ref{def:restricted} sends each $t_i$ to $1$ since $\Lambda_i$ is null-homologous in $\overline{L}$. For practical purposes, if $L$ is a connected decomposable filling of an $m$-component link $\Lambda$, we can compute the restricted system of augmentations $\varepsilon_{\overline{L}}$ associated to $L$ as follows.

Let us write $(\SA_\Lambda,\partial)$ for the DGA of $\Lambda$ with the \textit{Lie group} spin structure, which is a DGA over $\Z[t_1^{\pm 1},\ldots,t_m^{\pm 1}]$. Recall from Sections~\ref{ssec:system} and \ref{ssec:agree} the construction of the system of augmentations $\varepsilon_L :\thinspace \SA_\Lambda \to R$ where $R = (\Z[t_1^{\pm 1},\ldots,t_m^{\pm 1},s_1^{\pm 1},\ldots,s_\ell^{\pm 1}])/(w_1=\cdots=w_k=-1)$. We can further quotient the ring $R$ by the relations $t_1=\cdots=t_m=-1$ to get $\overline{R}:= R/(t_1=\cdots=t_m=-1)$: this corresponds to passing from $H_1(L)$ to $H_1(\overline{L})$, where the $-$ sign comes from the fact that we are using the Lie group, rather than the null-cobordant spin structure, on $L$.
From Remark~\ref{rmk:Caitlin}, $t_1\cdots t_m=(-1)^m$ in $R$, and this new quotient imposes $m-1$ new relations. We have $R \cong \Z[H_1(L) \oplus \Z^{m-1}]$ and $\overline{R} \cong \Z[H_1(\overline{L}) \oplus \Z^{m-1}]$; imposing the conditions $t_1=\cdots=t_m=-1$ on the system of augmentations for $L$ produces the restricted system of augmentations $\varepsilon_{\overline{L}}$ for the Lagrangian filling $L$, enhanced by link automorphisms in each case.

\begin{remark} \label{rmk:Caitlin2}
When $\Lambda$ is a single-component Legendrian knot, there is no difference between the system and the {\it restricted} system of augmentations for a filling $L$. This comes from the result of Leverson \cite{Leverson} that any augmentation in this case must necessarily send the unique $t$ variable to $-1$; geometrically, this correlates with the fact that $\Lambda$ is already null-homologous in $L$ before we pass to $\overline{L}$.\hfill$\Box$
\end{remark}

The purpose of restricted systems of augmentations for $L$ is that they correspond to local systems that extend to local systems for the closed exact Lagangian surface  $\overline{L}$. Let $L_1,L_2\sse(\D^4,\la_\st)$ be Lagrangian fillings of a Legendrian link $\Lambda\sse(\S^3,\xi_\st)$. If $L_1,L_2$ are Hamiltonian isotopic, then their associated augmentations are DGA homotopic; see Theorem~\ref{thm:EHK}. As noted in Remark~\ref{rmk:DGAhomotopy}, for the Legendrian links studied in this paper, we can replace ``DGA homotopic'' by a simpler notion.
Following Definition~\ref{def:system}, we define two restricted systems of augmentations $$\varepsilon_{\overline{L}_1}:\SA_{\Lambda}\lr \Z[H_1(\overline{L}_1)],\quad \varepsilon_{\overline{L}_2}:\SA_{\Lambda}\lr \Z[H_1(\overline{L}_2)]$$
to be equivalent if there exists an isomorphism $\psi :\thinspace \Z[H_1(\overline{L}_1)] \to \Z[H_1(\overline{L}_2)]$ such that $$\varepsilon_{\overline{L}_2} = \psi \circ \varepsilon_{\overline{L}_1}.$$
Then for Legendrian links such that the entire DGA $\SA_\La$ is concentrated in nonnegative degree, as is the case for all of the examples in this paper, DGA homotopic (restricted) systems of augmentations are necessarily equivalent (restricted) systems of augmentations.

Now, both Corollaries \ref{cor:Stein1} and \ref{cor:Stein2} will be proven by using the following Proposition \ref{prop:Surgery}, which is not essentially new and uses the recent articles \cite{EkholmLekili19,Ekholm19,GPS2}. We thank T.\ Ekholm, S.\ Ganatra, and Y.\ Lekili for illuminating discussions regarding the proof of Proposition \ref{prop:Surgery}.

\begin{prop}\label{prop:Surgery} Let $\Lambda\sse(\S^3,\xi_\st)$ be a Legendrian link and $L_1,L_2\sse(\D^4,\la_\st)$ two fillings of $\Lambda$. Suppose that the two restricted systems of augmentations $$\varepsilon_{\overline{L}_1}:\SA_{\Lambda}\lr \Z[H_1(\overline{L}_1)],\quad \varepsilon_{\overline{L}_2}:\SA_{\Lambda}\lr \Z[H_1(\overline{L}_2)]$$
are not DGA homotopic. Then the exact Lagrangian surfaces $\overline{L}_1,\overline{L}_2\sse W(\La)$ are not Hamiltonian isotopic in the Weinstein 4-manifold $W(\La)$.
\end{prop}

\begin{proof} Let $\mathfrak{W}(W(\La))$ be the wrapped Fukaya category of the Weinstein 4-manifold $W(\La)$, $C$ the union of the $m$ co-cores of the Weinstein handles of $W(\La)$, and $\mbox{CW}(C)$ the endomorphism ring of $C$ as an object in $\mathfrak{W}(W(\La))$. By \cite[Theorem 1.1]{CRGG17}, $C$ generates $\mathfrak{W}(W(\La))$, see also \cite[Theorem 1.10]{GPS2}, and thus we consider the category $\mathfrak{W}(W(\La))$ through its Yoneda embedding $\mbox{Hom}(C,-):=\mbox{CW}(C,-)$. The Lagrangian surfaces $\overline{L}_1,\overline{L}_2$ are exact and hence represent objects in $\mathfrak{W}(W(\La))$, equally denoted $\overline{L}_1,\overline{L}_2$. Under the Yoneda embedding, these two objects become $\mbox{Hom}(C,\overline{L}_i):=\mbox{CW}(C,\overline{L}_i)$, $i=1,2$. We will now argue that $\overline{L}_1,\overline{L}_2\in\mbox{Ob}(\mathfrak{W}(W(\La)))$ are distinct objects, which proves that the exact Lagrangian surfaces $\overline{L}_1,\overline{L}_2\sse W(\La)$ are not Hamiltonian isotopic. It suffices to show that $\mbox{CW}(C,\overline{L}_1)$ and $\mbox{CW}(C,\overline{L}_2)$ are distinct as $\mbox{CW}(C)$-modules.
	
Let $L\sse(\D^4,\la_\st)$ be a filling of $\La$ and $\varepsilon_{L}:\SA_\La\lr\Z[H_1(L)]$ its associated augmentation. The holomorphic disks that define $\varepsilon_{L}$ are explained in detail in \cite{EHK}, see also \cite[Theorem 6.8]{ENsurvey} and Sections \ref{sec:cobordism} and \ref{sec:elementary} above. In short, a Reeb chord $a\in\SA_\La$ is sent to the contributions from rigid holomorphic disks $u:(\D^2,\dd\D^2)\lr(\D^4,\la_\st)$ with a positive puncture at the Reeb chord $a$, and each disk contribution is weighted by the homology class $[\dd u]\in \Z[H_1(L)]$, where $\dd u\sse L$ is appropriately capped in $L$. The claim is that the holomorphic disks that define the  $\mbox{CW}(C)$-module structure of  $\mbox{CW}(C,\overline{L})$, namely the composition $A_\infty$-map
$$\eta_{L}:\mbox{CW}(C)\otimes\mbox{CW}(C,\overline{L})\lr\mbox{CW}(C,\overline{L}),$$
or equivalently $\mbox{CW}(C)\lr\mbox{End(CW}(C,\overline{L})\mbox{)}$, are in bijection with those contributing to the restricted augmentation $\varepsilon_{L}$. Indeed, we first observe that $C\cap\overline{L}$, which generates $\mbox{CW}(C,\overline{L})$, consists of precisely a point per each component of $C$. The disks contributing to $\eta_{L}$ have: a positive puncture at a generator of $\mbox{CW}(C)$, which is either a minimum of a Morse function on $C$ or a Reeb chord of its Legendrian boundary $\dd C\sse\dd W(\La)$; a positive puncture at a generator of $\mbox{CW}(C,\overline{L})$; and a negative puncture at a generator of $\mbox{CW}(C,\overline{L})$ (in fact, the two generators of $\mbox{CW}(C,\overline{L})$ here must be the same). These disks are depicted in the right diagram of Figure \ref{fig:Surgery}. In our case, the contributions of these disks are weighted by their boundary homology classes, where we only keep track of the piece of the boundary that belongs to the closed Lagrangian surface $\overline{L}$. These contributions yield coefficients in the ground ring $\Z[H_1(\overline{L})]$.

\begin{center}
	\begin{figure}[h!]
		\centering
		\includegraphics[scale=0.9]{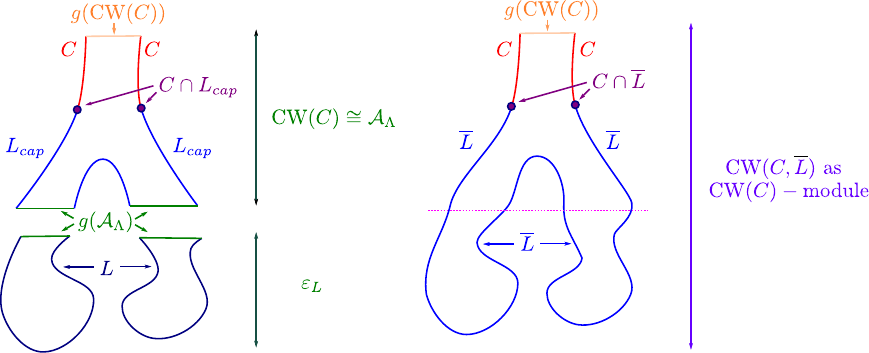}
		\caption{On the left, the holomorphic disks contributing to the augmentation $\varepsilon_{L}$ (bottom) and the surgery isomorphism $\mbox{CW(C)}\cong\SA_\La$ (top). On the right, the holomorphic disks contributing to the module structure $\mbox{CW}(C)\lr\mbox{End(CW}(C,\overline{L})\mbox{)}$. The notation $g(\mbox{CW}(C))$ and $g(\SA_\La)$ stands for generators of the algebras $\mbox{CW}(C)$ and $\SA_{\Lambda}$: $g(\mbox{CW}(C))$ are Reeb chords of $\dd C$ or the minimum of $C$, and $g(\SA_\La)$ are Reeb chords of $\La$.}
		\label{fig:Surgery}
	\end{figure}
\end{center}

Now, the decomposition $\overline{L}=L\cup_\La L_{cap}$ of the Lagrangian surface $\overline{L}\sse W(\La)$ into a Lagrangian filling $L$ and the cores $L_{cap}$ is compatible with neck-stretching along the contact hypersurface $(\dd\D^4,\xi_\st)$ containing $\La$, where the Weinstein handles are attached. That is, the Weinstein 4-manifold decomposes as $$W(\La)\cong (\D^4,\la_\st)\cup_{\Op(\La)}\left( \bigcup_{i=1}^l(T^*\D^2,\la_\st)\right),$$
and performing a neck-stretching procedure to the holomorphic disks contributing to $\eta_L$ breaks them into two pieces. See \cite[Chapter 3]{Abbas14}, \cite[Section 3]{BEHWZ03}, or \cite[Section 5]{CRGG20} for the neck-stretching technique along such a contact hypersurface, in this case a standard contact level set of the symplectization of $(\S^3,\xi_\st)$. For a sufficiently large stretching, see e.g. \cite[Corollary 3.10]{EHK} or \cite[Section 11.3]{BEHWZ03}, there is a one-to-one correspondence between the rigid holomorphic disks contributing to $\eta_L$ and two-level broken disks.

The first level consists of holomorphic disks in the moduli space $\mathcal{M}^{co}({\bf c})$, following the notation in \cite{EkholmLekili19}, where ${\bf c}:=c_\La z^v c z^w$, $c_\La$ is a product of Reeb chords in $\La$, $z^v,z^w$ are intersections in $C\cap L_{cap}$, and $c$ is a generator of $\mbox{CW}(C)$. The boundaries of these holomorphic disks start belonging in $L_{cap}$, at the left of the leftmost positive puncture in $c_\La$, then continue to belong to $L_{cap}$ as the Reeb chords in $c_\La$ are visited, and switch to belonging to $C$, when $z^v$ is reached; then, the boundary (away from the punctures) belongs to $C$ as the chords in $c$ are visited and we reach $z^w$, where the boundary switches back to $L_{cap}$. The curves in this first level are depicted in the top of the left diagram of \ref{fig:Surgery}, where $z^v,z^w$ are the two points marked by $C\cap L_{cap}$, and these moduli were studied in detail in \cite{EkholmLekili19} by using the properties proved in \cite{Ekholm19}. In particular, \cite[Theorem 2]{EkholmLekili19} shows that the $A_\infty$-map $\{\Phi_i\}_{i\in\N}:\mbox{CW}(C)^{\otimes i}\lr\SA_\La$ defined by counting rigid contributions of the moduli spaces $\mathcal{M}^{co}({\bf c})$ (for $i=1$; for $i>1$, we have multiple positive punctures at generators of $\mbox{CW}(C)$) is an $A_\infty$-quasi-isomorphism, see \cite[Theorem 72]{EkholmLekili19} for details.

The second level consists of holomorphic disks with a positive puncture at the Reeb chords of $\La$ and boundary in $L$. These are the same rigid holomorphic disks as those contributing to the augmentation map $\varepsilon_{L}:\thinspace \SA_\La\lr \Z[H_1(L)]$. 
However, we note that the weights are counted with coefficients in $\Z[H_1(\overline{L})]$; that is, the count of holomorphic disks contributing to the second level of $\eta_L$ is precisely given by the restricted augmentation $\varepsilon_{\overline{L}} :\thinspace \SA_\La \to \Z[H_1(\overline{L})]$.

In conclusion, the moduli space of disks contributing to the $\mbox{CW}(C)$-module structure $\eta_L$ splits into $\mathcal{M}^{co}({\bf c})$, which yields the $A_\infty$-quasi-isomorphism $\mbox{CW}(C)\cong\SA_\La$, and the moduli space of holomorphic disks contributing to the restricted augmentation $\varepsilon_{\overline{L}}$, associated to the closed Lagrangian $\overline{L}$. Thus, under the surgery isomorphism, the $\mbox{CW}(C)$-module structure $\eta_L$ is precisely given by the augmentation $\varepsilon_{\overline{L}}$ on $\SA_\La$. 

Now if $L_1,L_2$ are Hamiltonian isotopic fillings of $\Lambda$, then $\overline{L}_1,\overline{L}_2$ are (exact) Lagrangian isotopic in $W(\La)$, even relative to the co-cores. It follows that their associated restricted augmentations $\varepsilon_{\overline{L}_1}$ and $\varepsilon_{\overline{L}_2}$ are DGA homotopic, and the result follows.
\end{proof}

For the following two corollaries, we emphasize, and implicitly use, that the DGA $\SA_{\Lambda}$ associated to the Legendrian braids $\La=\La(\beta)\sse(\S^3,\xi_\st)$, $\beta\in\mbox{Br}^+_N$, are concentrated in nonnegative degree, and thus DGA homotopic (restricted systems of) augmentations are the same as equivalent (restricted systems of) augmentations.

\begin{proof}[Proof of Corollary \ref{cor:Stein1}] For $g=2$, we consider the Legendrian knot $\La=\La(\beta)\sse(\S^3,\xi_\st)$ given by the positive braid $\beta=(\sigma_2\sigma_1\sigma_3\sigma_2)^4\sigma_2\sigma_1\sigma_3$. By Proposition \ref{prop:knots}, $\La(\beta)$ admits infinitely many genus $2$ exact Lagrangian fillings $\{L_i\}_{i\in\N}$, distinguished by their augmentations $\varepsilon_{L_i}:\SA_{\La}\lr\Z[H_1(L_i)]$. Consider the Weinstein 4-manifold $W:=W(\La)$, which is homotopic to a 2-sphere $\S^2$ because $\La(\beta)$ is a knot. For the same reason, all Lagrangian fillings of $\La$ are restricted. 
Note that since $\La$ is a knot, the restricted augmentation $\varepsilon_{\overline{L}_i}$ is the same as $\varepsilon_{L_i}$ for all $i$.
By Proposition \ref{prop:Surgery}, it follows that the exact Lagrangian surfaces $\{\overline{L}_i\}_{i\in\N}$ in $W$ are not Hamiltonian isotopic. This proves the assertion in the case of $g=2$. 

For higher $g\geq2$, it suffices to apply the same argument to the Legendrian knots associated to the braids $\beta_{g}=(\sigma_2\sigma_1\sigma_3\sigma_2)^4\sigma_2\sigma_1\sigma_3\sigma_1^{2(g-2)}$. Since there exists an exact Lagrangian cobordism from $\La(\beta)=\La(\beta_2)$ to $\La(\beta_g)$ for all $g\geq2$, each knot $\La(\beta_g)$ admits infinitely many exact Lagrangian fillings of genus $g$. Hence Proposition \ref{prop:Surgery} implies that the Weinstein 4-manifold $W_g:=W(\La(\beta_g))$, homotopic to a 2-sphere $\S^2$, also admits infinitely many exact Lagrangian surfaces of genus $g$ which are not Hamiltonian isotopic. In each case, $W_g$ does not admit any embedded exact Lagrangian surface of genus $h\leq g-1$ since its intersection form is given by the $1\times1$ matrix $\left( \begin{matrix} tb(\La(\beta_g))-1 \end{matrix} \right)$. This concludes the proof.
\end{proof}

\begin{proof}[Proof of Corollary \ref{cor:Stein2}] Consider the Legendrian link $\La=\Lambda(\beta_{11})$ and the Weinstein 4-manifold $W(\Lambda(\beta_{11}))$, which is homotopic to $\S^2\vee\S^2$ because $\Lambda(\beta_{11})$ has two components. Theorem \ref{thm:main} implies that this 2-component link admits infinitely many distinct exact Lagrangian fillings. In order to apply Proposition \ref{prop:Surgery}, we need to ensure that these infinitely many fillings are distinguished by their restricted systems of augmentations. For that, let us study the augmentation $\varepsilon_{L}$ associated to the (initial) Lagrangian filling $L$ in Subsection \ref{ssec:VariationsAffineD4} and its $\vartheta$-loop iterates. There are four homology variables $t_1,t_2,t_3,t_4$; under $\varepsilon_L$ , these are augmented to
\[
t_1\to \frac{s_9 s_{12} s_{13}}{s_{11}},\quad t_2\to -s_{11} s_{16},\quad t_3\to -\frac{1}{s_{10}s_{16}},\quad 
t_4\to \frac{s_{10}}{s_9 s_{12} s_{13}}.
\]
Note that the $\vartheta$-loop monodromy fixes each homology variable, and so the $\vartheta$-loop iterates $\varepsilon_L \circ \vartheta^k$ have the same effect as $\varepsilon_L$ on $t_1,t_2,t_3,t_4$ for all $k\in\N$.

The first two variables $t_1,t_2$ lie in one component of $\La$ and $t_3,t_4$ lie in the other component. From the discussion following Definition~\ref{def:restricted}, we can impose the additional conditions $(t_1,t_2,t_3,t_4) = (+1,-1,+1,-1)$ to obtain the restricted system of augmentations $\varepsilon_{\overline{L}}$ (enhanced by link automorphisms); this is because we first set $t_1=t_3=1$ to reduce to a single base point on each component, and then set $t_2=t_4=-1$ to pass from $H_1(L)$ to $H_1(\overline{L})$. In terms of the $s$ variables, there are $3$ new conditions (the $4$th is redundant):
\[
s_9 s_{12} s_{13} = s_{11}, \quad s_{11}s_{16} = 1, \quad s_{10}s_{16} = -1.
\]

Now we note that $(s_9,s_{10},s_{11},s_{12},s_{13},s_{16}) = (1,1,-1,-1,1,-1)$ in particular satisfy these conditions. These values of the $s_i$ also produce the maximal value of $|(\varepsilon_L \circ \vartheta^k)(a_9)|$ for all $k\in\N$, from the computation in Section~\ref{ssec:VariationsAffineD4}. It follows that the same argument that we used there, to show that the $\vartheta$-orbit of the system of augmentations $\varepsilon_L$ is entire, also shows that the same is true of the restricted system of augmentations $\varepsilon_{\overline{L}}$. We can now apply Proposition \ref{prop:Surgery} to conclude that $W(\Lambda(\beta_{11}))$ admits infinitely many distinct exact Lagrangian tori, up to Hamiltonian isotopy.
\end{proof}

\appendix\section{The Cobordism Map for an Elementary Saddle Cobordism}
\label{sec:saddle-map}

The goal of this section is to prove Proposition~\ref{prop:EHK}, the formula for the cobordism map over $\Z$ for a saddle cobordism at a proper contractible Reeb chord. The proof will come in several steps. In Section~\ref{ssec:minidip}, we will first add what we call ``mini-dips'' on either side of the Reeb chord, which then propagate through the cobordism; this changes the cobordism by a Hamiltonian isotopy. The advantage of adding these mini-dips is that they localize the disks that contribute to the cobordism map, so that the map mod $2$ is quite simple and can be written down very explicitly. The main technical result is lifting this map to $\Z$ and showing that Proposition~\ref{prop:EHK} holds for the cobordism with mini-dips; this is the content of Proposition~\ref{prop:signs} below. The proof of Proposition~\ref{prop:signs} is somewhat indirect and involves making the cobordism even more complicated, with the trade-off benefit being that the cobordism map becomes easier to handle. This is in the spirit of a well-known technique in Legendrian knot theory called ``dipping'', and occupies Sections~\ref{ssec:splash} and~\ref{ssec:sign-proof}. Finally, in Section~\ref{ssec:EHK-proof}, we deduce Proposition~\ref{prop:EHK} from its mini-dipped special case, Proposition~\ref{prop:signs}, by tracing the effect of mini-dips on the cobordism map.

\subsection{Formula for the saddle cobordism map}
\label{ssec:minidip}

\begin{center}
	\begin{figure}[h!]
		\centering
		\includegraphics[scale=1.3]{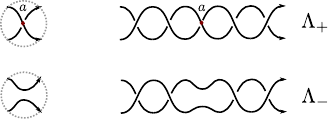}
		\caption{A saddle cobordism (left) and the mini-dipped version of this cobordism (right).}
		\label{fig:standard-saddle}
	\end{figure}
\end{center}

As in Section~\ref{ssec:EHK-map}, we consider a saddle cobordism $L_a$ between Legendrian links $\Lambda_-$ and $\Lambda_+$, where $\Lambda_-$ is obtained from $\Lambda_+$ by replacing a contractible Reeb chord $a$ of $\Lambda_+$ by the oriented resolution of the crossing. To simplify the cobordism map, we perturb $\Lambda_\pm$ by a Legendrian isotopy (and consequently the saddle cobordism by a Hamiltonian isotopy) as follows: use two Reidemeister II moves to push the understrand of $a$ over the overstrand on either side of $a$, as shown in Figure~\ref{fig:standard-saddle}. We call these moves ``mini-dips'' of $a$.\footnote{These are independently introduced in \cite{GSW}, where they are called ``double dipping'' and are used for the same purpose of simplifying cobordism maps.} Note that the crossing $a$ is situated differently in Figure~\ref{fig:standard-saddle} than the similar-looking Figure 18 from \cite{EHK}, and consequently our mini-dip is different from the dip considered there. Also note that the crossing data for the mini-dips (with the understrand of $a$ passing over the overstrand in the minidips) is forced by the condition that we want the resulting diagrams to represent Lagrangian projections of Legendrian links---apply Stokes' Theorem to a bigon whose two corners are the contractible chord $a$ and an adjacent crossing in either of the mini-dips.

For the next few subsections, we will assume that $\Lambda_+$ and $\Lambda_-$ contain the mini-dips shown in Figure~\ref{fig:standard-saddle}; we will return to the general case without mini-dips in Section~\ref{ssec:EHK-proof}. 
Over $\Z_2$, the mini-dips force the cobordism map $\Phi_{L_a} :\thinspace \SA_{\Lambda_+} \to \SA_{\Lambda_-}$ to have the following simple form:
\begin{align*}
\Phi_{L_a}(a) &= s \\
\Phi_{L_a}(a_1) &= a_1 + s^{-1} \\
\Phi_{L_a}(a_2) &= a_2 + s^{-1} \\
\Phi_{L_a}(a_i) &= a_i
\end{align*}
where the final equation holds for all Reeb chords $a_i$ of $\Lambda_+$ besides $a,a_1,a_2$. This follows directly from the work of \cite{EHK} (cf.\ Section~\ref{ssec:EHK-map}) because there are only two disks with a positive puncture at $a$, the bigon between $a_1$ and $a$ and the bigon between $a$ and $a_2$.

Over the course of Sections~\ref{ssec:splash} and~\ref{ssec:sign-proof}, we will prove that the cobordism map $\Phi_{L_a}$ with signs is given as follows.

\begin{center}
	\begin{figure}[h!]
		\centering
		\includegraphics[scale=1.3]{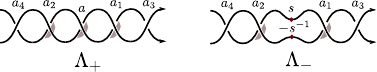}
		\caption{A saddle cobordism with mini-dips. Crossings and base points are labeled, and quadrants with negative orientation sign are shaded.}
		\label{fig:decorated-saddle}
	\end{figure}
\end{center}

\begin{prop}
Suppose that $\Lambda_+$ and $\Lambda_-$ are related by a saddle cobordism as in Figure~\ref{fig:decorated-saddle}: the Lagrangian projection of $\Lambda_+$ has a contractible crossing $a$ flanked by mini-dips with the crossings on either side of $a$ labeled by $a_1$ and $a_2$, and $\Lambda_-$ is the result of resolving the crossing $a$ and placing base points labeled $s$ and $-s^{-1}$ on either side of the resolved crossing. Then, the cobordism map $\Phi_{L_a} :\thinspace \SA_{\Lambda_+} \to \SA_{\Lambda_-}$ over $\Z$ is given, up to a link automorphism of $\Lambda_-$, by:
\begin{align*}
\Phi(a) &= s \\
\Phi(a_1) &= a_1 - s^{-1} \\
\Phi(a_2) &= a_2 - s^{-1} \\
\Phi(a_i) &= a_i
\end{align*}
where the final equation holds for all Reeb chords $a_i$ of $\Lambda_+$ besides $a,a_1,a_2$.
More precisely, there is a link automorphism $\Omega :\thinspace \SA_{\Lambda_-} \to \SA_{\Lambda_-}$ such that the $\Phi_{L_a} :\SA_{\Lambda_+} \to \SA_{\Lambda_-}$ is chain homotopy equivalent to $\Omega \circ \Phi$ with $\Phi$ as defined above.
\label{prop:signs}
\end{prop}

\begin{remark}
We believe that the auxiliary data needed to define signs (capping operators, etc.) can be chosen so that the combinatorial formula for $\Phi$ in Proposition~\ref{prop:signs} is precisely the geometric map $\Phi_{L_a}$, without composing with a link automorphism of $\Lambda_-$. However, we will not need the stronger statement for our purposes.\hfill$\Box$
\end{remark}

\subsection{Splashes and diagonal automorphisms}
\label{ssec:splash}

Our strategy for proving Proposition~\ref{prop:signs} is as follows: the signs for the formula for $\Phi_{L_a}$ given there are essentially forced, up to a link automorphism of $\Lambda_-$, by the algebraic requirement that $\Phi_{L_a}$ needs to be a chain map over $\Z$. This forcing is not true in full generality, but we will see that it is true if we isotop $\Lambda_{\pm}$ via Reidemeister II moves so that their differentials consist of many terms, each of which is easy to handle. This sort of strategy is familiar in the subject through the technique of dipping; see, e.g., \cite{FR,Sabloff05,Sivek_Bordered}. We will present a variant of this technique in this subsection, and then return to the proof of Proposition~\ref{prop:signs} in Section~\ref{ssec:sign-proof} below.

Let $\Lambda$ be a Legendrian link. By applying planar isotopy and Reidemeister II moves, we can isotop $\Lambda$ so that its Lagrangian projection $\Pi_{xy}(\Lambda)$ satisfies the following properties: 
\begin{itemize}
\item
all vertical tangencies (parallel to the $y$ axis) lie on two lines $x=c_0$ and $x=c_1$, and there are at least $2$ vertical tangencies on each of these lines;
\item 
no crossings in either the Lagrangian or front projections occur at the same $x$ coordinate.
\end{itemize}
Note in particular that $\Pi_{xy}(\Lambda)$ is the plat closure of some braid between $x=c_0$ and $x=c_1$ where the braid strands go from left to right. Furthermore,
in the front projection $\Pi_{xz}(\Lambda)$, all left cusps lie on the line $x=c_0$ and all right cusps lie on $x=c_1$.

Subdivide the interval $[c_0,c_1]$ by choosing $x_0<x_1<\cdots<x_p$ with $x_0=c_0$, $x_p=c_1$ such that:
\begin{itemize}
\item
there are no crossings in either $\Pi_{xy}(\Lambda)$ or $\Pi_{xz}(\Lambda)$ in the intervals $x\in [x_0,x_1]$ and $[x_{p-1},x_p]$;
\item
for $i=1,\ldots,p-2$, in the interval $x\in [x_i,x_{i+1}]$ there is exactly one crossing in either $\Pi_{xy}(\Lambda)$ or $\Pi_{xz}(\Lambda)$, and no crossing in the other.
\end{itemize}

\begin{center}
	\begin{figure}[h!]
		\centering
		\includegraphics[scale=1.3]{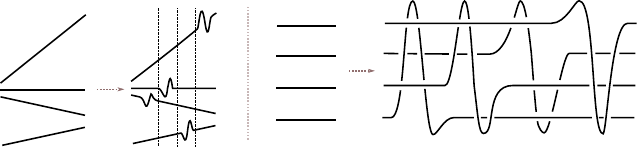}
		\caption{A set of splashes, in the front projection (left) and corresponding Lagrangian projection (right).}
		\label{fig:splash}
	\end{figure}
\end{center}

Now in a neighborhood of the $x=x_i$ slices for $i=1,\ldots,p-1$, introduce a collection of ``splashes''\footnote{This terminology is inspired by \cite{FR}, though our splashes are slightly different from theirs and more resemble what \cite{EHK} call ``dips''.} as shown in Figure~\ref{fig:splash}. This is a $C^0$-small perturbation in the front projection, while in the Lagrangian projection, each strand is pushed through the other strands. For definiteness, we order the collection of splashes at $x=x_i$ from left to right in increasing order of the $y$-coordinate of the splashed strand; in the Lagrangian projection, the crossing information for the new crossings is determined by the relative $z$ coordinates of the strands at $x=x_i$.
Let $\Lambda'$ denote the resulting Legendrian link, and note that $\Pi_{xy}(\Lambda')$ is obtained from $\Pi_{xy}(\Lambda)$ by a (large) number of Reidemeister II moves. See Figure~\ref{fig:splash-ex} for  a sample illustration of $\Lambda'$.

\begin{center}
	\begin{figure}[h!]
		\centering
		\includegraphics[scale=0.9]{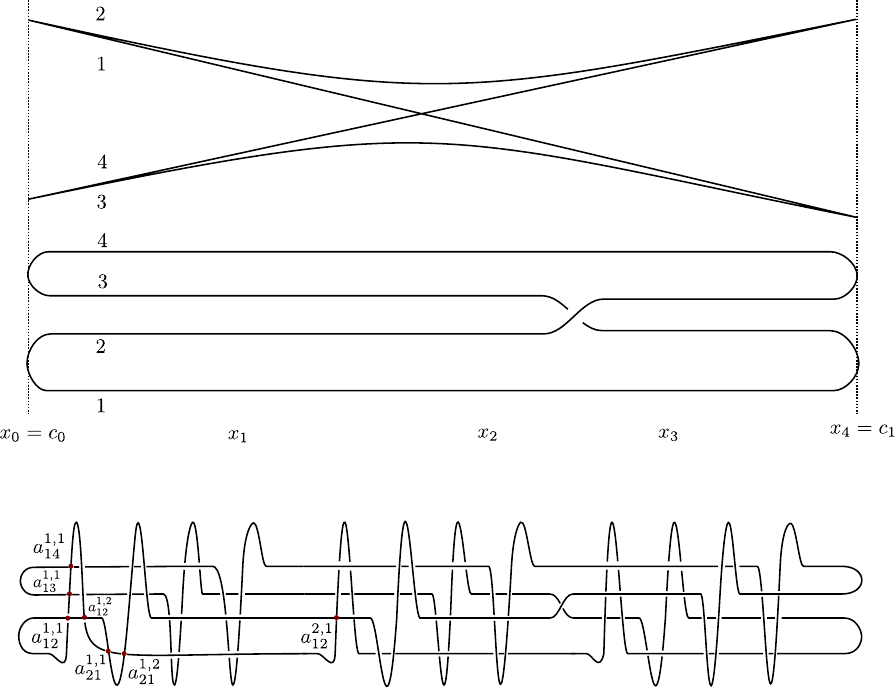}
		\caption{A full example of splashing. Top to bottom: the front and Lagrangian projections of $\Lambda$, with strands joining $x=c_0$ and $x=c_1$ labeled $1,2,3,4$, and the Lagrangian projection of $\Lambda'$ with some crossings labeled.}
		\label{fig:splash-ex}
	\end{figure}
\end{center}

Write the Chekanov--Eliashberg DGA of $\Lambda'$ as $(\SA_{\Lambda'},\partial)$. Say that an automorphism $\Psi$ of the algebra $\SA_{\Lambda'}$ is \textit{diagonal} if it is of the following form: if $a_i$ denote the Reeb chords of $\Lambda'$, then there is a collection of (invertible) scalars $\lambda_i$ such that $\Psi(a_i) = \lambda_i a_i$ for all $i$.

\begin{prop}
Suppose that $\Psi$ is a diagonal automorphism of $\SA_{\Lambda'}$ that is also a chain map: $\Psi \circ \partial = \partial \circ \Psi$. Then $\Psi$ is a link automorphism of $\Lambda'$.
\label{prop:splash}
\end{prop}

In order to prove Proposition~\ref{prop:splash}, we need some more notation. Let $s$ denote the number of vertical tangencies in the Lagrangian projection of $\Lambda$ at each of $x=c_0$ and $x=c_1$, so that the Lagrangian projection is the plat closure of a $2s$-stranded braid. Number these strands $1,2,\ldots,2s$ so that in $[x_0,x_1]$, the strands are numbered in increasing order of $y$ coordinate; keep the numbering of braid strands consistent throughout the braid, and that in general the strands will not remain numbered in increasing order beyond $x=x_1$.

Now suppose that $\Psi$ satisfies the hypotheses of Proposition~\ref{prop:splash}. We will construct units $u_1,\ldots,u_{2s}$ such that the following condition holds for all Reeb chords $a$ of $\Lambda'$:
\begin{equation}
\Psi(a) = u_{r(a)} u_{c(a)}^{-1} a. \tag{$\ast$}
\label{eq:splash}
\end{equation}
Here we use $r(a)$ and $c(a)$ to denote the labels of the strands that are the endpoint and beginning point of $a$, respectively.

The following lemma is a useful tool for propagating condition \eqref{eq:splash}. Say that an embedded bigon with boundary on $\Pi_{xy}(\Lambda')$ and two convex corners at Reeb chords of $\Lambda'$ is a \textit{standard bigon} if one corner is $+$ and one is $-$; similarly say that an embedded triangle is a \textit{standard triangle} if one corner is $+$ and the other two are $-$.

\begin{lemma}
If $a_1,a_2$ are Reeb chords of $\Lambda'$ such that there is a (unique) standard bigon with corners at $a_1,a_2$, then \eqref{eq:splash} holds for $a_1$ if and only if it holds for $a_2$. If $a_1,a_2,a_3$ are Reeb chords such that there is a (unique) standard triangle with corners at $a_1,a_2,a_3$, then if \eqref{eq:splash} holds for two of $a_1,a_2,a_3$, then it holds for the third as well.
\label{lem:corners}
\end{lemma}

\begin{proof}
A bigon with $+$ corner at $a_1$ and $-$ corner at $a_2$ contributes a term $a_2$ to $\partial(a_1)$; since $r(a_1)=r(a_2)$ and $c(a_1)=c(a_2)$ and $\Psi\partial = \partial\Psi$, it follows that if  \eqref{eq:splash} holds for one of $a_1,a_2$, then it holds for the other. Similarly, a triangle with $+$ corner at $a_1$ and $-$ corners at $a_2,a_3$ contributes a term $a_2a_3$ to $\partial(a_1)$; now use the fact that $r(a_1)=r(a_2)$, $c(a_2)=r(a_3)$, and $c(a_1)=c(a_3)$ to conclude the desired result. See Figure~\ref{fig:bigons-triangles}.
\end{proof}

\begin{center}
	\begin{figure}[h!]
		\centering
		\includegraphics[scale=1.3]{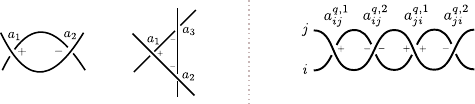}
		\caption{Left two diagrams: a standard bigon and a standard triangle. Right diagram: a chain of bigons joining $a_{ij}^{q,1},a_{ij}^{q,2},a_{ji}^{q,1},a_{ji}^{q,2}$. (If strand $i$ is instead above strand $j$, then these are still standard bigons but all $+$ and $-$ labels are interchanged.)}
		\label{fig:bigons-triangles}
	\end{figure}
\end{center}

We next label the crossings of $\Lambda'$ in the splashes as follows. Consider the splashed portion of strand $i$ at $x=x_q$. For any $j\neq i$, this splash crosses strand $j$ twice; label these two crossings $a_{ij}^{q,1}$ (left) and $a_{ij}^{q,2}$ (right). In this way we label all splashed crossings in $\Lambda'$ as $a_{ij}^{q,l}$ for $1\leq i,j\leq 2s$ ($i\neq j$), $1\leq q\leq p-1$, and $1\leq l\leq 2$. See Figure~\ref{fig:splash-ex} for an example.

\begin{lemma}
For fixed $i,j,q$, if one of the crossings $a_{ij}^{q,1},a_{ij}^{q,2},a_{ji}^{q,1},a_{ji}^{q,2}$ satisfies \eqref{eq:splash}, then so do the other three.
\label{lem:four}
\end{lemma}

\begin{proof}
In a neighborhood of $x=x_q$, strand $i$ lies either completely above or completely below strand $j$ in the $z$ coordinate. It follows that there is a chain of three standard bigons linking $a_{ij}^{q,1},a_{ij}^{q,2},a_{ji}^{q,1},a_{ji}^{q,2}$; see Figure~\ref{fig:bigons-triangles}. The result follows from Lemma~\ref{lem:corners}.
\end{proof}

\begin{lemma}
If $\Psi$ satisfies the hypotheses of Proposition~\ref{prop:splash}, then there are $u_1,\ldots,u_{2s}$ such that \eqref{eq:splash} holds for all Reeb chords of $\Lambda'$.
\label{lem:splash}
\end{lemma}

\begin{proof}
We will prove that \eqref{eq:splash} holds for all $a=a_{ij}^{q,l}$ by induction on $q$. In the course of the proof, we will also show that \eqref{eq:splash} holds for all other Reeb chords of $\Lambda'$, which correspond precisely to the Reeb chords of $\Lambda$.

We first establish the induction base case $q=1$. Set $u_1=1$. Then for $j=2,\ldots,2s$, the Reeb chords $a_{1j}^{1,1}$ have one endpoint on strand $1$ and one endpoint on strand $j$; since each $\Psi(a_{1j})^{1,1}$ is an invertible scalar multiple of $a_{1j}^{1,1}$, it follows that there are unique choices of $u_2,\ldots,u_{2s}$ so that \eqref{eq:splash} holds for $a=a_{1j}^{1,1}$ for all $j=2,\ldots,2s$. Thus by Lemma~\ref{lem:four}, \eqref{eq:splash} also holds for $a_{1j}^{1,l}$ and $a_{j1}^{1,l}$, $j=2,\ldots,2s$, $l=1,2$. Next suppose $j>i\geq 2$. Consider the two triangles shown in Figure~\ref{fig:triangles}. Of the two corners at $a_{ij}^{1,2}$, one must be $+$ and one must be $-$, and similarly for the two corners at $a_{i1}^{1,1}$. Of the corner at $a_{1j}^{1,2}$ and the corner at $a_{j1}^{1,1}$, again one must be $+$ and one must be $-$ since the union of the two triangles is a standard bigon. Since no triangle can have three $-$ corners by Stokes' Theorem, it follows that one of the two triangles in Figure~\ref{fig:triangles} must be standard. Thus by Lemma~\ref{lem:corners}, \eqref{eq:splash} holds for $a_{ij}^{1,2}$, whence it holds for $a_{ij}^{1,l}$ and $a_{ji}^{1,l}$ by Lemma~\ref{lem:four}. This completes the base case $q=1$.

\begin{center}
	\begin{figure}[h!]
		\centering
				\includegraphics[scale=1.3]{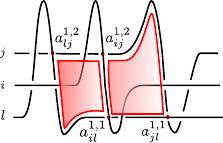}
		\caption{Showing that $a_{ij}^{1,2}$ satisfies \eqref{eq:splash}.}
		\label{fig:triangles}
	\end{figure}
\end{center}

Now suppose that \eqref{eq:splash} holds for $a=a_{ij}^{q,l}$ for fixed $q$ and all $i,j,l$; we need to show that it also holds for $a=a_{ij}^{q+1,l}$ for all $i,j,l$. There are two cases depending on whether the crossing between $x=x_q$ and $x=x_{q+1}$ is in $\Pi_{xy}(\Lambda)$ or in $\Pi_{xz}(\Lambda)$. First suppose that the crossing is in $\Pi_{xy}(\Lambda)$, and let $k_1,k_2$ denote the labels of the strands involved in the crossing. Choose any two indices $i\neq j$ and assume without loss of generality that $a_{ji}^{q,2}$ is to the right of $a_{ij}^{q,2}$. 
As long as $\{i,j\} \neq \{k_1,k_2\}$, there is a standard bigon joining $a_{ji}^{q,2}$ to $a_{ij}^{q+1,1}$, and it follows from Lemmas~\ref{lem:corners} and~\ref{lem:four} and the induction hypothesis that $a_{ij}^{q+1,1},a_{ij}^{q+1,2},a_{ji}^{q+1,1},a_{ji}^{q+1,2}$ satisfy \eqref{eq:splash}. If on the other hand $\{i,j\}=\{k_1,k_2\}$, then if we label the crossing between $x=x_q$ and $x=x_{q+1}$ by $a$, there are standard bigons joining $a_{ji}^{q,2}$ to $a$ and $a$ to $a_{ij}^{q+1,1}$, and it follows as before that $a_{ij}^{q+1,1},a_{ij}^{q+1,2},a_{ji}^{q+1,1},a_{ji}^{q+1,2}$, along with $a$ itself, all satisfy \eqref{eq:splash}.

It remains to treat the case where the crossing between $x=x_q$ and $x=x_{q+1}$ is in $\Pi_{xz}(\Lambda)$. Say that this crossing is between strands $k_1$ and $k_2$, where we choose the labels so that strand $k_2$ has larger $y$ coordinate than strand $k_1$ between $x=x_q$ and $x=x_{q+1}$. The only difference between the splashes at $x=x_q$ and $x=x_{q+1}$ is that strand $k_1$ lies above $k_2$ at $x_q$ while $k_2$ lies above $k_1$ at $x_{q+1}$, or vice versa. It follows that for any two indices $i\neq j$, as long as $\{i,j\} \neq \{k_1,k_2\}$, there is a standard bigon joining $a_{ji}^{q,2}$ to $a_{ij}^{q+1,1}$ (or $a_{ij}^{q,2}$ to $a_{ji}^{q+1,1}$) as in the previous case, and we conclude as before that $a_{ij}^{q+1,1},a_{ij}^{q+1,2},a_{ji}^{q+1,1},a_{ji}^{q+1,2}$ satisfy \eqref{eq:splash}.

Finally suppose $\{i,j\} = \{k_1,k_2\}$. We will show that $a_{k_1k_2}^{q+1,1}$ satisfies \eqref{eq:splash}, whence by Lemma~\ref{lem:four} all four crossings of the form $a_{ij}^{q+1,l}$ for $\{i,j\}=\{k_1,k_2\}$ and $l=1,2$ satisfy \eqref{eq:splash}, and the induction step will be complete. 
Since $\Lambda$ has at least $4$ strands joining left and right, there is some other strand labeled $k_3$ with $k_3 \neq k_1,k_2$. There are three cases depending on the position of the $y$ coordinate of strand $k_3$ relative to strands $k_1$ and $k_2$ in $[x_q,x_{q+1}]$. 

If $k_3$ lies above both $k_1$ and $k_2$ in the $y$ direction, then consider the two triangles shown in Figure~\ref{fig:triangles2}. For both of these triangles, one corner is at $a_{k_1k_2}^{q+1,1}$ and the other two corners satisfy \eqref{eq:splash}. Since these triangles split in two a standard bigon with corners at $a_{k_3k_2}^{q,2}$ and $a_{k_2k_3}^{q+1,1}$, as in the $q=1$ case one of the triangles must be standard. It follows from Lemma~\ref{lem:corners} that \eqref{eq:splash} holds for $a_{k_1k_2}^{q+1,1}$, as desired. If $k_3$ lies between $k_1$ and $k_2$, or $k_3$ lies below both $k_1$ and $k_2$, entirely similar arguments using the triangles shown in Figure~\ref{fig:triangles3} again show that $a_{k_1k_2}^{q+1,1}$ satisfies \eqref{eq:splash}, and we are done.
\end{proof}

\begin{center}
	\begin{figure}[h!]
		\centering
				\includegraphics[scale=1.3]{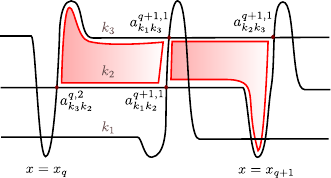}
		\caption{Showing that $a_{k_1k_2}^{q+1,1}$ satisfies \eqref{eq:splash}.}
		\label{fig:triangles2}
	\end{figure}
\end{center}

\begin{center}
	\begin{figure}[h!]
		\centering
		\includegraphics[scale=1.3]{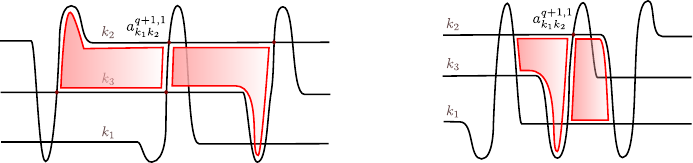}
		\caption{Two more cases to show that $a_{k_1k_2}^{q+1,1}$ satisfies \eqref{eq:splash}.}
		\label{fig:triangles3}
	\end{figure}
\end{center}

We can now finally prove Proposition~\ref{prop:splash}.

\begin{proof}[Proof of Proposition~\ref{prop:splash}]
Suppose $\Psi :\thinspace \SA_{\Lambda'} \to \SA_{\Lambda'}$ is a chain map and an isomorphism. By Lemma~\ref{lem:splash}, we have $u_1,\ldots,u_{2s}$ so that \eqref{eq:splash} holds for all Reeb chords of $\Lambda'$. Now the strands $1,\ldots,2s$ are joined in pairs at the left end of $\Lambda'$, and joined in pairs again at the right end. On the left end, for $k=1,\ldots,s$, strands $2k-1$ and $2k$ are connected, and this yields an embedded disk with a single corner at $a_{2k-1,2k}^{1,1}$, which must be a $+$ corner by Stokes. This contributes a constant ($1$) term to $\d(a_{2k-1,2k}^{1,1})$. Since $\Psi$ is a chain map and $\Psi(a_{2k-1,2k}^{1,1}) = u_{2k}u_{2k-1}^{-1} a_{2k-1,2k}^{1,1}$ by \eqref{eq:splash}, it follows that $u_{2k-1}=u_{2k}$. 

More generally, the same argument shows that if strands $i$ and $j$ are joined at either end of $\Lambda'$, then $u_i = u_j$. It follows that $u_i=u_j$ whenever $i$ and $j$ are part of the same connected component of $\Lambda'$. Thus we may remove duplicates and rename $u_1,\ldots,u_{2s}$ as $u_1,\ldots,u_m$, where $m$ is the number of components of $\Lambda'$. Then \eqref{eq:splash} becomes precisely the condition for $\Psi$ to be a link automorphism of $\Lambda'$, and we are done.
\end{proof}

\subsection{Proof of Proposition~\ref{prop:signs}}
\label{ssec:sign-proof}

With the auxiliary result Proposition~\ref{prop:splash} in hand, we next prove Proposition~\ref{prop:signs}. Suppose that $\Lambda_+$ and $\Lambda_-$ are related by a saddle cobordism at a contractible crossing flanked by mini-dips, as in the statement of Proposition~\ref{prop:splash} or the right hand side of Figure~\ref{fig:standard-saddle}.
We first show that the desired map $\Phi$ is indeed a chain map, and then proceed to the main proof.

\begin{lemma}
The map $\Phi :\thinspace \SA_{\Lambda_+} \to \SA_{\Lambda_-}$ defined in Proposition~\ref{prop:signs} is a chain map: $\Phi \circ \partial_+ = \partial_- \circ \Phi$.
\label{lem:chainmap}
\end{lemma}

\begin{proof}
We show that $\Phi \circ \partial_+$ and $\partial_- \circ \Phi$ agree on all Reeb chords of $\Lambda_+$. Note that $\partial_+(a) = 0$, so $\Phi(\partial_+(a)) = 0 = \partial_-(s) = \partial_-(\Phi(a))$. Also if we denote the mini-dip crossing next to $a_1$ by $a_3$, then $\partial_+(a_1) = \partial_-(a_1) = -a_3$, so $\Phi(\partial_+(a_1)) = -\Phi(a_3) = -a_3 = \partial_-(a_1-s^{-1}) = \partial_-(\Phi(a_1))$; similarly $\Phi(\partial_+(a_2)) = \partial_-(\Phi(a_2))$.

\begin{center}
	\begin{figure}[h!]
		\centering
				\includegraphics[scale=1.5]{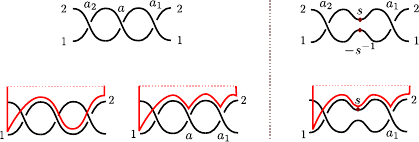}
		\caption{Labeling the strands of $\Lambda_+$ (left) and $\Lambda_-$ (right) in the cobordism region, and disks that pass through the cobordism region and contribute to $\dd_+(a_i)$ and $\dd_-(a_i)$.}
		\label{fig:chain}
	\end{figure}
\end{center}

Now suppose that $a_i$ is a Reeb chord of $\Lambda_+$ besides $a,a_1,a_2$: we need to show that $\Phi(\dd_+(a_i)) = \dd_-(a_i)$. The disks that make up the differentials $\dd_+(a_i)$ and $\dd_-(a_i)$ are exactly the same except where they pass through the cobordism region encompassing $a_1,a,a_2$. Where these disks pass through the cobordism region, there is also a precise correspondence between the disks for $\dd_+(a_i)$ and $\dd_-(a_i)$. The oriented boundary of such a disk enters the region on one of the strands on the left and exits on one of the strands on the right, or it enters on the right and exits on the left. If we label the strands as shown in Figure~\ref{fig:chain}, then for instance any disk contributing to $\dd_-(a_i)$ that enters on the left on strand 1 and exits on the right on strand 2 must pass $s$ and turn a corner at $a_1$; there are two corresponding disks contributing to $\dd_+(a_i)$ with the same enter and exit data, one of which turns no corners in the cobordism region and one of which turns corners at $a$ and $a_1$. See Figure~\ref{fig:chain}; the result replaces a monomial $sa_1$ in $\dd_-(a_i)$ by $1+aa_1$ in $\dd_+(a_i)$. In all, there are $8$ ways to pass through the cobordism region, with resulting contributions to $\dd_\pm(a_i)$ as follows:
\[
\begin{array}{|c|c|c||c|c|c|}
\hline
& \dd_+(a_i) & \dd_-(a_i) && \dd_+(a_i) & \dd_-(a_i) \\
\hline \hline
1\rightarrow 1 & a & s & 1 \leftarrow 1 & -a_1-a_2-a_1aa_2 & s^{-1}-a_1sa_2 \\ \hline
1 \rightarrow 2 & 1+aa_1 & sa_1 & 1\leftarrow 2 & 1+aa_2 & sa_2 \\ \hline
2 \rightarrow 1 & 1+a_2a & a_2s & 2\leftarrow 1 & 1+a_1a & a_1s \\ \hline
2 \rightarrow 2 & a_1+a_2+a_2aa_1 & -s^{-1}+a_2sa_1 & 2 \leftarrow 2 & -a & -s \\ \hline
\end{array}
\]
Now an inspection of this table shows that each entry in the $\dd_-(a_i)$ column is obtained from the corresponding entry in the $\dd_+(a_i)$ column by replacing $a,a_1,a_2$ by $s,a_1-s^{-1},a_2-s^{-1}$ respectively. It immediately follows that $\Phi(\dd_+(a_i)) = \dd_-(a_i)$.
\end{proof}

We now have a chain map $\Phi :\thinspace \SA_{\Lambda_+} \to \SA_{\Lambda_-}$. In order to prove Proposition~\ref{prop:signs}, we want to show that this is equal to the geometric cobordism map $\Phi_{L_a}$ up to a link automorphism. To do this, we will first localize the differentials of $\Lambda_\pm$ by introducing splashes in the spirit of Section~\ref{ssec:splash}. In what follows, we continue to refer to the small region of $\Lambda_\pm$ containing $a_1$ and $a_2$ (and $a$ for $\Lambda_+$) as the ``cobordism region'', outside of which $\Lambda_+$ and $\Lambda_-$ coincide. We now change $\Lambda_-$ by a sequence of Reidemeister II moves that avoid the cobordism region, first pulling all vertical tangencies of $\Pi_{xy}(\Lambda_-)$ left or right so that they line up vertically, then adding splashes to separate any crossings in $\Pi_{xy}(\Lambda_-)$ or $\Pi_{xz}(\Lambda_-)$ outside the cobordism region. From this we obtain a link $\Lambda_-'$, Legendrian isotopic to $\Lambda_-$, for which there are $x_0<x_1<\cdots<x_p$ such that:
\begin{itemize}
\item
all vertical tangencies lie on $x=x_0$ or $x=x_p$, and the number of vertical tangencies on each of these lines is at least $2$;
\item
there is a collection of splashes in a neighborhood of $x=x_i$ for $i=1,\ldots,p-1$;
\item
there is one $i\in\{1,\ldots,p-2\}$ such that $[x_i,x_{i+1}]$ contains the cobordism region, and in that interval $[x_i,x_{i+1}]$ the only crossings in either $\Pi_{xy}(\Lambda_-)$ or $\Pi_{xz}(\Lambda_-)$ are between the two strands involved in the cobordism region;
\item
for every other $i=1,\ldots,p-2$, in the interval $[x_i,x_{i+1}]$ there is exactly one crossing in either $\Pi_{xy}(\Lambda_-)$ or $\Pi_{xz}(\Lambda_-)$, and no crossing in the other;
\item
$[x_0,x_1]$ and $[x_{p-1},x_p]$ contain no crossings in $\Pi_{xy}(\Lambda_-)$ or $\Pi_{xz}(\Lambda_-)$.
\end{itemize}
In short, we follow the prescription from Section~\ref{ssec:splash}, except that we do not separate the crossings in the cobordism region from each other. 

If we follow the same sequence of Reidemeister II moves going from $\Lambda_-$ to $\Lambda_-'$, but start with $\Lambda_+$, then we obtain a Legendrian link $\Lambda_+'$ that differs from $\Lambda_-'$ only in the cobordism region. We summarize the picture as follows:
\[
\xymatrix{
\Lambda_+ \ar[r] \ar[d] & \Lambda_+' \ar[d] \\
\Lambda_- \ar[r] & \Lambda_-',
}
\]
where the horizontal arrows are Legendrian isotopies given by (the same) Reidemeister II moves, and the vertical arrows are (identical) elementary saddle cobordisms. Note that the saddle cobordism between $\Lambda_+$ and $\Lambda_-$ is Hamiltonian isotopic to the concatenation of the three cobordisms specified by the other three sides of the square: from top to bottom, the isotopy from $\Lambda_+$ to $\Lambda_+'$, followed by the saddle cobordism between $\Lambda_+'$ and $\Lambda_-'$, followed by the isotopy from $\Lambda_-'$ to $\Lambda_-$. By \cite{EHK,Karlsson-cob}, the cobordism map $\Phi_{L_a} :\thinspace \SA_{\Lambda_+} \to \SA_{\Lambda_-}$ is chain homotopy equivalent to the composition of the cobordism maps given by the three cobordisms. We will show that this composition is the map $\Phi$ from the statement of Proposition~\ref{prop:signs}.

We first consider the cobordism map $\Phi' : \SA_{\Lambda_+'} \to \SA_{\Lambda_-'}$. By Lemma~\ref{lem:chainmap}, we know of another chain map $\Phi_1 :\thinspace \SA_{\Lambda_+'} \to \SA_{\Lambda_-'}$: this is defined by $\Phi_1(a) = s$, $\Phi_1(a_1) = a_1-s^{-1}$, $\Phi_1(a_2) = a_2-s^{-1}$, and $\Phi_1(a_i) = a_i$ for all other Reeb chords $a_i$. Since \cite{EHK} gives a formula for geometric cobordism maps mod $2$ and this formula is especially simple in our case, we know that the geometric map $\Phi'$ agrees with $\Phi_1$ up to signs. By replacing $s$ by $-s$ if necessary, we can assume that $\Phi'(a) = s$.

\begin{lemma}
There is a link automorphism $\Omega :\thinspace \SA_{\Lambda_-'} \to \SA_{\Lambda_-'}$ such that $\Phi' = \Omega \circ \Phi_1$.
\label{lem:linkaut}
\end{lemma}

\begin{proof}
Write $\dd_+'$ and $\dd_-'$ for the differentials on $\SA_{\Lambda_+'}$ and $\SA_{\Lambda_-'}$ respectively.

Since the terms in $\Phi'$ agree with the terms in $\Phi_1$ up to sign, there are signs $\sigma_i \in \{\pm 1\}$ such that $\Phi'(a_1) = \sigma_1a_1 \pm s^{-1}$, $\Phi'(a_2) = \sigma_2a_2 \pm s^{-1}$, and $\Phi'(a_i) = \sigma_i a_i$ for all other $i$. In fact, because $\Phi'$ is a chain map, we must more specifically have $\Phi'(a_1) = \sigma_1a_1-s^{-1}$ and $\Phi'(a_2) = \sigma_2a_2-s^{-1}$. To see this for $a_1$ (with a similar argument for $a_2$), we use the fact that $\Lambda_\pm'$ have more than $2$ strands joining left and right in the $x$ direction, as stipulated in their construction. In particular, there is a strand of $\Lambda_\pm'$ that lies either above or below the cobordism region in the $xy$ projection. Assume this strand lies above (the argument for below is very similar). The splashes from this strand on either side of the cobordism region intersect the strands from the cobordism region in a number of crossings, two of which are labeled $a_3$ and $a_4$ in Figure~\ref{fig:cob-splash}. In $\Pi_{xy}(\Lambda_+')$, there is a standard bigon with corners at $a_3$ and $a_4$, contributing either $a_4$ to $\dd_+'(a_3)$ or $a_3$ to $\dd_+'(a_4)$. For definiteness assume the former (the argument is same for the latter). An inspection of Figure~\ref{fig:cob-splash} shows that $\dd_+'(a_3)$ contains the terms $\pm(1+a_1a)a_4$ while $\dd_-'(a_3)$ contains $\pm a_1sa_4$, and furthermore that these are the only terms in $\dd_\pm'(a_3)$ that involve $a_4$. Since $\dd_-\Phi'(a_3) = \Phi'\dd_+(a_3)$, we must have $\pm a_1sa_4 = \Phi'((1+a_1a)a_4) = \pm (1+(\sigma_1a_1\pm s^{-1})s)a_4$, which implies that the $\pm$ sign is $-$ as claimed.

\begin{center}
	\begin{figure}[h!]
		\centering
								\includegraphics[scale=1.2]{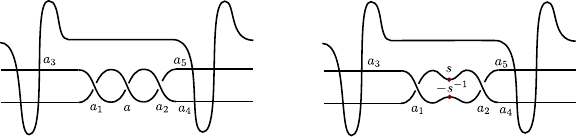}
		\caption{Splashes on either side of the cobordism region in $\Lambda_+'$ (left) and $\Lambda_-'$ (right), with relevant crossings labeled.}
		\label{fig:cob-splash}
	\end{figure}
\end{center}

Now let $\Omega$ to be the algebra automorphism of $\SA_{\Lambda_-'}$ defined by $\Omega(a_i) = \sigma_ia_i$ for all $i$; then by our expression for $\Phi'$, we have $\Phi' = \Omega \circ \Phi_1$. 
It follows from the fact that $\Phi'$ and $\Phi_1$ are both chain maps that $\Omega$ is also a chain map. Indeed, for any $i$ we have
\[
\dd_-'\Omega a_i  = \dd_-'\Phi' a_i  = \Phi'\dd_+' a_i = \Omega\Phi_1\dd_+'a_i = \Omega\dd_-'\Phi_1a_i = \Omega\dd_-' a_i,
\]
where when $i=1,2$ the first and last equality follow from the fact that $\dd_-'(s^{-1})=0$.

It remains to show that $\Omega$ is a link automorphism of $\Lambda_-'$. To do this, we use the fact that $\Omega$ is a diagonal automorphism of $\SA_{\Lambda_-'}$ and a chain map, and appeal to a variant of Proposition~\ref{prop:splash}. We cannot use Proposition~\ref{prop:splash} directly because $\Lambda_-'$ does not have a splash between $a_1$ and $a_2$. However, we can still follow the inductive proof of Proposition~\ref{prop:splash} in this setting. The only thing we need to check is the inductive step where we are given that $a_1$ satisfies the condition \eqref{eq:splash} from the proof and need to conclude that $a_2$ also satisfies this condition. To do this, let $a_3$ and $a_5$ be the crossings depicted in Figure~\ref{fig:cob-splash}, and note that there is a standard bigon in $\Pi_{xy}(\Lambda_-')$ with corners at $a_3$ and $a_5$. If the positive corner of this bigon is at $a_3$, then $\dd_-'(a_3)$ contains the terms $(a_1sa_2-s^{-1})a_5$, while if the positive corner is at $a_5$, then $\dd_-'(a_5)$ contains the terms $a_3(a_1sa_2-s^{-1})$. In either case, since $\Omega$ is a chain map, $\Omega(a_1)s\Omega(a_2)-s^{-1}$ must be equal to $\pm(a_1sa_2-s^{-1})$. Since $a_1$ satisfies\eqref{eq:splash}, $\Omega(a_1) = u_{r(a_1)}u_{c(a_1)}^{-1}a_1$; but this implies that
$\Omega(a_2) = (u_{r(a_1)}u_{c(a_1)}^{-1})^{-1}a_2 = u_{r(a_2)} u_{c(a_2)}^{-1} a_2$ and so $a_2$ satisfies \eqref{eq:splash}, as desired. This completes the proof of Lemma~\ref{lem:linkaut}.
\end{proof}

We next examine the maps given by the Legendrian isotopies between $\Lambda_+$ and $\Lambda_+'$, and between $\Lambda_-$ and $\Lambda_-'$.
Suppose that $\Lambda_-'$ is obtained from $\Lambda_-$ by $N$ Reidemeister II moves. Then we can follow \cite{Chekanov,ENS} to construct a DGA isomorphism $\Psi_-$ between $\SA_{\Lambda_-'}$ and the DGA $S^N(\SA_{\Lambda_-})$ given by stabilizing $\SA_{\Lambda_-}$ $N$ times (adding $2N$ generators in the process). This isomorphism comes from $N$ applications of the isomorphism coming from a single Reidemeister II move, as already described in Section~\ref{ssec:isotopy}. By that construction, if we start with $\Lambda_-$ and add the Reidemeister II moves one by one, we see that the nontrivial parts of $\Psi_-$ come from disks with two positive punctures, one of which is at a crossing in the Reidemeister II move. By inspection, there is no point at which there is such a disk where the other positive puncture is at either $a_1$ or $a_2$, and it follows that $\Psi_-(a_1) = a_1$ and $\Psi_-(a_2) = a_2$.

Similarly, since $\Lambda_+'$ is obtained from $\Lambda_+$ by the same Reidemeister II moves, we have a DGA isomorphism $\Psi_+$ between $\SA_{\Lambda_+'}$ and $S^N(\SA_{\Lambda_+})$, and $\Psi_+(a_1) = a_1$, $\Psi_+(a) = a$, $\Psi_+(a_2) = a_2$. Indeed, we can say more about the relation between $\Psi_+$ and $\Psi_-$. The key point is that there is a precise correspondence between the twice-positive-punctured disks that determine $\Psi_+$ and the twice-positive-punctured disks that determine $\Psi_-$: algebraically, one obtains the latter from the former by replacing $a,a_1,a_2$ by $s,a_1-s^{-1},a_2-s^{-1}$ just as in the proof of Lemma~\ref{lem:chainmap}. Consequently, for any Reeb chord $a_i$ of $\Lambda_+'$ (and thus of $\Lambda_-'$) besides $a,a_1,a_2$, $\Psi_-(a_i)$ is obtained from $\Psi_+(a_i)$ by this algebraic replacement.

Put another way, let $\Phi_1$ be as above, and similarly define $\Phi_2 :\thinspace S^N(\SA_{\Lambda_+}) \to S^N(\SA_{\Lambda_-})$ by $\Phi_2(a) = s$, $\Phi_2(a_1)=a_1-s^{-1}$, $\Phi_2(a_2) = a_2-s^{-1}$, and $\Phi_2$ is the identity on all other generators of $S^N(\SA_{\Lambda_+})$. Note that by Lemma~\ref{lem:chainmap}, $\Phi_1$ and $\Phi_2$ are both chain maps. By the above discussion, we conclude that the following diagram commutes:
\[
\xymatrix{
S^N(\SA_{\Lambda_+}) \ar[d]^{\Phi_2} && \SA_{\Lambda_+'} \ar[ll]_<<<<<<<<<{\Psi_+}^<<<<<<<<<\cong \ar[d]^{\Phi_1} \\
S^N(\SA_{\Lambda_-}) && \SA_{\Lambda_-'} \ar[ll]_<<<<<<<<<{\Psi_-}^<<<<<<<<<\cong.
}
\]
The cobordism map $\SA_{\Lambda_+} \to \SA_{\Lambda_+'}$ is simply the composition of the inclusion map $i :\thinspace \SA_{\Lambda_+} \to S^N(\SA_{\Lambda_+})$ and the inverse of $\Psi_+$, and the cobordism map $\SA_{\Lambda_-'} \to \SA_{\Lambda_-}$ is the composition of $\Psi_-$ and the projection map $p :\thinspace S^N(\SA_{\Lambda_-}) \to \SA_{\Lambda_-}$. 

We can now finally turn to the geometric cobordism map $\Phi_{L_a} :\thinspace \SA_{\Lambda_+} \to \SA_{\Lambda_-}$. To complete the proof of Proposition~\ref{prop:signs}, we want to show that $\Phi_{L_a} = \Omega \circ \Phi$ for some link automorphism $\Omega$ of $\Lambda_-$. 

At this point we have broken down $\Phi_{L_a}$ into a composition of three cobordism maps: $\Psi_+^{-1} \circ i :\thinspace \SA_{\Lambda_+} \to \SA_{\Lambda_+'}$, $\Phi' :\thinspace \SA_{\Lambda_+'} \to \SA_{\Lambda_-'}$, and $p \circ \Psi_- :\thinspace \SA_{\Lambda_-'} \to \SA_{\Lambda_-}$. That is, $\Phi_{L_a}$ is chain homotopy equivalent to the composition $p\circ\Psi_-\circ\Phi'\circ\Psi_+^{-1}\circ i$ of the five maps going around the sides of the following rectangle:
\[
\xymatrix{
\SA_{\Lambda_+} \ar[rr]^>>>>>>>>>>i \ar[d]^{\Phi_{L_a}} && S^N(\SA_{\Lambda_+}) \ar@{-->}[d]^{\Phi_2} && \SA_{\Lambda_+'} \ar[ll]_<<<<<<<<{\Psi_+}^<<<<<<<<\cong \ar[d]^{\Phi'} \\
\SA_{\Lambda_-} && S^N(\SA_{\Lambda_-}) \ar[ll]^<<<<<<<<<<p && \SA_{\Lambda_-'} \ar[ll]_<<<<<<<<{\Psi_-}^<<<<<<<<\cong.
}
\]
From Lemma~\ref{lem:linkaut}, there is a link automorphism $\Omega$ of $\Lambda_-'$ such that $\Phi' = \Omega \circ \Phi_1$. Since $\Lambda_-'$ and $\Lambda_-$ are Legendrian isotopic, $\Omega$ induces a link automorphism of $\Lambda_-$, which we also call $\Omega$, so that $\Omega$ commutes with the chain map $p \circ \Psi_- :\thinspace \SA_{\Lambda_-'} \to \SA_{\Lambda_-}$ induced by the isotopy. Thus
\[
\Phi_{L_a} \simeq p\circ\Psi_-\circ\Phi'\circ\Psi_+^{-1}\circ i = p\circ\Psi_-\circ\Omega\circ\Phi_1\circ\Psi_+^{-1}\circ i = \Omega \circ p \circ\Psi_-\circ \Phi_1\circ\Psi_+^{-1}\circ i
= \Omega \circ p \circ\Phi_2\circ i.
\]
But $p \circ\Phi_2\circ i$ is exactly equal to $\Phi$ as defined in the statement of Proposition~\ref{prop:signs}, and we are done with the proof.

\subsection{Proof of Proposition~\ref{prop:EHK}}
\label{ssec:EHK-proof}

The remainder of this section is devoted to the proof of Proposition~\ref{prop:EHK}. At this point, by Proposition~\ref{prop:signs}, we know the saddle cobordism map for a saddle flanked by mini-dips; to prove Proposition~\ref{prop:EHK}, we just need to compose this map with maps corresponding to the Reidemeister II moves of adding and removing mini-dips. This is similar to the proof of Proposition~\ref{prop:signs} in the previous subsection, except that it will now be important to calculate these Reidemeister II maps in more detail.

Suppose that, as in the statement of Proposition~\ref{prop:EHK}, we have a saddle cobordism between $\Lambda_+$ and $\Lambda_-$, where the cobordism is given by resolving a proper contractible Reeb chord $a$ of $\Lambda_+$. Let $\Lambda_+'$ be the result of adding a mini-dip to $\Lambda_+$ just after $a$ following the orientation of $\Lambda_+$, and let $\Lambda_+''$ be result of further adding a mini-dip to $\Lambda_+'$ on the other side of $a$. Similarly define $\Lambda_-'$ and $\Lambda_-''$. Then $\Lambda_\pm'$ are obtained from $\Lambda_\pm$ by a single Reidemeister II move, $\Lambda_\pm''$ are obtained from $\Lambda_\pm'$ by another Reidemeister II move, and $\Lambda_+''$ and $\Lambda_-''$ are related by a saddle move of the precise form that we considered in Proposition~\ref{prop:signs}. See Figure~\ref{fig:minidips}.

\begin{center}
	\begin{figure}[h!]
		\centering
				\includegraphics[scale=1.2]{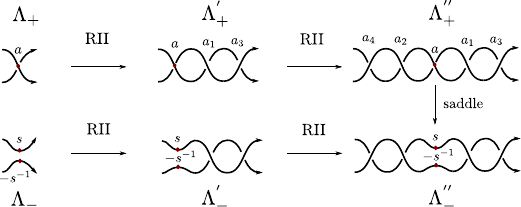}
		\caption{Adding mini-dips to $\Lambda_\pm$.}
		\label{fig:minidips}
	\end{figure}
\end{center}

The properness condition for $a$ translates into the following result.

\begin{lemma}
Given that the Reeb chord $a$ is proper:
\begin{itemize}
\item
if $a_q$ is any Reeb chord of $\Lambda_+'$, and $\dd'$ denotes the differential on $\Lambda_+'$, then any term in $\dd'(a_q)$ that contains $a_3$ must contain $a_3$ exactly once and cannot contain $a_1$;
\item
if $a_q$ is any Reeb chord of $\Lambda_+''$, and $\dd''$ denotes the differential on $\Lambda_+''$, then any term in $\dd'(a_q)$ that contains $a_4$ must contain $a_4$ exactly once and cannot contain any of $a_1,a_2,a_3$.
\end{itemize}
\label{lem:proper}
\end{lemma}

\begin{center}
	\begin{figure}[h!]
		\centering
				\includegraphics[scale=1.2]{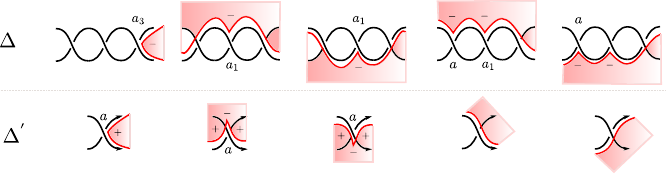}
		\caption{Turning an immersed disk $\Delta'$ for $\Lambda_+'$ into an immersed disk $\Delta$ for $\Lambda_+$.}
		\label{fig:Delta}
	\end{figure}
\end{center}

\begin{proof}
We will establish the statement for $\Lambda_+'$; the proof of the statement for $\Lambda_+''$ is similar. If $a_q$ is any of $a,a_1,a_3$, then the statement is trivially true: by action considerations, the only term in $\dd'(a_q)$ that could contain $a_3$ is just the term $a_3$ itself in $\dd'(a_1)$. Now assume $a_q$ is not $a,a_1,a_3$. Consider any word in $\dd'(a_q)$, corresponding to an immersed disk $\Delta'$ in $\Lambda_+'$ with sole $+$ corner at $a_q$ and a $-$ corner at $a_3$. Then $\Delta'$ in turn produces an immersed disk $\Delta$ in $\Lambda_+$, now possibly with concave corners at $a$: see Figure~\ref{fig:Delta}. If $\Delta'$ contained multiple corners at $a_3$, or corners at both $a_1$ and $a_3$, then the boundary of $\Delta$ would pass through $\Pi_{xy}(a)$ more than once, violating the properness condition from Definition~\ref{def:proper}.
\end{proof}

We will now piece together the five maps $\SA_{\Lambda_+} \to \SA_{\Lambda_+'} \to \SA_{\Lambda_+''} \to \SA_{\Lambda_-''} \to \SA_{\Lambda_-'} \to \SA_{\Lambda_-}$ to get the desired cobordism map. The central map $\SA_{\Lambda_+''} \to \SA_{\Lambda_-''}$ has already been computed, while the remaining maps come from Reidemeister II isotopies.

We will focus for now on the map $\SA_{\Lambda_+} \to \SA_{\Lambda_+'}$, which we call $\Psi_+^\rightarrow$. This is the chain map induced by adding a Legendrian Reidemeister II move, as derived in \cite{Chekanov,ENS} and summarized in Section~\ref{ssec:isotopy} above, and we describe it explicitly now. Label the Reeb chords of $\Lambda_+$ besides $a$ as $a_5,\ldots,a_r$, so that we can write $\SA_{\Lambda_+} = \SA(a,a_5,\ldots,a_r)$ and $\SA_{\Lambda_+'} = \SA(a,a_1,a_3,a_5,\ldots,a_r)$. We stabilize $\SA_{\Lambda_+}$ by adding two new generators $e_1,e_2$ with $|e_1|=0$, $|e_2|=-1$, $\dd(e_1) = e_2$, $\dd(e_2)=0$, to produce a new DGA $S(\SA_{\Lambda_+}) = \SA(a,a_5,\ldots,a_r,e_1,e_2)$. As described in Section~\ref{ssec:isotopy} and specifically defined in \eqref{eq:RII}, there is a chain isomorphism $\Psi :\thinspace \SA_{\Lambda_+'} \to S(\SA_{\Lambda_+})$, which in our case is defined by $\Psi(a_1) = e_1$, $\Psi(a_3) = -e_2$, $\Psi(a) = a$, and for $\ell \geq 5$,
\[
\Psi(a_\ell) = a_\ell - H\Psi\dd'a_\ell
\]
where $\dd'$ is the differential on $\Lambda_+'$. Then $\Psi_+^\rightarrow$ is defined to be equal to $\Psi^{-1} \circ i$.

We now claim that $\Psi_+^\rightarrow$ satisfies the following formula, which can be compared to the definition of $\Phi^\rightarrow$ from Section~\ref{ssec:EHK-map}.

\begin{lemma}
For all $\ell\geq 5$, we have
\begin{equation}
\Psi_+^\rightarrow(a_\ell) = a_\ell - \sum_{\Delta\in\Delta_a^\rightarrow(a_\ell)} (-1)^{|w_1(\Delta)|}\sgn(\Delta) \Psi_+^\rightarrow(w_1(\Delta)) a_1 w_2(\Delta).
\label{eq:Psi-def}
\end{equation}
\label{lem:Psi-def}
\end{lemma}

\begin{proof}
Assume without loss of generality that $a,a_5,\ldots,a_r$ are ordered by height (note that $a$ is contractible and thus has the shortest height). We first claim that for all $q\geq 5$, $\Psi(a_q)-a_q$ only includes terms that involve at least one $e_1$ and no $e_2$: we abbreviate this condition by $\Psi(a_q)-a_q = O(e_1)$. We prove this by induction on $q$, where the base case is actually $a_q=a$ (note $\Psi(a)-a=0$). For the induction step, note that $\Psi(a_q)-a_q = -H\Psi\dd'a_q$, and the right hand side only contains terms involving at least one $e_1$; we need to show that $H\Psi\dd'a_q$ does not involve $e_2$. 

Consider any word $w$ in $\dd'a_q$. If $a_3$ does not appear in $w$, then $w$ involves only $a,a_1,a_5,\ldots,a_{q-1}$, and so by induction $\Psi(w)-w=O(e_1)$ and $H\Psi(w) = H(w)=0$. On the other hand, if $w$ does involve $a_3$, then by Lemma~\ref{lem:proper}, $w = w_1a_3w_2$ where $w_1,w_2$ involve only $a,a_5,\ldots,a_{q-1}$; then by induction again, $H\Psi(w) = 
- H((\Psi(w_1))e_2(\Psi(w_2))) = -H(w_1e_2\Psi(w_2)) = \pm w_1e_1\Psi(w_2)$ does not involve $e_2$. This completes the proof that $\Psi(a_q)-a_q = O(e_1)$ for all $q \geq 5$.

We now prove the lemma, again by induction on $\ell$. The base case is actually $\Psi_+^\rightarrow(a) = a$, which is \eqref{eq:Psi-def} with $a=a_\ell$. For the induction step, we compute that:
\[
\Psi_+^\rightarrow(a_\ell) = \Psi^{-1}(a_\ell) = a_\ell + \Psi^{-1}H\Psi \dd'a_\ell.
\]
Now suppose that $w$ is a word in $\dd'a_\ell$, and again apply Lemma~\ref{lem:proper}. If $w$ does not contain $a_3$, then $H\Psi(w) = 0$. If $w$ does contain $a_3$, then we write $w = w_1a_3w_2$ and compute:
\[
H\Psi(w) = H\Psi(w_1a_3w_2) = -H(\Psi(w_1)e_2\Psi(w_2)) = -H(w_1e_2\Psi(w_2)) = (-1)^{|w_1|} w_1e_1\Psi(w_2)
\]
and thus
\[
\Psi^{-1}H\Psi(w) = (-1)^{|w_1|} \Psi^{-1}(w_1e_1\Psi(w_2)) = (-1)^{|w_1|}\Psi^{-1}(w_1) a_1 w_2 = (-1)^{|w_1|} \Psi_+^\rightarrow(w_1) a_1 w_2,
\]
where we have used the fact that $w_1$ does not involve $a_1$ or $a_3$ and thus $\Psi^{-1}(w_1) = \Psi^{-1} i(w_1) = \Psi_+^\rightarrow(w_1)$. Finally note that the disk for $w$ in $\Lambda_+'$ precisely corresponds to a disk $\Delta$ in $\Delta_a^\rightarrow(a_\ell)$ in $\Lambda_+$, and that the sign for $w$ in $\dd' a_\ell$ is $-\sgn(\Delta)$ since $\Delta$ replaces a corner at $a_3$ with positive orientation sign with a corner at $a$ with negative orientation sign. Now the signed sum of $\Psi^{-1}H\Psi(w)=\Psi_+^\rightarrow(w_1) a_1 w_2$ over all disks in $\Delta_a^\rightarrow(a_\ell)$ gives \eqref{eq:Psi-def}, and this completes the induction.
\end{proof}

In a similar way, we write $\Psi_+^\leftarrow$ for the cobordism map from 
$ \SA_{\Lambda_+'} = \SA(a,a_1,a_3,a_5,\ldots,a_r)$ to $\SA_{\Lambda_+''} = \SA(a,a_1,a_2,a_3,a_4,a_5,\ldots,a_r)$ induced by the Reidemeister II isotopy between $\Lambda_+'$ and $\Lambda_+''$.

\begin{lemma}
For all $\ell \geq 5$, 
\label{lem:Psi-def-left}
we have
\[
\Psi_+^\leftarrow(a_\ell) = a_\ell - \sum_{\Delta\in\Delta_a^\leftarrow(a_\ell)} (-1)^{|w_1(\Delta)|}\sgn(\Delta) \Psi_+^\rightarrow(w_1(\Delta)) a_2 w_2(\Delta).
\]
\end{lemma}

\begin{proof}
This is essentially identical to the proof of Lemma~\ref{lem:Psi-def}. Given our choice of orientation signs, there are two sign differences here from the proof of Lemma~\ref{lem:Psi-def}: $\Psi(a_4)$ is now $e_2$ rather than $-e_2$, and the sign of a word contributing to $\dd'(a_\ell)$ is now equal to $+\sgn(\Delta)$ rather than $-\sgn(\Delta)$ for the corresponding disk $\Delta \in \Delta_a^\leftarrow(a_\ell)$. These two sign changes cancel out. One other subtle difference is that if we follow the proof of the previous lemma, then $\Delta_a^\leftarrow(a_\ell)$ in the statement of the present lemma should be for $\Lambda_+'$ rather than for $\Lambda_+$. However, by the properness condition for $a$, there is a one-to-one correspondence between disks in $\Delta_a^\leftarrow(a_\ell)$ for $\Lambda_+'$ and $\Lambda_+$, and so the desired formula holds for either form of $\Delta_a^\leftarrow(a_\ell)$.
\end{proof}

We can now finally piece together our various subsidiary results to prove Proposition~\ref{prop:EHK}. To distinguish between the saddle cobordisms in the dipped and undipped settings, let $L_a$ be the cobordism between $\Lambda_+$ and $\Lambda_-$ as in the statement of Proposition~\ref{prop:EHK}, and let $\widetilde{L}_a$ be the cobordism between $\Lambda_+''$ and $\Lambda_-''$. As shown in Figure~\ref{fig:minidips}, we can concatenate $\widetilde{L}_a$ and four Lagrangians coming from Legendrian isotopies to create a five-story cobordism between $\Lambda_+$ and $\Lambda_-$ which is Hamiltonian isotopic to $L_a$: from top to bottom, the five cobordisms go between $\Lambda_+$, $\Lambda_+'$, $\Lambda_+''$, $\Lambda_-''$, $\Lambda_-'$, and $\Lambda_-$. 

The chain map $\Phi_{L_a} :\thinspace \SA_{\Lambda_+} \to \SA_{\Lambda_-}$ is then chain homotopic to the composition of the chain maps coming from the five cobordisms. We summarize this in the following diagram, which commutes up to chain homotopy:
\[
\xymatrix{
\SA_{\Lambda_+} \ar[rr]^{\Psi_+^\rightarrow} \ar[d]^{\Phi_{L_a}} && \SA_{\Lambda_+'} \ar[rr]^{\Psi_+^\leftarrow} && \SA_{\Lambda_+''}  \ar[d]^{\Phi_{\widetilde{L}_a}} \\
\SA_{\Lambda_-} && \SA_{\Lambda_-'} \ar[ll]_{p_2} && \SA_{\Lambda_-''} \ar[ll]_{p_1}.
}
\]
Here $\Psi_+^\rightarrow$ and $\Psi_+^\leftarrow$ are the maps computed in Lemmas~\ref{lem:Psi-def} and~\ref{lem:Psi-def-left}, while $p_1$ and $p_2$ are the maps induced by the reverse Reidemeister II moves from $\Lambda_-''$ to $\Lambda_-'$ and from $\Lambda_-'$ to $\Lambda_-$. By Remark~\ref{rmk:RII}, these last two maps (which correspond to $p\circ\Psi$ in Remark~\ref{rmk:RII}) are given simply by projection: $p_1(a_2)=p_1(a_4)=p_2(a_1)=p_2(a_3)=0$ and $p_1,p_2$ are the identity on all other generators. 

Now by Proposition~\ref{prop:signs}, $\Phi_{\widetilde{L}_a} = \Omega \circ \Phi$ where $\Omega$ is a link automorphism of $\Lambda_-''$ and $\Phi$ is the map given in the statement of the proposition. Since $\Lambda_-''$ and $\Lambda_-$ are isotopic, $\Omega$ induces a link automorphism of $\Lambda_-$ which we also denote by $\Omega$, and $p_2\circ p_1\circ\Omega = \Omega\circ p_2\circ p_1$. At this point we have:
\[
\Phi_{L_a} \simeq p_2 \circ p_1 \circ \Phi_{\widetilde{L}_a} \circ \Psi_+^\leftarrow \circ \Psi_+^\rightarrow = 
\Omega \circ p_2 \circ p_1 \circ \Phi \circ \Psi_+^\leftarrow \circ \Psi_+^\rightarrow.
\]

We will be done if we can show that the composition $p_2 \circ p_1 \circ \Phi \circ \Psi_+^\leftarrow \circ \Psi_+^\rightarrow$ is equal to the map $\Phi_{L_a}^\comb = \Phi^\leftarrow \circ \Phi^\rightarrow \circ \Phi_0$ from Proposition~\ref{prop:EHK}. But $\Phi_{L_a}^\comb$ is specifically designed so that this is the case. Specifically, if $a_\ell$ is any Reeb chord of $\Lambda_+$ besides $a$, then $\Phi^\rightarrow(a_\ell)$ and $\Phi^\leftarrow(a_\ell)$ are precisely the result of replacing $a_1$ and $a_2$ by $-s^{-1}$ in the expressions for $\Psi_+^\rightarrow(a_\ell)$ and $\Psi_+^\leftarrow(a_\ell)$ from Lemmas~\ref{lem:Psi-def} and~\ref{lem:Psi-def-left}. But by the definition of $\Phi$, this replacement is exactly the effect of composing with the map $p_2 \circ p_1 \circ \Phi$, which sends $a_1,a_2$ to $-s^{-1}$ and sends $a_\ell$ to itself for $\ell\geq 5$. It follows that
\[
\Phi_{L_a}^\comb(a_\ell) = (\Phi^\leftarrow \circ \Phi^\rightarrow)(a_\ell) = (p_2 \circ p_1 \circ \Phi)((\Psi_+^\leftarrow \circ \Psi_+^\rightarrow)(a_\ell)
\]
for all $\ell$. Combined with the fact that $\Phi_{L_a}^\comb(a) = s = (p_2 \circ p_1 \circ \Phi \circ \Psi_+^\leftarrow \circ \Psi_+^\rightarrow)(a)$, this establishes that
$\Phi_{L_a}^\comb = p_2 \circ p_1 \circ \Phi \circ \Psi_+^\leftarrow \circ \Psi_+^\rightarrow$. The proof of Proposition~\ref{prop:EHK} is complete.


\bibliographystyle{alpha}
\bibliography{AugFillCN-v2}

\end{document}